\documentclass[a4paper,11pt]{article}
\usepackage{amssymb,amsmath,amsthm,mathrsfs,graphicx,xcolor,hyperref}

\newtheorem{lemma}{Lemma}[section]
\newtheorem{theorem}{Theorem}[section]
\newtheorem{corollary}{Corollary}[section]
\newtheorem{definition}{Definition}[section]
\newtheorem{remark}{Remark}[section]

\newtheorem{hypothesis}{Hypothesis}[section]
\numberwithin{equation}{section}

\newcommand{\op}[1]{\operatorname{\text{\rm #1}}}
\renewcommand\div{\op{div}}
\newcommand{\footremember}[2]{%
    \footnote{#2}
    \newcounter{#1}
    \setcounter{#1}{\value{footnote}}%
}

\title{Existence and regularity results for the penalized thin obstacle problem with variable coefficients}
\author{%
Donatella Danielli\footremember{DD}{Department of Mathematics, Purdue University. Email: danielli@math.purdue.edu}%
\and Brian Krummel\footremember{BK}{Department of Mathematics, Purdue University. Email: bkrummel@purdue.edu}%
}
\date{}

\begin{document}

\maketitle

\begin{abstract}
In this paper we give a comprehensive treatment of a two-penalty boundary obstacle problem for a divergence form elliptic operator, motivated by applications to fluid dynamics and thermics. Specifically, we prove existence, uniqueness and optimal regularity of solutions, and establish structural properties of the free boundary. The proofs are based on tailor-made monotonicity formulas of Almgren, Weiss, and Monneau-type, combined with the classical theory of oblique derivative problems.
\end{abstract}

\section{Introduction}

The focus of this paper is  the study of regularity of solutions and the structure of the free boundary in a penalized boundary obstacle problem of interest in thermics, fluid mechanics, and electricity. To set the stage, we let $\Omega \subset \mathbb{R}^n$ be a bounded Lipschitz domain, and $\Gamma$ be a relatively open subset of $\partial \Omega$.   We will denote by $A=[a^{ij}]_{i,j=1\dots,n}$, with $a^{ij} \in L^{\infty}(\Omega)$,  a symmetric matrix satisfying the uniform ellipticity condition
\begin{equation} \label{ellipticity hyp}
	\lambda \,|\xi|^2 \leq \sum_{i,j=1}^n a^{ij}(x) \,\xi_i \,\xi_j \leq \Lambda \,|\xi|^2
\end{equation}
for all $x \in \Omega$ and $\xi \in \mathbb{R}^n$, and some constants $0 < \lambda \leq \Lambda < \infty$.  Finally, we let $1 < p < \infty$ be a constant, $k_+,k_- \in L^1(\Gamma)$ with $k_+,k_- \geq 0$ on $\Gamma$, and $h \in L^{\infty}(\Gamma)$. We are interested in studying the oblique derivative problem
\begin{align}
	&D_i (a^{ij} D_j u) = 0 \qquad \text{ in } \Omega, \nonumber \\
	&a^{ij} D_j u \,\nu_i = -k_+ ((u-h)^+)^{p-1} + k_- ((u-h)^-)^{p-1}\quad \text{ on } \Gamma . \label{main ptop}
\end{align}
Here $v^+=\max\{v,0\}$, $v^-=-\min\{v,0\}\geq 0$ and $\nu = (\nu_1,\nu_2,\ldots,\nu_n)$ is the outward unit normal to $\Omega$ with respect to the standard Euclidian metric. We say that $u \in W^{1,2}(\Omega)$ is a \emph{weak solution} to \eqref{main ptop} if
\begin{equation} \label{main weak ptop}
	\int_{\Omega} a^{ij} D_j u \,D_i \zeta + \int_{\Gamma} (k_+ ((u-h)^+)^{p-1} - k_- ((u-h)^-)^{p-1}) \,\zeta = 0
\end{equation}
for all $\zeta \in W^{1,2}(\Omega)$ with $\int_{\Gamma} (k_+ + k_-) \,|\zeta|^p < \infty$ and $\zeta = 0$ near $\partial \Omega \setminus \Gamma$.

The study of the model \eqref{main ptop} is motivated by applications to fluid dynamics and temperature control problems, which we now briefly describe. Following \cite[Section 2.2.2]{DL}, when considering the process of \emph{osmosis through semi-permeable walls} the region $\Omega$ consists of a porous medium occupied by a  viscous fluid which is only slightly compressible, with its field of pressure denoted by $u(x)$. We assume that the portion  $\Gamma$ of $\partial\Omega$ consists of a semi-permeable membrane of finite  thickness, i.e. the fluid can freely enter in $\Omega$, but the outflow of fluid is prevented. Combining the law of conservation of mass with Darcy's law, one finds that the equilibrium configuration for $u$ satisfies the equation
$$
\Delta u =f\mbox{ in }\Omega,
$$
where $f=f(x)$ is a given function. When a fluid pressure $h(x)$, for $x\in\Gamma$, is applied to $\Gamma$ on the outside of $\Omega$, one of two cases holds:
$
h(x)<u(x,t),\mbox{ or } h(x)\geq u(x,t).
$
In the former, the semi-permeable wall prevents the fluid from leaving $\Omega$, so that the flux  is null and thus
\begin{equation}\label{case1}
\frac{\partial u}{\partial \nu}=0.
\end{equation}
In the latter case, the fluid enters $\Omega$. It is reasonable to assume the outflow to be proportional to the difference in pressure, so that
\begin{equation}\label{case2}
-\frac{\partial u}{\partial \nu} =k(u-h),
\end{equation}
where $k>0$ measures the conductivity of the wall. Combining  \eqref{case1} and \eqref{case2}, we obtain the boundary condition
\begin{equation}\label{bdry}
\frac{\partial u}{\partial \nu} =k(u-h)^-\mbox{ on }\Gamma.
\end{equation}
In our problem \eqref{main ptop}, we allow  for fluid flow to occur both into and out of $\Omega$ with different permeability functions $k_+$ and $k_-$ (not necessarily constant), under the assumption that  the flux in each direction is proportional to a power of the pressure.

An alternate interpretation of the model \eqref{main ptop} is as a \emph{boundary temperature control problem}.  We assume that a continuous medium occupies the region $\Omega$ in $\mathbb{R}^n$, with boundary $\Gamma$ and outer unit normal $\nu$. Given a reference temperature $h(x)$, for $x\in \Gamma$, it is required that the temperature at the boundary $u(x)$ deviates as little as possible from $h(x)$. To this end, thermostatic controls are placed on the boundary to inject an appropriate heat flux when necessary. The controls are regulated as follows:

(i) If $u(x,t)=h(x)$, no correction is needed and therefore the heat flux is null (i.e., \eqref{case1} holds).

(ii) If $u(x,t)\neq h(x)$, a quantity of heat proportional to the difference between $u(x,t)$ and $h(x)$ is injected.

\noindent
We can thus write the boundary condition as
$
-\frac{\partial u}{\partial \nu}=\Phi(u),
$
where
\begin{equation*}
\Phi(u)=\begin{cases}
k_-(u-h)\qquad&\text{ if } u< h\\
0&\text{ if }u=h\\
k_+(u-h)\qquad&\text{ if } u> h
\end{cases}
\end{equation*}
which is precisely the condition in \eqref{main ptop} when $p=2$. For further details, we refer to  \cite[Section 2.3.1]{DL}, see also  \cite{AC1}  for the limiting case $k_-=0$ and $k_+=+\infty$,  and \cite{ALP15} for the case ${a}^{ij}=\delta^{ij}$, $p=1$, and $h=0$ in \eqref{J defn} below.

The study of \eqref{main ptop} was initiated in \cite{DJ}, where the constant coefficient case $A=I$ was considered, together with the simplifying assumptions $h\equiv 0$ and $k_+,\ k_-$ being constant. In \cite{DJ} the authors establish regularity properties of the solution and describe the structure of the free boundary. From the point of view of applications, however, it is important to allow for variable coefficients in \eqref{main ptop}. In fact, if $\Gamma$ is sufficiently smooth,  a standard flattening procedure  leads to the analysis of a similar model for a flat thin manifold,  where the relevant operator has,  however, variable  coefficients.

Quite relevant to our analysis, as observed in \cite{DJ},  are the two limiting cases $k_+=k_-=0$, and $k_+=0$ and $k_-=+\infty$ (or equivalently $k_+=+\infty$ and $k_-=0$). In the first one  $u$ is clearly the solution of a classical Neumann problem. The other one is more interesting, since the boundary condition becomes
$$
(u-h)\frac{\partial u}{\partial \nu}=0,
$$ and $u$ is a solution of the \emph{Signorini problem}, also known as the \emph{thin obstacle problem}. The Signorini problem has received a resurgence of attention in the last decade, due to the discovery of several families of powerful monotonicity formulas, which in turn have allowed to establish the optimal regularity of the solution, a full classification of free boundary points, smoothness of the free boundary at regular points, and the structure of the free boundary at singular points. We refer  the interested reader to \cite{AC1}, \cite{ACS}, \cite{CSS}, \cite{GP09}, \cite{GSVG14}, \cite{GPS}, \cite{DSS}, \cite{KPS}, see also the survey \cite{DS} and the references therein.

The general scheme of a solution to the variable-coefficient Signorini problem provides a blueprint for the solution of problem \eqref{main ptop}, but there are two new substantial difficulties. Due to the non-homogeneous nature of the boundary condition in \eqref{main ptop}, this problem does not admit global homogeneous solutions of any degree. This is in stark contrast with the thin obstacle problem, where the existence and classification  of such solutions  play a fundamental role. Moreover, in the  Signorini problem it is easily verified that continuity arguments force $u$ to be always above $h$ (that is, $h$ plays the role of an obstacle), whereas the case $h(x)>u(x)$ is no longer ruled out in \eqref{main ptop}. Allowing for both constants $k_+,\ k_-$ to be finite (even when one of the two vanishes) essentially destroys the one-phase character of the problem. Finally, we notice that allowing for variable coefficients and removing the hypothesis $h=0$ introduces a host of new technical challenges not present in \cite{DJ}.

We begin our study by showing existence and uniqueness of solutions to \eqref{main ptop}  using  variational techniques (see Theorem \ref{exist uniq thm} and Lemma \ref{Jmin is Jstat lemma}). We then proceed to establish the H\"older regularity of solutions in Theorem \ref{regularity thm}. More precisely, we show that if
$\kappa \geq 1$ is an integer and $\alpha \in (0,1)$ with $\kappa +\alpha\leq p$, $\Gamma$ is a relatively open $C^{\kappa,\alpha}$-portion of $\partial \Omega$, $a^{ij} \in C^{\kappa-1,\alpha}(\Omega \cup \Gamma)$, $k_+,k_- \in C^{\kappa-1,\alpha}(\Gamma)$ , and $h \in C^{\kappa-1,1}(\Gamma)$, then $$u\in C^{\kappa,\alpha}(\Omega\cup\Gamma).$$ Our approach is centered on a bootstrapping argument, which in turn is based on the Caccioppoli inequality, an initial modulus of continuity for the solution, and variational methods, combined with the classical regularity theory for the oblique derivative problem. Furthermore, we prove that such regularity is optimal (see Theorem \ref{irregularity thm}), by means of a careful construction of a suitable family of rescalings, and of an explicit solution to the ensuing limiting problem.  As an immediate consequence of the regularity of the solution and of the implicit function theorem, we obtain the following result concerning the regularity of the free boundary.
\begin{definition}
The {regular set of the free boundary} is defined as
$$
{\mathcal{R}(u)=\{(x',0)\in B'_1(0)\ |\ u(x',0)=h(x'),\ \nabla_{x'}u(x',0)\neq \nabla_{x'}h(x')\}}
$$
\end{definition}

\begin{theorem}\label{FBreg} {If $x_0\in \mathcal{R}(u)$, then in a neighborhood of $x_0$ the free boundary $\{ u(x',0)=0\}$ is a $C^{1,\alpha}-$ graph for all $0<\alpha<1$.}
\end{theorem}

The remainder of the paper is devoted to the study of rectifiability properties of the free boundary $ \Sigma(u) = \partial_{\Gamma}\{ x \in \Gamma : u(x) = h(x) \}$,  inspired by  the approach introduced in  \cite{GP09},  \cite{GSVG14},  \cite{GPS} and \cite{DJ}, but with the added difficulty of working with variable coefficients and non-zero ``obstacles" at once. In order to focus on this  aspect, it is useful to concentrate our attention on the following simplified version of \eqref{main ptop}
\begin{align}
	&D_i (a^{ij} D_j u) = 0 \qquad \text{ in } B_1^+(0), \nonumber \\
	&a^{nn} D_n u  = k_+ ((u-h)^+)^{p-1} - k_- ((u-h)^-)^{p-1}\quad \text{ on } B'_1(0). \label{hompb}
\end{align}
Here
\begin{align*}
	B^+_{\rho}(y) &= B_{\rho}(y) \cap \{ x_n > 0 \}, \\
	B'_{\rho}(y) &= B_{\rho}(y) \cap \{ x_n = 0 \}
\end{align*}
for each $y \in \mathbb{R}^{n-1} \times \{0\}$ and $\rho > 0$. We have used the notation, for points $x\in\mathbb{R}^n$, $x=(x',x_n)$, with $x'=(x_1, \dots, x_{n-1})$. For future reference, we will denote $$(\partial B_{\rho}(y))^+ = \partial B_{\rho}(y) \cap \{ x_n > 0 \}.$$
The first step in our program consists in essentially subtracting the  $h$ (or better, its Taylor polynomial of order $\kappa$) from the solution, so that we can rewrite our problem as
\begin{align}
	&D_i (a^{ij} D_j v) = f \qquad \text{ in } B_1^+(0), \nonumber \\
	&a^{nn} D_n v  = k_+ (v^+)^{p-1} - k_- (v^-)^{p-1}\quad \text{ on } B'_1(0) , \label{nonhom}
\end{align}
with the regularity of the right-hand side $f$ being correlated to the one of $h$. Next, we establish one of the main ingredients in our analysis in Theorem \ref{Almgren monotonicity thm}, namely a monotonicity formula of Almgren's type for solutions to \eqref{nonhom}. We were inspired by some of the ideas first introduced in \cite{GL} for unconstrained divergence form equations $\div(A\nabla u)=0$, and then in \cite{GSVG14} for the Signorini problem with variable coefficients, but we had to tackle  some delicate differences due to the nature of the boundary conditions. With this tool at our disposal, we proceed to establish optimal growth rates of the solution at free boundary points (Lemma \ref{L2 growth lemma}), which in turn allow us to control the convergence of a suitable family of rescalings to a homogeneous harmonic polynomial, the so-called blow-up (see Theorem \ref{tangent thm}). At this point, we introduce two families of functionals of Weiss' and Monneau's type, and we prove their almost-monotonicity in Theorem \ref{Weiss mono} and Theorem \ref{Monneau mono}, respectively. These results yield the nondegeneracy of the solution (Lemma \ref{nondeg lemma}) and the continuous dependance of the blow-up limits on the free boundary point (Theorem \ref{conts tangent thm}). We explicitly observe that, to the best of our knowledge, this is the first result of this kind for variable coefficients and non-zero obstacles. Our final main result asserts the rectifiability of the free boundary.

\begin{theorem}\label{rectifiability thm}
Let $2 \leq p < \infty$, and let $\kappa$, $\nu$ be positive integers with $2\leq \nu < \kappa$.  Let $a^{ij} \in C^{\kappa -1,1}(B^+_1(0) \cup B'_1(0))$ such that $a^{ij} = a^{ji}$ on $B^+_1(0)$, satisfying \eqref{ellipticity hyp}  for some constants $0 < \lambda \leq \Lambda < \infty$, and such that \eqref{regularity assumption} holds true on $B'_1(0)$.  Let $k_+,k_- \in C^{0,1}(B'_1(0))$ with $k_+,k_- \geq 0$ and $h \in C^{\kappa ,1}(B'_1(0))$.  Let $u \in W^{1,2}(B^+_1(0))$ be a solution to \eqref{main ptop}.   Then for each $d = 0,1,2,\ldots,n-2$, $\Sigma_{\nu}^d$ is contained in a countable union of $d$-dimensional $C^1$-submanifolds of $B'_1(0)$.
\end{theorem}
The set $\Sigma_{\nu}^d$ consists of free boundary points with Almgren's frequency $\nu$ and degree $d$. We refer to Section \ref{rect} for its precise definition. The proof of Theorem \ref{rectifiability thm} hinges on  Theorem~\ref{conts tangent thm}, combined with  Whitney's extension and the implicit function theorem.

\subsection{Structure of the paper}
The paper is organized as follows. In Section 2 we introduce the variational formulation of the problem. In Section 3 we prove existence and uniqueness of solutions.  In Section 4 we establish the optimal regularity of solutions. In Section 5 we show how to transform problem \eqref{hompb} into \eqref{nonhom}. Section 6 is devoted to proof  of the almost-monotonicity of a truncated functional of Almgren type, and in Section 7 we infer  growth  properties of the solution near free boundary points as a consequence. In Section 8 we introduce the Almgren rescalings, and discuss their blow-up limits. In Sections 9 and 10 we prove the almost-monotonicity of  Weiss-type and Monneau-type functionals, respectively.  Finally, in Section 11 we establish non-degeneracy of solutions and continuous dependance of the blow-up limits on the free boundary point, and  prove Theorem \ref{rectifiability thm}.

\section{Preliminaries}

Our goal is to understand the existence and regularity of solutions to the penalized thin obstacle problem, following some of the ideas introduced in \cite{DJ}.  Throughout this section,
\begin{definition} \label{J defn} {\rm
$J$ is a functional of $v \in W^{1,2}(\Omega)$ defined by
\begin{equation} \label{J functional}
	J(v) = \frac{1}{2} \int_{\Omega} a^{ij} D_i v \,D_j v + \frac{1}{p} \int_{\Gamma} (k_+ (v-h)^+)^p + k_- (v-h)^-)^p) ,
\end{equation}
where $v$ takes the values of its trace on $\Gamma$ and we let $(v-h)^+ = \max\{v-h,0\}$ and $(v-h)^- = -\min\{v-h,0\}$.
} \end{definition}

The functional $J$ satisfies the following transformation formula under $C^1$-diffeomorphisms, which follows by a straight-forward application of change of variables.

\begin{lemma} \label{J change var lemma}
Let $\Omega, \widetilde{\Omega} \subset \mathbb{R}^n$ be bounded Lipschitz domains, $\Gamma$ be a smooth relatively open subset of $\partial \Omega$, $\widetilde{\Gamma}$ be a smooth relatively open subset of $\partial \widetilde{\Omega}$, and $\varphi : \Omega \cup \Gamma \rightarrow \widetilde{\Omega} \cup \widetilde{\Gamma}$ be a $C^1$-diffeomorphism.  Let $J$ be the functional defined by \eqref{J functional}, where $1 < p < \infty$ is a constant, $a^{ij} \in L^{\infty}(\Omega)$ such that $a^{ij} = a^{ji}$ on $\Omega$, $k_+,k_- \in L^1(\Gamma)$ with $k_+,k_- \geq 0$ on $\Gamma$, and $h \in L^{\infty}(\Gamma)$.  Then for each $v \in W^{1,2}(\Omega)$,
\begin{equation*}
	J(v) = \frac{1}{2} \int_{\widetilde{\Omega}} \widetilde{a}^{ij} D_i \widetilde{v} \,D_j \widetilde{v}
		+ \frac{1}{p} \int_{\widetilde{\Gamma}} (\widetilde{k}_+ (\widetilde{v}-\widetilde{h})^+)^p + \widetilde{k}_- (\widetilde{v}-\widetilde{h})^-)^p)
\end{equation*}
where $\widetilde{v} = v \circ \varphi^{-1}$, $\widetilde{a}^{ij} = ((a^{kl} D_k \varphi^i \,D_l \varphi^j) \circ \varphi^{-1}) \,J\varphi^{-1}$, $\widetilde{k}_{\pm} = (k_{\pm} \circ \varphi^{-1}) \,J_{\widetilde{\Gamma}} \varphi^{-1}$, and $\widetilde{h} = h \circ \varphi^{-1}$ where $J \varphi^{-1}$ is the Jacobian of $\varphi^{-1} : \widetilde{\Omega} \rightarrow \Omega$ and $J_{\widetilde{\Gamma}} \varphi^{-1}$ is the Jacobian of $\varphi^{-1} |_{\widetilde{\Gamma}} : \widetilde{\Gamma} \rightarrow \Gamma$.
\end{lemma}

Let $J$ be as in Definition~\ref{J defn}.  We say that $u \in W^{1,2}(\Omega)$ is \emph{$J$-minimizing} if $J(u) < \infty$ and
\begin{equation*}
	J(u) \leq J(v)
\end{equation*}	
for all $v \in W^{1,2}(\Omega)$ with $u = v$ on $\partial \Omega \setminus \Gamma$, where $u,v$ take the values of their trace on $\partial \Omega \setminus \Gamma$.  %Let $\nu = (\nu_1,\nu_2,\ldots,\nu_n)$ be the outward unit normal to $\Omega$ with respect to the standard Euclidian metric.
Recall that $u \in W^{1,2}(\Omega)$ is a weak solution to \eqref{main ptop} if \eqref{main weak ptop} holds true for all $\zeta \in W^{1,2}(\Omega)$ with $\int_{\Gamma} (k_+ + k_-) \,|\zeta|^p < \infty$ and $\zeta = 0$ near $\partial \Omega \setminus \Gamma$.
The assumption that $\int_{\Gamma} (k_+ + k_-) \,|\zeta|^p < \infty$ guarantees that if $u \in W^{1,2}(\Omega)$ with $J(u) < \infty$, then by Young's inequality
\begin{align*}
	\int_{\Gamma} k_{\pm} ((u+t\zeta-h)^{\pm})^p &\leq \int_{\Gamma} k_{\pm} ((u-h)^{\pm} + |t| \,|\zeta|)^p
		\\&\leq 2^{p-1} \int_{\Gamma} k_{\pm} ((u-h)^{\pm})^p + 2^{p-1} \int_{\Gamma} k_{\pm} \,|\zeta|^p < \infty , \\
	\int_{\Gamma} k_{\pm} ((u-h)^{\pm})^{p-1} \,|\zeta|
		&\leq \frac{p-1}{p} \int_{\Gamma} k_{\pm} ((u-h)^{\pm})^p + \frac{1}{p} \int_{\Gamma} k_+ \,|\zeta|^p < \infty .
\end{align*}
Hence $J(u+t\zeta) < \infty$ for all $-1 \leq t \leq 1$ and the second integral on the right-hand side of \eqref{main weak ptop} is also finite.  Moreover, $k_+,k_- \in L^1(\Gamma)$ guarantees that $\int_{\Gamma} (k_+ + k_-) \,|\zeta|^p < \infty$ for all $\zeta \in C^1_c(\Omega \cup \Gamma)$.

\begin{lemma} \label{Jmin is Jstat lemma}
Let $J$ be as in Definition~\ref{J defn}.  For each $u \in W^{1,2}(\Omega)$ with $J(u) < \infty$, $u$ is $J$-minimizing if and only if $u$ is a weak solution to \eqref{main ptop}.
\end{lemma}

\begin{proof}
If $u$ is $J$-minimizing, then by a standard variational argument $u$ is a weak solution to \eqref{main ptop}; see for instance the proof of Lemma~4.1 of~\cite{ALP15}.  To see the converse, suppose that $u$ is a weak solution to \eqref{main ptop} and fix any function $\zeta \in W^{1,2}(\Omega)$ with $\int_{\Gamma} (k_+ + k_-) \,|\zeta|^p < \infty$ and $\zeta = 0$ near $\partial \Omega \setminus \Gamma$.  Define $F : \Gamma \times \mathbb{R} \rightarrow [0,\infty)$ by
\begin{equation} \label{F defn}
	F(x,z) = \frac{1}{p} \,k_+(x) \,(z^+)^p + \frac{1}{p} \,k_-(x) \,(z^-)^p
\end{equation}
for $x \in \Gamma$ and $z \in \mathbb{R}$.  Since $F(x,\,\cdot\,)$ is convex for each $x \in \Gamma$,
\begin{align} \label{Jmin is Jstat eqn1}
	&F(x,u + \zeta - h) - F(x,u - h) - D_z F(x,u - h) \,\zeta \nonumber
		\\&\hspace{15mm} = \int_0^1 (D_z F(x,u + t \zeta - h) - D_z F(x,u - h)) \,\zeta \,dt \geq 0
\end{align}
on $\Gamma$.  Hence by \eqref{Jmin is Jstat eqn1}
\begin{align} \label{Jmin is Jstat eqn2}
	J(u + \zeta) =&\, \frac{1}{2} \int_{\Omega} a^{ij} D_i (u+\zeta) \,D_j (u+\zeta) + \int_{\Gamma} F(x,u + \zeta - h) \nonumber
		\\ \geq&\, \frac{1}{2} \int_{\Omega} (a^{ij} D_i u \,D_j u + 2 a^{ij} D_j u \,D_i \zeta + a^{ij} D_i \zeta \,D_j \zeta) \nonumber
			\\& + \int_{\Gamma} (F(x,u - h) + D_z F(x,u - h) \,\zeta) .
\end{align}
Note that we have equality in \eqref{Jmin is Jstat eqn2} if and only if we have $D_z F(x,u + t \zeta - h) = D_z F(x,u - h)$ on $\Gamma$ for all $0 \leq t \leq 1$, so that we have equality in \eqref{Jmin is Jstat eqn1}.  The assumption that $u$ is a weak solution to \eqref{main ptop} means that \eqref{main weak ptop} holds true.  We can rewrite \eqref{main weak ptop} as
\begin{equation} \label{Jmin is Jstat eqn3}
	\int_{\Omega} a^{ij} D_j u \,D_i \zeta + \int_{\Gamma} D_z F(x,u-h) \,\zeta = 0 .
\end{equation}
Combining \eqref{Jmin is Jstat eqn2}, \eqref{Jmin is Jstat eqn3} and \eqref{ellipticity hyp},  we obtain
\begin{align} \label{Jmin is Jstat eqn4}
	J(u + \zeta) \geq&\, \frac{1}{2} \int_{\Omega} a^{ij} D_i u \,D_j u + \frac{1}{2} \int_{\Omega} a^{ij} D_i \zeta \,D_j \zeta + \int_{\Gamma} F(x,u - h) \geq J(u)
\end{align}
(with equality if and only if $\zeta$ is constant and $D_z F(x,u + t \zeta - h) = D_z F(x,u - h)$ on $\Gamma$ for all $0 \leq t \leq 1$).
\end{proof}

\section{Existence and uniqueness of $J$-minimizers}

Theorem~\ref{exist uniq thm} below is the general existence and uniqueness theorem for $J$-minimizing functions on bounded domains $\Omega$.  We distinguish between the cases $\mathcal{H}^{n-1}(\partial \Omega \setminus \Gamma) > 0$ and $\mathcal{H}^{n-1}(\partial \Omega \setminus \Gamma) = 0$ (i.e.~$\Gamma = \partial \Omega$ up to a set of $\mathcal{H}^{n-1}$-measure zero).  For the case $\mathcal{H}^{n-1}(\partial \Omega \setminus \Gamma) > 0$ we include the Dirichlet boundary condition $u = \varphi$ on $\partial \Omega \setminus \Gamma$.  In the case $\Gamma = \partial \Omega$ uniqueness is determined by the values of $k_+$, $k_-$ and $h$, and one might have infinitely many $J$-minimizing constant functions $c$ with $J(c) = 0$.  %An illuminating example is when $\Gamma = \partial \Omega$ and $k_+ > 0$ and $k_- \geq 0$ are constants, for which there is a unique $J$-minimizing function if $k_- > 0$ and every constant function $\leq \inf_{\Gamma} h$ is $J$-minimizing if $k_- = 0$.
We show existence by the direct method, and uniqueness using the equality case of \eqref{Jmin is Jstat eqn4}.

\begin{theorem} \label{exist uniq thm}
Let $\Omega \subset \mathbb{R}^n$ be a bounded Lipschitz domain and $\Gamma$ be a relatively open subset of $\partial \Omega$.  Let $J$ be as in Definition~\ref{J defn} where $a^{ij} \in L^{\infty}(\Omega)$ are such that $a^{ij} = a^{ji}$ on $\Omega$ and \eqref{ellipticity hyp} hold true for some constants $0 < \lambda \leq \Lambda < \infty$, $1 < p < \infty$ is a constant, $k_+,k_- \in L^1(\Gamma)$ with $k_+,k_- \geq 0$ on $\Gamma$, and $h \in L^{\infty}(\Gamma)$.
\begin{enumerate}
	\item[(i)]  If $\mathcal{H}^{n-1}(\partial \Omega \setminus \Gamma) > 0$, for each $\varphi \in W^{1,2}(\Omega)$ with $J(\varphi) < \infty$ there exists a unique $J$-minimizing function $u \in W^{1,2}(\Omega)$ with $u = \varphi$ on $\partial \Omega \setminus \Gamma$.
\end{enumerate}
Suppose instead that $\mathcal{H}^{n-1}(\partial \Omega \setminus \Gamma) = 0$ and set
\begin{equation*}
	M_+ = \inf_{\Gamma \cap \{k_+ > 0\}} h, \quad M_- = \sup_{\Gamma \cap \{k_- > 0\}} h
\end{equation*}
(with the usual convention that $M_+ = +\infty$ if $k_+ = 0$ on $\Gamma$ and $M_- = -\infty$ if $k_- = 0$ on $\Gamma$).
\begin{enumerate}
	\item[(ii)] If $\mathcal{H}^{n-1}(\partial \Omega \setminus \Gamma) = 0$ and $M_+ \leq M_-$, then there exists a unique $J$-minimizing function $u \in W^{1,2}(\Omega)$.

	\item[(iii)] If $\mathcal{H}^{n-1}(\partial \Omega \setminus \Gamma) = 0$ and $M_- \leq M_+$, then the $J$-minimizing functions $u \in W^{1,2}(\Omega)$ are precisely the constant functions $u$ with $M_- \leq u \leq M_+$.
\end{enumerate}
\end{theorem}

To prove Theorem~\ref{exist uniq thm}(i), we will need the following variant of the Poincar\'e inequality.

\begin{lemma} \label{poincare lemma}
Let $\Omega \subset \mathbb{R}^n$ be a bounded Lipschitz domain and $\Gamma$ be a relatively open subset of $\partial \Omega$ with $\mathcal{H}^{n-1}(\partial \Omega \setminus \Gamma) > 0$.  Suppose that $v \in W^{1,2}(\Omega)$ with $v = 0$ on $\partial \Omega \setminus \Gamma$.  Then
\begin{equation*}
	\int_{\Omega} v^2 \leq C \int_{\Omega} |Dv|^2
\end{equation*}
for some constant $C = C(n,\Omega,\Gamma) \in (0,\infty)$.
\end{lemma}

\begin{proof}
Suppose to the contrary that for $\ell = 1,2,3,\ldots$ there exist non-zero functions $v_{\ell} \in W^{1,2}(\Omega)$ such that $v_{\ell} = 0$ on $\partial \Omega \setminus \Gamma$ and $\|Dv_{\ell}\|_{L^2(\Omega)} \leq (1/\ell) \,\|v_{\ell}\|_{L^2(\Omega)}$.  By scaling we may also assume that $\|v_{\ell}\|_{L^2(\Omega)} = 1$ and thus $\|Dv_{\ell}\|_{L^2(\Omega)} \leq 1/\ell$.  By the Rellich compactness theorem, there exists a constant function $v$ on $\Omega$ such that, after passing to a subsequence, $v_{\ell} \rightarrow v$ strongly in $L^2(\Omega)$ and $Dv_{\ell} \rightarrow 0$ weakly in $L^2(\Omega)$.  In particular, $v$ is a non-zero constant function with $\|v\|_{L^2(\Omega)} = 1$.  Since the trace operator $W^{1,2}(\Omega) \rightarrow L^2(\partial \Omega)$ is a bounded compact linear operator, after passing to a further sequence, $v_{\ell} \rightarrow v$ strongly in $L^2(\partial \Omega)$ and pointwise $\mathcal{H}^{n-1}$-a.e.~on $\partial \Omega$.  In particular, since $v_{\ell} = 0$ on $\partial \Omega \setminus \Gamma$, $v = 0$ on $\partial \Omega \setminus \Gamma$.  Since $v$ is a constant function on $\Omega$, the trace of $v$ on $\partial \Omega$ is equal to the constant value of $v$.  Therefore, $v = 0$ on $\Omega$, contradicting $v$ being non-zero.
\end{proof}

\begin{proof}[Proof of Theorem~\ref{exist uniq thm}(i)]
To prove the existence of a $J$-minimizing function, for $\ell = 1,2,3,\ldots$ let $u_{\ell} \in W^{1,2}(\Omega)$ such that $u_{\ell} = \varphi$ on $\partial \Omega$ and
\begin{equation*}
	\lim_{\ell \rightarrow \infty} J(u_{\ell}) = \inf \{ J(v) : v \in W^{1,2}(\Omega), \,v = \varphi \text{ on } \partial \Omega \} .
\end{equation*}
By Lemma~\ref{poincare lemma} applied to $v = u_{\ell} - \varphi$ and \eqref{ellipticity hyp},
\begin{align*}
	\|u_{\ell}\|_{W^{1,2}(\Omega)}
		&\leq \|u_{\ell} - \varphi\|_{W^{1,2}(\Omega)} + \|\varphi\|_{W^{1,2}(\Omega)}
		\leq C \|Du_{\ell} - D\varphi\|_{L^2(\Omega)} + \|\varphi\|_{W^{1,2}(\Omega)}
		\\&\leq C \|Du_{\ell}\|_{L^2(\Omega)} + (C+1) \,\|\varphi\|_{W^{1,2}(\Omega)}
		\leq \frac{C}{\sqrt{\lambda}} \,J(u_{\ell})^{1/2} + (C+1) \,\|\varphi\|_{W^{1,2}(\Omega)}
		\\&\leq \frac{C}{\sqrt{\lambda}} \,(J(\varphi)^{1/2} + 1) + (C+1) \,\|\varphi\|_{W^{1,2}(\Omega)}
\end{align*}
for all sufficiently large $\ell$, where $C = C(n,\Omega,\Gamma) \in (0,\infty)$ is a constant.  Thus by the Rellich compactness theorem, there exists a function $u \in W^{1,2}(\Omega)$ such that, after passing to a subsequence, $u_{\ell} \rightarrow u$ strongly in $L^2(\Omega)$ and $Du_{\ell} \rightarrow Du$ weakly in $L^2(\Omega,\mathbb{R}^n)$.  Since the trace operator $W^{1,2}(\Omega) \rightarrow L^2(\partial \Omega)$ is a bounded compact linear operator, after passing to a further sequence, $u_{\ell} \rightarrow u$ strongly in $L^2(\partial \Omega)$ and pointwise $\mathcal{H}^{n-1}$-a.e.~on $\partial \Omega$.  Since $u_{\ell} = \varphi$ on $\partial \Omega \setminus \Gamma$, $u = \varphi$ on $\partial \Omega \setminus \Gamma$.  The weak convergence $Du_{\ell} \rightarrow Du$ in $L^2(\Omega,\mathbb{R}^n)$ implies that
\begin{equation*}
	\lim_{k \rightarrow \infty} \int_{\Omega} a^{ij} D_i u \,\zeta_j = \liminf_{\ell \rightarrow \infty} \int_{\Omega} a^{ij} D_i u_{\ell} \,\zeta_j
\end{equation*}
for all $\zeta = (\zeta_1,\zeta_2,\ldots,\zeta_n) \in C^0_c(\Omega, \mathbb{R}^n)$ (as $a^{ij} \,\zeta_j \in C^0_c(\Omega)$); that is, $Du_{\ell} \rightarrow Du$ weakly in $L^2(\Omega,\mathbb{R}^n)$ equipped with the inner product $\langle \xi, \zeta \rangle = \int_{\Omega} a^{ij} \,\xi_i \,\zeta_j$.  Hence
\begin{equation} \label{exist uniq eqnA1}
	\int_{\Omega} a^{ij} D_i u \,D_j u \leq \liminf_{\ell \rightarrow \infty} \int_{\Omega} a^{ij} D_i u_{\ell} \,D_j u_{\ell} .
\end{equation}
Recalling from above that $u_{\ell} \rightarrow u$ pointwise $\mathcal{H}^{n-1}$-a.e.~on $\partial \Omega$, $(u_{\ell} - h)^+ \rightarrow (u - h)^+$ pointwise $\mathcal{H}^{n-1}$-a.e.~on $\Gamma$.  By Fatou's lemma,
\begin{equation} \label{exist uniq eqnA2}
	\int_{\Gamma} k_+ ((u - h)^+)^p \leq \liminf_{\ell \rightarrow \infty} \int_{\Gamma} k_+ ((u_{\ell} - h)^+)^p .
\end{equation}
Similarly,
\begin{equation} \label{exist uniq eqnA3}
	\int_{\Gamma} k_- ((u - h)^-)^p \leq \liminf_{\ell \rightarrow \infty} \int_{\Gamma} k_- ((u_{\ell} - h)^-)^p .
\end{equation}
By \eqref{exist uniq eqnA1}, \eqref{exist uniq eqnA2}, and \eqref{exist uniq eqnA3},
\begin{equation*}
	J(u) \leq \liminf_{\ell \rightarrow \infty} J(u_{\ell})
\end{equation*}
and it follows that $u$ is $J$-minimizing.

To show uniqueness of $J$-minimizing functions, suppose that $u, v \in W^{1,2}(\Omega)$ are both $J$-minimizing functions with $u = v = \varphi$ on $\partial \Omega \setminus \Gamma$.  Set $\zeta = v-u$ so that $\zeta \in W^{1,2}(\Omega)$ with $\int_{\Gamma} (k_+ + k_-) \,|\zeta|^p < \infty$ (by Minkowski's inequality) and $\zeta = 0$ on $\partial \Omega \setminus \Gamma$.  Since $J(u+\zeta) = J(u)$, i.e.~equality holds true in \eqref{Jmin is Jstat eqn4}, $\zeta$ must be a constant function on $\Omega$ and in particular $\zeta = 0$ on $\Omega$.  Therefore $u = v$ on $\Omega$.
\end{proof}

\begin{proof}[Proof of Theorem~\ref{exist uniq thm}(ii)]
To prove the existence of a $J$-minimizing function, for $\ell = 1,2,3,\ldots$ let $u_{\ell} \in W^{1,2}(\Omega)$ such that
\begin{equation*}
	\lim_{\ell \rightarrow \infty} J(u_{\ell}) = \inf \{ J(v) : v \in W^{1,2}(\Omega) \} .
\end{equation*}
Note that since $M_+ \leq M_-$ we have $M_+ > -\infty$ and $M_- < \infty$.
Truncate $u_{\ell}$ by letting $\overline{u}_{\ell} \in W^{1,2}(\Omega)$ be given by
\begin{equation*}
	\overline{u}_{\ell} = \begin{cases}
		M_+ &\text{if } u_{\ell} < M_+ \\
		u_{\ell} &\text{if } M_+ \leq u_{\ell} \leq M_- \\
		M_- &\text{if } u_{\ell} > M_-
	\end{cases}
\end{equation*}
Clearly $M_+ \leq \overline{u}_{\ell} \leq M_-$ on $\Omega$.  Moreover,
\begin{align*}
	J(\overline{u}_{\ell}) = \frac{1}{2} \int_{\Omega \cap \{M_+ < u_{\ell} < M_-\}} |Du_{\ell}|^2
		&+ \frac{1}{p} \int_{\Gamma \cap \{\overline{u}_{\ell} > h \geq M_+\}} k_+ (\overline{u}_{\ell} - h)^p
		\\&+ \frac{1}{p} \int_{\Gamma \cap \{\overline{u}_{\ell} < h \leq M_-\}} k_- (h - \overline{u}_{\ell})^p \leq J(u_{\ell}) ,
\end{align*}
where the first step uses $k_+ = 0$ on $\Gamma \cap \{h < M_+\}$ and $k_- = 0$ on $\Gamma \cap \{h > M_-\}$ and the last step uses $\overline{u}_{\ell} \leq u_{\ell}$ on $\Gamma \cap \{\overline{u}_{\ell} > h \geq M_+\}$ and $\overline{u}_{\ell} \geq u_{\ell}$ on $\Gamma \cap \{\overline{u}_{\ell} < h \leq M_-\}$.  Thus, using \eqref{ellipticity hyp},
\begin{equation*}
	\limsup_{\ell \rightarrow \infty} \lambda \,\|D\overline{u}_{\ell}\|_{L^2(\Omega)}^2
		\leq \lim_{\ell \rightarrow \infty} J(\overline{u}_{\ell})
		\leq \lim_{\ell \rightarrow \infty} J({u}_{\ell}) =\inf \{ J(v) : v \in W^{1,2}(\Omega) \} .
\end{equation*}
Thus by the Rellich compactness theorem, there exists a function $u \in W^{1,2}(\Omega)$ such that, after passing to a subsequence, $\overline{u}_{\ell} \rightarrow u$ strongly in $L^2(\Omega)$ and $D\overline{u}_{\ell} \rightarrow Du$ weakly in $L^2(\Omega,\mathbb{R}^n)$.  Moreover, by the compactness of the trace operator $W^{1,2}(\Omega) \rightarrow L^2(\partial \Omega)$, after passing to a further sequence, $u_{\ell} \rightarrow u$ strongly in $L^2(\partial \Omega)$ and pointwise $\mathcal{H}^{n-1}$-a.e.~on $\partial \Omega$.  Clearly $M_+ \leq u \leq M_-$ on $\Omega$.  Arguing as we did for (i), $u$ is $J$-minimizing.

To show uniqueness of $J$-minimizing functions, suppose that $u, v \in W^{1,2}(\Omega)$ are both $J$-minimizing functions. Without loss of generality assume that $M_+ \leq u \leq M_-$ on $\Omega$.  Set $\zeta = v - u$ so that $\zeta \in W^{1,2}(\Omega)$ with $\int_{\Gamma} (k_+ + k_-) \,|\zeta|^p < \infty$ (by Minkowski's inequality).  Since $J(u+\zeta) = J(u)$, i.e.~equality holds true in \eqref{Jmin is Jstat eqn4}, $\zeta$ must be a constant function on $\Omega$ and
\begin{equation} \label{exist uniq eqnB1}
	0 = \left. \frac{d}{dt} D_z F(x,u + t \zeta - h) \right.
	= \begin{cases}
		\hspace{3mm} (p-1) \,k_+ (u + t \zeta - h)^{p-2} \,\zeta &\text{if } u + t \zeta > h \\
		-(p-1) \,k_- (h - u - t \zeta)^{p-2} \,\zeta &\text{if } u + t \zeta < h
	\end{cases}
\end{equation}
on $\Gamma \cap \{ u + t \zeta \neq h \}$ for all $0 \leq t \leq 1$, where $F$ is as in \eqref{F defn}.  If $\zeta > 0$, then by \eqref{exist uniq eqnB1} we have $k_+ = 0$ on $\Gamma \cap \{ u + \zeta > h \}$.
%However, by the definition of $M_+$
However, recalling that $M_+ > -\infty$ and applying the definition of $M_+$, the set $S = \Gamma \cap \{ k_+ > 0 \} \cap \{ h < M_+ + \zeta \}$ has $\mathcal{H}^{n-1}(S) > 0$ and we have $k_+ > 0$ and $u + \zeta \geq M_+ + \zeta > h$ on $S$, giving us a contradiction.  Hence we cannot have $\zeta > 0$.  By similar reasoning we cannot have $\zeta < 0$.  Therefore, $\zeta = 0$ on $\Omega$, that is $u = v$ on $\Omega$.
\end{proof}

\begin{proof}[Proof of Theorem~\ref{exist uniq thm}(iii)]
We observe that $J(v) \geq 0$ for all $v \in W^{1,2}(\Omega)$ and $J(v) = 0$ if and only if $v$ is a constant function on $\Omega$, $v \leq h$ on $\Gamma \cap \{k_+ > 0\}$, and $v \geq h$ on $\Gamma \cap \{k_- > 0\}$, i.e.~$v$ is a constant function with $M_- \leq v \leq M_+$.
\end{proof}

\section{Optimal regularity near points of $\Gamma$} \label{sec:regularity}

The goal of this section is to investigate the local regularity of a weak solution $u \in W^{1,2}(\Omega)$ to \eqref{main ptop}.    Suppose that $\Gamma$ is a $C^{1,1}$-hypersurface and $x_0 \in \Gamma$.  After a $C^{1,1}$-change of coordinates we may assume that $x_0 = 0$ and that $\Omega = B^+_1(0)$ and $\Gamma = B'_1(0)$, see Lemma~\ref{J change var lemma}.  Without loss of generality, throughout the remainder of the paper, we will assume
\begin{equation}\label{identity_origin}
a^{ij}(0)=\delta_{ij}, \quad i,j=1,\dots,n.
\end{equation}
Moreover, by~\cite[Theorem~B.1]{GSVG14} after a further $C^{1,1}$-change of coordinates we may assume that
\begin{equation} \label{regularity assumption}
	a^{in} = a^{ni} = 0  \text{ on }B'_1(0), \quad i = 1,2,\ldots,n-1.
\end{equation}

We begin by establishing a Caccioppoli-type inequality.
\begin{lemma} \label{regularity lemma1}
Let $a^{ij} \in L^{\infty}(B^+_1(0))$ such that $a^{ij} = a^{ji}$ on $B^+_1(0)$, \eqref{ellipticity hyp} holds true for some constants $0 < \lambda \leq \Lambda < \infty$, and \eqref{regularity assumption} holds true on $B'_1(0)$.  Let $1 < p < \infty$ be a constant, $k_+,k_- \in L^1(B'_1(0))$ with $k_+,k_- \geq 0$ on $B'_1(0)$, and $h \in L^{\infty}(B'_1(0))$.  If $u \in W^{1,2}(B^+_1(0))$ is a weak solution to \eqref{main ptop}, then
\begin{align} \label{regularity1 concl}
	&\int_{B^+_{\rho/2}(0)} |Du|^2 + \frac{1}{p} \int_{B'_{\rho/2}(0)} (k_+ ((u-h)^+)^p + k_- ((u-h)^-)^p) \nonumber
		\\&\leq \frac{C}{\rho^2} \int_{B^+_{\rho}(0)} u^2 + C \int_{B'_{\rho}(0)} (k_+ + k_-) \,|h|^p
\end{align}
for all $0 < \rho \leq 1$ and some constant $C = C(p,\lambda,\Lambda) \in (0,\infty)$.
\end{lemma}
\begin{proof}
Set $\zeta = \eta^2 u$ in \eqref{main weak ptop}, where $\eta \in C^1(B^+_1 \cup B'_1)$ with $\eta = 0$ on $(\partial B_1)^+$ to obtain
\begin{align*}
	&\int_{B^+_1} a^{ij} D_i u \,D_j u \,\eta^2 + \int_{B'_1} (k_+ ((u-h)^+)^{p-1} - k_- ((u-h)^-)^{p-1}) \,u \,\eta^2
		\\&= -2 \int_{B^+_1} a^{ij} u \,D_j u \,\eta \,D_i \eta .
\end{align*}
Using \eqref{ellipticity hyp} and Cauchy's inequality,
\begin{equation*}
	\frac{\lambda}{2} \int_{B^+_1} |Du|^2 \,\eta^2 + \int_{B'_1} (k_+ ((u-h)^+)^{p-1} - k_- ((u-h)^-)^{p-1}) \,u \,\eta^2
		\leq \frac{2\Lambda^2}{\lambda} \int_{B^+_1} u^2 \,|D\eta|^2 .
\end{equation*}
Since $((u-h)^{\pm})^{p-1} u = \pm ((u-h)^{\pm})^p + ((u-h)^{\pm})^{p-1} h$ on $B'_1$,
\begin{align*}
	&\frac{\lambda}{2} \int_{B^+_1} |Du|^2 \,\eta^2 + \int_{B'_1} (k_+ ((u-h)^+)^p + k_- ((u-h)^-)^p) \,\eta^2
	\\&\leq \frac{2\Lambda^2}{\lambda} \int_{B^+_1} u^2 \,|D\eta|^2 - \int_{B'_1} (k_+ ((u-h)^+)^{p-1} - k_- ((u-h)^-)^{p-1}) \,h \,\eta^2
\end{align*}
Hence by Young's inequality $ab \leq \frac{p-1}{p} \,a^{\frac{p}{p-1}} + \frac{1}{p} \,b^p$ with $a = (u-h)^{\pm}$ and $b = |h|$,
\begin{align*}
	&\frac{\lambda}{2} \int_{B^+_1} |Du|^2 \,\eta^2 + \frac{1}{p} \int_{B'_1} (k_+ ((u-h)^+)^p + k_- ((u-h)^-)^p) \,\eta^2
		\\&\leq \frac{2\Lambda^2}{\lambda} \int_{B^+_1} u^2 \,|D\eta|^2 + \frac{1}{p} \int_{B'_1} (k_+ + k_-) \,|h|^p \,\eta^2 .
\end{align*}
%Now choose $\eta$ so that $0 \leq \eta \leq 1$, $\eta = 1$ on $B^+_{\rho/2}$, $\eta = 0$ on $B^+_1 \setminus B^+_{\rho}$, and $|D\eta| \leq 3/\rho$, we obtain \eqref{regularity1 concl}.
Now to obtain \eqref{regularity1 concl} choose $\eta$ so that $0 \leq \eta \leq 1$, $\eta = 1$ on $B^+_{\rho/2}$, $\eta = 0$ on $B^+_1 \setminus B^+_{\rho}$, and $|D\eta| \leq 3/\rho$.
\end{proof}

Next, we prove the boundedness of solutions to \eqref{main ptop}.
\begin{lemma} \label{regularity lemma2}
Let $a^{ij} \in L^{\infty}(B^+_1(0))$ such that $a^{ij} = a^{ji}$ on $B^+_1(0)$, \eqref{ellipticity hyp} holds true for some constants $0 < \lambda \leq \Lambda < \infty$ and \eqref{regularity assumption} holds true on $B'_1(0)$.  Let $1 < p < \infty$ be a constant, $k_+,k_- \in L^1(B'_1(0))$ with $k_+,k_- \geq 0$ on $B'_1(0)$, and $h \in L^{\infty}(B'_1(0))$.  If $u \in W^{1,2}(B^+_1(0))$ is a weak solution to \eqref{main ptop}, then $u \in L^{\infty}(B^+_{1/2}(0))$ and
\begin{align}
	\label{regularity2 concl} \sup_{B^+_{1/2}(0)} |u| \leq C \|u\|_{L^2(B^+_1(0))} + C \|h\|_{L^{\infty}(B'_1(0))}
\end{align}
for some constant $C = C(n,\lambda,\Lambda) \in (0,\infty)$.
\end{lemma}
\begin{proof}
Set
\begin{equation*}
	M_- = \sup_{B'_1 \cap \{k_- > 0\}} h, \quad M_+ = \inf_{B'_1 \cap \{k_+ > 0\}} h .
\end{equation*}
%Recalling \eqref{regularity assumption}, extend $a^{ij}$ to a continuous function on $B_1$ by an even reflection, i.e. $a^{ij}(x',x_n) = a^{ij}(x',-x_n)$.
Extend $a^{ij}$ to an $L^{\infty}$-function on $B_1$ by even reflection, i.e.~$a^{ij}(x',x_n) = a^{ij}(x',-x_n)$, if $i,j < n$ and if $i = j = n$ and by odd reflection, i.e.~$a^{ij}(x',-x_n) = -a^{ij}(x',x_n)$, if $i < n$ and $j = n$ and if $i = n$ and $j < n$.
Similarly, extend $u$ to a $W^{1,2}$-function on $B_1$ %also
by even reflection, i.e. $u(x',x_n) = u(x',-x_n)$.  First we claim that
\begin{equation} \label{regularity2 eqn1}
	D_i (a^{ij} D_j (u - M_-)^+) \geq 0 \text{ weakly in } B_1 .
\end{equation}
Let $t > 0$ and $\zeta \in W^{1,2}_0(B_1)$ be an arbitrary nonnegative function.  Let's compare $u$ to $u_t \in W^{1,2}(B_1)$, where
\begin{equation*}
	u_t %= M_- + (u - M_- - t \zeta)^+ - (u - M_-)^-
	= \begin{cases}
		u &\text{if } u \leq M_- \\
		M_- &\text{if } M_- < u \leq M_- + t\zeta \\
		u-t\zeta &\text{if } u > M_- + t\zeta .
	\end{cases}
\end{equation*}
Notice that since $\zeta = 0$ on $\partial B_1$, $u_t = u$ on $\partial B_1$. Since $u_t \leq u$ on $B'_1$, $(u_t - h)^+ \leq (u-h)^+$ on $B'_1$.  Moreover, we have $u \geq u_t \geq M_- \geq h$ on $B'_1 \cap \{k_- > 0 \text{ and } u > M_-\}$ and $u_t = u$ on $B'_1 \cap \{u \leq M_-\}$, and therefore $k_- ((u_t - h)^-)^p = k_- ((u - h)^-)^p$ on $B'_1$.  Hence, recalling that $J(u) \leq J(u_t)$ for all $t > 0$ because  $u$ is $J$-minimizing, and utilizing the symmetry properties, we infer
\begin{align*}
	\int_{B_1 \cap \{u > M_-\}} a^{ij} D_i u \,D_j u &\leq \int_{B_1 \cap \{u > M_- + t\zeta\}} a^{ij} D_i (u-t\zeta) \,D_j (u-t\zeta)
		\\&\leq \int_{B_1 \cap \{u > M_-\}} a^{ij} D_i (u-t\zeta) \,D_j (u-t\zeta)
\end{align*}
for all $t > 0$.  Thus by differentiating $\int_{B^+_1 \cap \{u > M_-\}} a^{ij} D_i (u-t\zeta) \,D_j (u-t\zeta)$ at $t = 0^+$,
\begin{equation*}
	\int_{B^+_1 \cap \{u > M_-\}} a^{ij} D_i u \,D_j \zeta \leq 0
\end{equation*}
for all nonnegative functions $\zeta \in W^{1,2}_0(B^+_1)$.  Therefore \eqref{regularity2 eqn1} holds true.  By \eqref{regularity2 eqn1} and~\cite[Theorem~8.17]{GT},
\begin{equation} \label{regularity2 eqn2}
	\sup_{B^+_{1/2}} (u - M_-) \leq C \|(u - M_-)^+\|_{L^2(B^+_1)}
\end{equation}
for some constant $C = C(n,\lambda,\Lambda) \in (0,\infty)$.

By instead comparing $u$ to
\begin{equation*}
	u_t %= M_+ + (u - M_+ - t \zeta)^+ - (u - M_+)^-
	= \begin{cases}
		u + t\zeta &\text{if } u \leq M_+ - t\zeta \\
		M_+ &\text{if } M_+ - t\zeta < u \leq M_+ \\
		u &\text{if } u \geq M_+
	\end{cases}
\end{equation*}
where $t > 0$ and $\zeta \in W^{1,2}_0(B^+_1)$, we can show that
\begin{equation*}
	D_i (a^{ij} D_j (u - M_+)^-) \geq 0 \text{ weakly in } B_1 .
\end{equation*}
Then by applying~\cite[Theorem~8.17]{GT} to $(u - M_+)^-$
\begin{equation} \label{regularity2 eqn3}
	\sup_{B^+_{1/2}} (M_+ - u) \leq C \|(u - M_+)^-\|_{L^2(B^+_1)}
\end{equation}
for some constant $C = C(n,\lambda,\Lambda) \in (0,\infty)$.  It follows from \eqref{regularity2 eqn2} and \eqref{regularity2 eqn3} that \eqref{regularity2 concl} holds true.
\end{proof}

Our next step consists in proving an initial H\"older modulus of continuity for the solution.
\begin{lemma} \label{regularity lemma3}
Let $a^{ij} \in C^0(B^+_1(0) \cup B'_1(0))$ such that $a^{ij} = a^{ji}$ on $B^+_1(0) \cup B'_1(0)$, \eqref{ellipticity hyp} holds true for some constants $0 < \lambda \leq \Lambda < \infty$, and \eqref{regularity assumption} holds true on $B'_1(0)$.  Let $1 < p < \infty$ be a constant and $k_+,k_-,h \in L^{\infty}(B'_1(0))$ with $k_+,k_- \geq 0$ on $B'_1(0)$.  If $u \in W^{1,2}(B^+_1(0))$ is a weak solution to \eqref{main ptop}, then $u \in C^{0,1/2}(B_{1/8}(0))$ with
\begin{equation} \label{regularity3 concl}
	[u]_{1/2,B^+_{1/8}(0)}^2 \leq C \,\|u\|_{L^2(B^+_1(0))}^2
		+ C \,\|k_+ + k_-\|_{L^{\infty}(B'_1(0))} \left( \|u\|_{L^2(B^+_1(0))}^p + \|h\|_{L^{\infty}(B'_1(0))}^p \right)
\end{equation}
for some constant $C \in (0,\infty)$ depending on $n$, $p$, $\lambda$, $\Lambda$, and the modulus of continuity of $a^{ij}$.
\end{lemma}
\begin{proof}
Let $y \in B'_{1/4}(0)$, $0 < \rho \leq 1/4$, and let $v \in W^{1,2}(B^+_{\rho}(0))$ be the weak solution to the constant coefficient problem
\begin{gather}
	D_i (a^{ij}(y) \,D_j v) = 0 \text{ in } B^+_{\rho}(y), \nonumber \\
	\label{regularity3 eqn1} D_n v = 0 \text{ on } B'_{\rho}(y), \\
	v = u \text{ on } (\partial B_{\rho}(y))^+ . \nonumber
\end{gather}
Extending $v$ by even reflection $v(x',x_n) = v(x',-x_n)$, $v$ is a solution to the constant coefficient elliptic equation $D_i (a^{ij}(y) \,D_j v) = 0$ in $B_{\rho}(y)$ with $v = u$ on $\partial B_{\rho}(y)$.  Thus by the maximum principle, $\sup_{B^+_{\rho}(y)} |v| \leq \sup_{\partial B^+_{\rho}(y)} |u|$.  Using the $J$-minimality of $u$ and Lemma~\ref{regularity lemma2},
\begin{align} \label{regularity3 eqn2}
	\int_{B^+_{\rho}(y)} (a^{ij} D_i u \,D_j u - a^{ij} D_i v \,D_j v) \nonumber
		\leq&\, \frac{2}{p} \int_{B'_{\rho}(y)} (k_+ ((v-h)^+)^p - k_+ ((u-h)^+)^p) \nonumber
			\\&+ \frac{2}{p} \int_{B'_{\rho}(y)} (k_- ((v-h)^-)^p - k_- ((u-h)^-)^p) \nonumber
		\\ \leq& C K \rho^{n-1} ,
\end{align}
where $C = C(n,p) \in (0,\infty)$ is a constant and
\begin{equation*}
	K = \|k_+ + k_-\|_{L^{\infty}(B'_1(0))} \,(\|u\|_{L^2(B^+_1(0))}^p + \|h\|_{L^{\infty}(B'_1(0))}^p) .
\end{equation*}
Since $v$ is the solution to \eqref{regularity3 eqn1},
\begin{equation*}
	\int_{B^+_{\rho}(y)} a^{ij}(y) \,D_i (u - v) \,D_j v = 0
\end{equation*}
and thus
\begin{align} \label{regularity3 eqn3}
	\int_{B^+_{\rho}(y)} a^{ij}(y) \,D_i (u - v) \,D_j (u - v) = \int_{B^+_{\rho}(y)} (a^{ij}(y) \,D_i u \,D_j u - a^{ij}(y) \,D_i v \,D_j v)
\end{align}
Hence by \eqref{regularity3 eqn3} and \eqref{regularity3 eqn2},
\begin{align*}
	\int_{B^+_{\rho}(y)} |Du - Dv|^2
	&\leq \frac{1}{\lambda} \int_{B^+_{\rho}(y)} a^{ij}(y) \,D_i (u - v) \,D_j (u - v)
	\\&= \frac{1}{\lambda} \int_{B^+_{\rho}(y)} (a^{ij}(y) \,D_i u \,D_j u - a^{ij}(y) \,D_i v \,D_j v)
	\\&\leq -\frac{1}{\lambda} \int_{B^+_{\rho}(y)} (a^{ij} - a^{ij}(y)) \,(D_i u \,D_j u - D_i v \,D_j v) + C K\rho^{n-1}
%	\\&\leq C \rho \int_{B^+_{\rho}(y)} (|Du|^2 + |Dv|^2) + C K \rho^{n-1} ,
\end{align*}
for some constant $C = C(n,p,\lambda) \in (0,\infty)$.  Thus there exists $\rho_0 \in (0,1/4]$ depending on $\lambda$ and the modulus of continuity of $a^{ij}$ such that if $0 < \rho \leq \rho_0$ then
\begin{equation*}
	\int_{B^+_{\rho}(y)} |Du - Dv|^2 \leq \frac{1}{4} \int_{B^+_{\rho}(y)} (|Du|^2 + |Dv|^2) + C K \rho^{n-1}
\end{equation*}
for some constant $C = C(n,p,\lambda) \in (0,\infty)$.  Thus, by the triangle inequality we have, if $0 < \rho \leq \rho_0$
\begin{align}
	\label{regularity3 eqn5} \int_{B^+_{\rho}(y)} |Du - Dv|^2 &\leq \frac{3}{2} \int_{B^+_{\rho}(y)} |Dv|^2 + 2 C K \rho^{n-1} , \\
	\label{regularity3 eqn6} \int_{B^+_{\rho}(y)} |Du - Dv|^2 &\leq \frac{3}{2} \int_{B^+_{\rho}(y)} |Du|^2 + 2 C K \rho^{n-1}
\end{align}
for some constant $C = C(n,p,\lambda) \in (0,\infty)$.

Let $0 < \sigma < \rho \leq \rho_0$.  Using the triangle inequality and \eqref{regularity3 eqn5},
\begin{align*}
	\sigma^{1-n} \int_{B^+_{\sigma}(y)} |Du|^2
		&\leq 2 \sigma^{1-n} \int_{B^+_{\sigma}(y)} |Du - Dv|^2 + 2 \sigma^{1-n} \int_{B^+_{\sigma}(y)} |Dv|^2
		\\&\leq 5 \sigma^{1-n} \int_{B^+_{\sigma}(y)} |Dv|^2 + 4 C K
\end{align*}
for some constant $C = C(n,p,\lambda) \in (0,\infty)$.  By~\cite[Equation (5.82)]{L13},
\begin{equation*}
	\sigma^{-n} \int_{B^+_{\sigma}(y)} |Dv|^2 \leq C \rho^{-n} \int_{B^+_{\rho}(y)} |Dv|^2
\end{equation*}
for some constant $C = C(n,\lambda,\Lambda) \in (0,\infty)$ and thus
\begin{equation*}
	\sigma^{1-n} \int_{B^+_{\sigma}(y)} |Du|^2
		\leq C \left(\frac{\sigma}{\rho}\right) \rho^{1-n} \int_{B^+_{\rho}(y)} |Dv|^2 + C K .
\end{equation*}
for some constant $C = C(n,p,\lambda,\Lambda) \in (0,\infty)$.  By the triangle inequality and \eqref{regularity3 eqn6},
\begin{align*}
	\sigma^{1-n} \int_{B^+_{\sigma}(y)} |Du|^2
		&\leq 2C \left(\frac{\sigma}{\rho}\right) \rho^{1-n} \int_{B^+_{\rho}(y)} |Du - Dv|^2
			+ 2C \left(\frac{\sigma}{\rho}\right) \rho^{1-n} \int_{B^+_{\rho}(y)} |Du|^2 + C K \nonumber
		\\&\leq C \left(\frac{\sigma}{\rho}\right) \rho^{1-n} \int_{B^+_{\rho}(y)} |Du|^2 + C K ,
\end{align*}
where $C = C(n,p,\lambda,\Lambda) \in (0,\infty)$ are constants.  Hence there exists $\theta = \theta(n,p,\lambda,\Lambda) \in (0,1)$ such that, setting $\sigma = \theta \rho$,
\begin{align} \label{regularity3 eqn7}
	(\theta\rho)^{1-n} \int_{B^+_{\theta\rho}(y)} |Du|^2 \leq \frac{1}{2} \,\rho^{1-n} \int_{B^+_{\rho}(y)} |Du|^2 + C K
\end{align}
for all $0 < \rho \leq \rho_0$ and some constant $C = C(n,p,\lambda,\Lambda) \in (0,\infty)$.  Iteratively applying \eqref{regularity3 eqn7} with $\rho = \theta^{j-1} \rho_0$ for $j = 1,2,3,\ldots,m$,
\begin{align} \label{regularity3 eqn8}
	(\theta^m \rho_0)^{1-n} \int_{B^+_{\theta^m \rho_0}(y)} |Du|^2
		&\leq \frac{1}{2^m} \,\rho_0^{1-n} \int_{B^+_{\rho_0}(y)} |Du|^2
			+ C K \sum_{j=0}^{m-1} \frac{1}{2^j} \nonumber
		\\&\leq \frac{1}{2^m} \,\rho_0^{1-n} \int_{B^+_{\rho_0}(y)} |Du|^2 + 2 C K
\end{align}
for each integer $m = 1,2,3,\ldots$.  Take $0 < \rho \leq \rho_0$ and choose $m$ to be the nonnegative integer such that $\theta^{m+1} \rho_0 < \rho \leq \theta^m \rho_0$.  Then \eqref{regularity3 eqn8} gives us
\begin{equation*}
	\rho^{1-n} \int_{B^+_{\rho}(y)} |Du|^2 \leq C \rho_0^{1-n} \int_{B^+_{\rho_0}(y)} |Du|^2 + C K
\end{equation*}
for some constant $C = C(n,p,\lambda,\Lambda) \in (0,\infty)$.  Therefore, using Lemma~\ref{regularity lemma1},
\begin{equation} \label{regularity3 eqn9}
	\int_{B^+_{\rho}(y)} |Du|^2 \leq C \rho^{n-1} \left( \|u\|_{L^2(B^+_1(0))}^2 + K \right)
\end{equation}
for all $y \in B'_{1/4}(0)$ and $0 < \rho \leq 1/4$, and for some constant $C \in (0,\infty)$ depending only on $n$, $p$, $\lambda$, $\Lambda$, and the modulus of continuity of $a^{ij}$.  By \eqref{regularity3 eqn9},
\begin{equation} \label{regularity3 eqn10}
	\int_{B^+_{\rho}(y)} |Du|^2 \leq \int_{B^+_{2\rho}(y',0)} |Du|^2 \leq C \rho^{n-1} \left( \|u\|_{L^2(B^+_1(0))}^2 + K \right)
\end{equation}
for all $y = (y',y_n) \in B_{1/8}(0)$ and $|y_n| \leq \rho \leq 1/8$, where $C \in (0,\infty)$ is a constant depending only on $n$, $p$, $\lambda$, $\Lambda$, and the modulus of continuity of $a^{ij}$.  By the argument leading to \eqref{regularity3 eqn9} with obvious modifications,
\begin{equation} \label{regularity3 eqn11}
	\int_{B_{\rho}(y)} |Du|^2 \leq C \left(\frac{\rho}{|y_n|}\right)^{n-1} \|u\|_{L^2(B_{|y_n|}(y))}^2
\end{equation}
for all $y = (y',y_n) \in B_{1/8}(0)$ and $0 < \rho \leq |y_n|$ and some constant $C \in (0,\infty)$ depending only on $n$, $\lambda$, $\Lambda$, and the modulus of continuity of $a^{ij}$.  By \eqref{regularity3 eqn10} and \eqref{regularity3 eqn11}, \eqref{regularity3 eqn9} holds true for all $y = (y',y_n) \in B_{1/8}(0)$ and $0 < \rho \leq |y_n|$.  In other words, \eqref{regularity3 eqn9} holds true for all $y \in B^+_{1/8}(0)$ and $0 < \rho \leq 1/8$.  Hence we may apply Morrey's Dirichlet Growth Theorem using \eqref{regularity3 eqn9} to conclude that $u \in C^{0,1/2}(B^+_{1/8}(0))$ and \eqref{regularity3 concl} holds true.
\end{proof}

We are now ready to prove our main result in this section.
\begin{theorem} \label{regularity thm}
Let $\kappa \geq 1$ be an integer and $\alpha \in (0,1)$.  Let $\Omega \subset \mathbb{R}^n$ be a bounded open subset and $\Gamma$ be a relatively open $C^{\kappa,\alpha}$-portion of $\partial \Omega$.  Let $a^{ij} \in C^{\kappa-1,\alpha}(\Omega \cup \Gamma)$ such that $a^{ij} = a^{ji}$ on $\Omega \cup \Gamma$ and \eqref{ellipticity hyp} holds true for some constants $0 < \lambda \leq \Lambda < \infty$.  Let $1 < p < \infty$ be a constant, $k_+,k_- \in C^{\kappa-1,\alpha}(\Gamma)$ with $k_+,k_- \geq 0$ on $\Gamma$, and $h \in C^{\kappa-1,1}(\Gamma)$.   Let $u \in W^{1,2}(\Omega)$ be a weak solution to \eqref{main ptop}.
\begin{enumerate}
	\item[(a)]  If $\kappa+\alpha \leq p$, then $u \in C^{\kappa,\alpha}(\Omega \cup \Gamma)$.
	
	\item[(b)]  Let $x_0 \in \Gamma$ and $0 < \delta < \op{dist}(x_0,\partial \Omega \setminus \Gamma)$.  Suppose that either
	\begin{enumerate}
		\item[(i)] $u \neq h$ on $\Gamma \cap B_{\delta}(x_0)$ or
		\item[(ii)] $p$ is an even integer and $k_+ = k_-$ on $\Gamma \cap B_{\delta}(x_0)$ or
		\item[(iii)] $p$ is an odd integer and $k_+ = k_- = 0$ on $\Gamma \cap B_{\delta}(x_0)$.
	\end{enumerate}
	Then $u \in C^{\kappa,\alpha}((\Omega \cup \Gamma) \cap B_{\delta}(x_0))$.
\end{enumerate}
\end{theorem}

\begin{remark} {\rm
In the special case that $a^{ij} \in C^{\infty}(\Omega \cup \Gamma)$ and $k_+,k_-,h \in C^{\infty}(\Gamma)$, Theorem~\ref{regularity thm} gives us that if $u \in W^{1,2}(\Omega)$ is a weak solution to \eqref{main ptop}, then $u \in C^{\lfloor p \rfloor, p - \lfloor p \rfloor}(\Omega \cup \Gamma)$ if $p$ is not an integer and $u \in C^{p-1,\alpha}(\Omega \cup \Gamma)$ for all $\alpha \in (0,1)$ if $p$ is an integer.  Moreover, given $x_0 \in \Gamma$ and $0 < \delta < \op{dist}(x_0,\partial \Omega \setminus \Gamma)$, if either $u \neq h$ in $\Gamma \cap B_{\delta}(x_0)$, or $p$ is an even integer and $k_+ = k_-$ on $\Gamma \cap B_{\delta}(x_0)$, or $p$ is an odd integer and $k_+ = k_- = 0$ on $\Gamma \cap B_{\delta}(x_0)$, then $u \in C^{\infty}((\Omega \cup \Gamma) \cap B_{\delta}(x_0))$.  This is consistent with what was found in~\cite{DJ} in the special case that $\Omega = B^+_1(0)$, $\Gamma = B'_1(0)$, $a^{ij} = \delta_{ij}$ on $B^+_1(0) \cup B'_1(0)$, $h = 0$ on $B'_1(0)$, and $k_+,k_-$ are constants on $B'_1(0)$.
} \end{remark}

\begin{proof}[Proof of Theorem~\ref{regularity thm}]
By elliptic regularity~\cite[Corollary~8.36 and Theorem~6.17]{GT}, $u \in C^{\kappa,\alpha}(\Omega)$.  Thus to show (a), it suffices to show that for each $x_0 \in \Gamma$ and some $0 < \delta < \op{dist}(x_0,\partial \Omega \setminus \Gamma)$ (depending on $x_0$) that $u \in C^{\kappa,\alpha}((\Omega \cup \Gamma) \cap B_{\delta}(x_0))$.  After a $C^{\kappa,\alpha}$-change of coordinates we may assume that $x_0 = 0$, $\delta = 1$, $\Omega \cap B_{\delta}(x_0) = B^+_1(0)$, and $\Gamma \cap B_{\delta}(x_0) = B'_1(0)$.  Moreover, by~\cite[Theorem~B.1]{GSVG14} after a further $C^{\kappa,\alpha}$-change of coordinates we may assume that \eqref{regularity assumption} holds true on $B'_1(0)$.  (Note that if $a^{ij}$ is in $C^{\kappa-1,\alpha}(\Omega\cup \Gamma)$ rather than in $C^{1,1}(\Omega\cup \Gamma)$ then the change of coordinates constructed in~\cite[Theorem~B.1]{GSVG14} is in fact $C^{\kappa,\alpha}$.)

Let's suppose that $\kappa+\alpha \leq p$ and show that $u \in C^{\kappa,\alpha}(B^+_1 \cup B'_1)$.  By Lemma~\ref{regularity lemma3}, $u \in C^{0,1/2}(B^+_1 \cup B'_1)$.  Notice that since $\kappa+\alpha \leq p$, $z \mapsto (z^+)^{p-1}$ and $z \mapsto (z^-)^{p-1}$ are $C^{\kappa-1,\alpha}$-functions of $z \in \mathbb{R}$.  All this together with $k_+,k_- \in C^{\kappa-1,\alpha}(B'_1)$ and $h \in C^{\kappa-1,1}(B'_1)$, gives us
\begin{equation*}
	k_+ ((u-h)^+)^{p-1} - k_- ((u-h)^-)^{p-1} \in \begin{cases}
		C^{0,\alpha/2}(B'_1) &\text{if } \kappa = 1 \\
		C^{0,1/2}(B'_1) &\text{if } \kappa \geq 2 .
	\end{cases}
\end{equation*}
By the regularity theory for weak solutions to oblique derivative problems (see e.g.~\cite[Proposition 5.53]{L13}),
\begin{equation} \label{regularity eqn1}
	u \in \begin{cases}
		C^{1,\alpha/2}(B^+_1 \cup B'_1) &\text{if } \kappa = 1 \\
		C^{1,1/2}(B^+_1 \cup B'_1) &\text{if } \kappa \geq 2 .
	\end{cases}
\end{equation}
But then
\begin{equation*}
	k_+ ((u-h)^+)^{p-1} - k_- ((u-h)^-)^{p-1} \in \begin{cases}
		C^{0,\alpha}(B'_1) &\text{if } \kappa = 1 \\
		C^{1,\alpha/2}(B'_1) &\text{if } \kappa = 2 \\
		C^{1,1/2}(B'_1) &\text{if } \kappa \geq 3
	\end{cases}
\end{equation*}
and so applying the regularity theory for solutions to oblique derivative problems again,
\begin{equation} \label{regularity eqn2}
	u \in \begin{cases}
		C^{1,\alpha}(B^+_1 \cup B'_1) &\text{if } \kappa = 1 \\
		C^{2,\alpha/2}(B^+_1 \cup B'_1) &\text{if } \kappa = 2 \\
		C^{2,1/2}(B^+_1 \cup B'_1) &\text{if } \kappa \geq 3
	\end{cases}
\end{equation}
arriving at the desired conclusion $u \in C^{1,\alpha}(B^+_1 \cup B'_1)$ if $\kappa = 1$.  If instead $\kappa \geq 2$, for each positive integer $\ell \leq \kappa+1$ by inductively applying the regularity theory solutions to oblique derivative problems a total of $\ell$ times we obtain
\begin{equation} \label{regularity eqn3}
	u \in \begin{cases}
		C^{\ell-1,\alpha}(B^+_1 \cup B'_1) &\text{if } \kappa = \ell-1 \\
		C^{\ell,\alpha/2}(B^+_1 \cup B'_1) &\text{if } \kappa = \ell \\
		C^{\ell,1/2}(B^+_1 \cup B'_1) &\text{if } \kappa \geq \ell+1.
	\end{cases}
\end{equation}
Indeed, by \eqref{regularity eqn1} and \eqref{regularity eqn2}, \eqref{regularity eqn3} holds true if $\ell = 1,2$.  If \eqref{regularity eqn3} holds true for some integer $\ell \leq \kappa$, then \eqref{regularity eqn3} implies that
\begin{equation*}
	k_+ ((u-h)^+)^{p-1} - k_- ((u-h)^-)^{p-1} \in \begin{cases}
		C^{\ell-1,\alpha}(B'_1) &\text{if } \kappa = \ell \\
		C^{\ell,\alpha/2}(B'_1) &\text{if } \kappa = \ell+1 \\
		C^{\ell,1/2}(B'_1) &\text{if } \kappa \geq \ell+2
	\end{cases}
\end{equation*}
and so applying the regularity theory one more time gives us \eqref{regularity eqn3} with $\ell+1$ in place of $\ell$.  Taking $\ell = \kappa+1$ in \eqref{regularity eqn3} gives us the desired conclusion that $u \in C^{\kappa,\alpha}(B^+_1 \cup B'_1)$.

(b) follows by a similar bootstrap argument, noting for (i) that $z \mapsto (z^+)^{p-1}$ and $z \mapsto (z^-)^{p-1}$ are smooth functions of $z \in \mathbb{R} \setminus \{0\}$.  Also, assuming (ii) or (iii),
\begin{equation*}
	k_+(x) \,(z^+)^{p-1} - k_-(x) \,(z^-)^{p-1} = k_+(x) \,z^{p-1}
\end{equation*}
for each $x \in \Gamma \cap B_{\delta}(x_0)$ and $z \in \mathbb{R}$ and $z \mapsto z^{p-1}$ is a smooth function of $z \in \mathbb{R}$.
\end{proof}

In Theorem~\ref{irregularity thm} below, we show that the regularity in Theorem~\ref{regularity thm} is optimal in the sense that if $p$ is an integer, $u$ is not locally $C^{p-1,1}$ near free boundary points $x_0 \in \Gamma$ at which $u(x_0) = h(x_0)$, $\nabla^{\Gamma} u(x_0) \neq \nabla^{\Gamma} h(x_0)$, and either $p$ is even and $k_+(x_0) \neq k_-(x_0)$ or $p$ is odd and $k_+(x_0) \neq 0$ or $k_-(x_0) \neq 0$.  This was previously shown in~\cite[Theorem~2]{DJ} in the special case $p = 2$, $a^{ij} = \delta_{ij}$, $k_+$ and $k_-$ are constants, and $h = 0$ on $\Gamma$.  The key obstruction is that after an orthogonal change of coordinates, we expect $u$ to be asymptotic to a degree $(p-1)$ polynomial plus a constant multiple of $\op{Re}( (x_1-ix_n)^p \log(x_n+ix_1))$.

\begin{theorem} \label{irregularity thm}
Let $\Omega \subset \mathbb{R}^n$ be a bounded open subset and $\Gamma$ be a relatively open $C^p$-portion of $\partial \Omega$.  Let $a^{ij} \in C^{p-1}(\Omega \cup \Gamma)$ such that $a^{ij} = a^{ji}$ on $\Omega$ and \eqref{ellipticity hyp} holds true for some constants $0 < \lambda \leq \Lambda < \infty$.  Let $p \geq 2$ be an integer, $k_+,k_- \in C^0(\Gamma)$ such that $k_+,k_- \geq 0$ on $\Gamma$, and $h \in C^1(\Gamma)$.  Let $u \in W^{1,2}(\Omega)$ be a weak solution to \eqref{main ptop}.  Moreover, let $x_0 \in \Gamma$ be such that $u(x_0) = h(x_0)$ and $\nabla^{\Gamma} u(x_0) \neq \nabla^{\Gamma} h(x_0)$, where $\nabla^{\Gamma}$ denotes tangential gradient of $\Gamma$ In addition, assume that either
\begin{enumerate}
	\item[(i)]  $p$ is even and $k_+(x_0) \neq k_-(x_0)$, or
	\item[(ii)] $p$ is odd and $k_+(x_0) \neq 0$ or $k_-(x_0) \neq 0$.
\end{enumerate}
Then there is no $\delta > 0$ such that $u \in C^{p-1,1}((\Omega \cap \Gamma) \cap B_{\delta}(x_0))$.
\end{theorem}

\begin{proof}
After a change of coordinates assume that $x_0 = 0$, $\delta = 1$, $\Omega = B^+_1(0)$, $\Gamma = B'_1(0)$,  and \eqref{identity_origin} and \eqref{regularity assumption} hold true.  Suppose to the contrary that $u \in W^{1,2}(B^+_1(0))$ is a weak solution to \eqref{main ptop} such that either $p$ is even and $k_+(0) = k_-(0)$, or $p$ is odd and $k_+(0) = k_-(0) = 0$, $u(0) = h(0)$, $\nabla^{\Gamma} u(0) \neq \nabla^{\Gamma} h(0)$, and $u \in C^{p-1,1}(B^+_1(0) \cup B'_1(0))$ with $\|D^p u\|_{L^{\infty}(B^+_1(0))} < \infty$.  Without loss of generality assume that
\begin{equation} \label{irregularity eqn1}
	D_1 (u-h)(0) = b, \quad D_i (u-h)(0) = 0 \text{ for all } i = 2,3,\ldots,n-1
\end{equation}
for some constant $b > 0$.  Notice that by \eqref{main ptop} and $u(0) = h(0)$, $D_n u(0) = 0$.  Let $P(x) = \sum_{|\alpha| \leq p-1} \frac{1}{\alpha!} \,D^{\alpha} u(0) \,x^{\alpha}$ be the degree $(p-1)$ Taylor polynomial of $u$ at the origin.  For each $\rho > 0$ define $u_{\rho} \in C^{p-1,1}(B^+_{1/\rho}(0) \cup B'_{1/\rho}(0))$ by
\begin{equation*}
	u_{\rho}(x) = \frac{u(\rho x) - P(\rho x)}{\rho^{\,p}} .
\end{equation*}
By \eqref{main ptop} and $u(0) = h(0)$, $|D_n u| \leq C |x|^{p-1}$ for all $x \in B'_1(0)$ and some constant $C \in (0,\infty)$.  Since $P$ is the degree $(p-1)$ Taylor polynomial of $u$ at the origin, $D_n P = 0$ on $\mathbb{R}^{n-1} \times \{0\}$.  Hence %by \eqref{main ptop}
\begin{align}\label{irregularity eqn2}
	&D_i (a^{ij}(\rho x) \,D_j u_{\rho}) = -\frac{(D_i (a^{ij} \,D_j P))(\rho x)}{\rho^{\,p-2}} \text{ in } B^+_{1/\rho}(0), \nonumber \\
	&a^{nn}(\rho x) D_n u_{\rho} = k_+(\rho x) \bigg( \bigg( \frac{P(\rho x) - h(\rho x)}{\rho} + \rho^{p-1} u_{\rho}(x) \bigg)^+ \bigg)^{p-1} \nonumber
		\\&\hspace{19mm} - k_-(\rho x) \bigg( \bigg( \frac{P(\rho x) - h(\rho x)}{\rho} + \rho^{p-1} u_{\rho}(x) \bigg)^- \bigg)^{p-1} \text{ on } B'_{1/\rho}(0) .
\end{align}
Since $u \in C^{p-1,1}(B^+_1(0) \cup B'_1(0))$ with $\|D^p u\|_{L^{\infty}(B^+_1(0))} < \infty$,
\begin{equation*}
	\|u_{\rho}\|_{C^{p-1,1}(B^+_{1/\rho}(0))} \leq C(n,p) \,\|D^p u\|_{L^{\infty}(B^+_1(0))} .
\end{equation*}
Thus, by the Arzel\`a-Ascoli theorem, there exists a sequence $\rho_{\ell} \rightarrow 0^+$ and function $v \in C^{p-1,1}(\mathbb{R}^n_+)$ such that $u_{\rho_{\ell}} \rightarrow v$ in the $C^{p-1}$-topology on compact subsets of $\overline{\mathbb{R}^n_+}$.  By continuity, $a^{ij}(\rho x) \rightarrow a^{ij}(0) = \delta_{ij}$ uniformly on compact subsets of $\overline{\mathbb{R}^n_+}$ as $\rho \rightarrow 0^+$ and $k_{\pm}(rx) \rightarrow \kappa_{\pm}$ uniformly on compact subsets of $\mathbb{R}^{n-1} \times \{0\}$ as $\rho \rightarrow 0^+$, where $\kappa_{\pm} = k_{\pm}(0)$.  Notice that $D_i (a^{ij} \,D_j P) \in C^{p-2}(B^+_1(0) \cup B'_1(0))$.  Keeping in mind that $u$ satisfies \eqref{main ptop} and $P$ is the degree $(p-1)$ Taylor polynomial of $u$ at the origin, the degree $(p-3)$ Taylor polynomial of $D_i (a^{ij} \,D_j P)$ is zero.  Thus,
\begin{equation*}
	\frac{(D_i (a^{ij} \,D_j P))(\rho x)}{\rho^{\,p-2}} \rightarrow Q(x)
\end{equation*}
uniformly on compact subsets of $\mathbb{R}^n$ as $r \rightarrow 0^+$, where $Q(x) = \sum_{|\alpha| = p-2} \sum_{i,j=1}^n \frac{1}{\alpha!} \,D^{\alpha} D_i (a^{ij} \,D_j P)(0) \,x^{\alpha}$ is a homogeneous degree $(p-2)$ polynomial.  Since $P$ is the degree $(p-1)$ Taylor polynomial of $u$ at the origin, $u(0) = h(0)$, and \eqref{irregularity eqn1} holds true,
\begin{equation*}
	\frac{P(\rho x) - h(\rho x)}{\rho} \rightarrow (\nabla^{\Gamma} u(0) - \nabla^{\Gamma} h(0)) \cdot x = b x_1
\end{equation*}
uniformly on compact subsets of $\mathbb{R}^{n-1} \times \{0\}$ as $\rho \rightarrow 0^+$.  Therefore, by letting $\rho = \rho_{\ell} \rightarrow 0^+$ in \eqref{irregularity eqn2}, $v \in C^{p-1,1}(\mathbb{R}^n_+)$ is a solution to
\begin{gather}
	\Delta v = Q \text{ in } \mathbb{R}^n_+, \nonumber \\
	\label{irregularity eqn3} D_n v = \kappa_+ \,(b x_1^+)^{p-1} - \kappa_- \,(b x_1^-)^{p-1} \text{ on } \mathbb{R}^{n-1} \times \{0\} ,
\end{gather}
where we recall that $Q$ is a homogeneous  polynomial of degree $p-2$ on $\mathbb{R}^n$ and $b > 0$ is a constant.  Now it suffices to show that:

\noindent\emph{Claim.  Let $p \geq 0$ be a positive integer and let $b > 0$ and $\kappa_+,\kappa_- \geq 0$ be constants such that either $p$ is even and $\kappa_+ \neq \kappa_-$ or $p$ is odd and $\kappa_+ \neq 0$ or $\kappa_- \neq 0$.  Let $Q$ be a homogeneous degree $p-2$ polynomial.  Then there are no weak solutions $v \in C^{p-1,1}(\overline{\mathbb{R}^n_+})$ to \eqref{irregularity eqn3}.}
%For each positive integer $p \geq 2$, constants $b > 0$ and $\kappa_+,\kappa_- \geq 0$ such that either $p$ is even and $\kappa_+ \neq \kappa_-$ or $p$ is odd and $\kappa_+ \neq 0$ or $\kappa_- \neq 0$, and homogeneous polynomial $Q$ degree $p-2$, there are no weak solutions $v \in C^{p-1,1}(\mathbb{R}^n_+)$ to \eqref{irregularity eqn3}.}

To see this, first we make the following simplifying assumptions.  Notice that we can find a homogeneous  polynomial $R$ of degree $p$ such that
\begin{gather*}
	\Delta R = Q \text{ in } \mathbb{R}^n_+, \nonumber \\
	D_n R = -\kappa_- \,(-b x_1)^{p-1} \text{ on } \mathbb{R}^{n-1} \times \{0\} .
\end{gather*}
To see this, suppose that
\begin{equation*}
	R(x',x_n) = \sum_{j=0}^p \alpha_j(x') \,x_n^j, \quad Q(x',x_n) = \sum_{j=0}^{p-2} \beta_j(x') \,x_n^j
\end{equation*}
for all $x' \in \mathbb{R}^{n-1}$ and $x_n \in \mathbb{R}$, where $\alpha_j(x')$ are homogeneous degree $(p-j)$ polynomials on $\mathbb{R}^{n-1}$ and $\beta_j(x')$ are homogeneous degree $(p-2-j)$ polynomials on $\mathbb{R}^{n-1}$.  Let's solve for $\alpha_j(x')$ for $j = 0,1,2,\ldots,p$.  We compute that
\begin{equation*}
	\Delta R(x',x_n) = \sum_{j=0}^{p-2} ((j+1)(j+2) \alpha_{j+2}(x') + \Delta \alpha_j(x')) \,x_n^j .
\end{equation*}
Setting $\Delta R = Q$ gives us that $\alpha_j(x')$ are defined by the recurrence relation
\begin{equation*}
	\alpha_1(x') = -\kappa_- (-b x_1)^{p-1}, \quad \alpha_{j+2}(x') = \frac{\beta_j(x') - \Delta \alpha_j(x')}{(j+1)(j+2)} \text{ for } j = 0,1,2,\ldots,p-2 ,
\end{equation*}
with $\alpha_0(x')$ to be chosen freely (as $R$ is uniquely defined up at adding a homogeneous degree $p$ harmonic polynomial which is even in $x_n$).  By replacing $v$ with $v - R$, we may assume that $Q = 0$, $\kappa_+ \neq 0$, and $\kappa_- = 0$ and show that there are no weak solutions $v \in C^{p-1,1}(\overline{\mathbb{R}^n_+})$ to
\begin{gather}
	\Delta v = 0 \text{ in } \mathbb{R}^n_+, \nonumber \\
	\label{irregularity eqn4} D_n v = \kappa_+ (x_1^+)^{p-1} \text{ on } \mathbb{R}^{n-1} \times \{0\} .
\end{gather}
%Notice that
%\begin{align*}
%	\Delta (v-R) &= Q-Q = 0 \text{ in } \mathbb{R}^n_+, \nonumber \\
%	D_n (v-R) &= \kappa_+ (b x_1^+)^{p-1} - \kappa_- (b x_1^-)^{p-1} + \kappa_- (-b x_1)^{p-1}
%		\\&= (\kappa_+ + (-1)^{p-1} \kappa_-) \,(b x_1^+)^{p-1} \text{ on } \mathbb{R}^{n-1} \times \{0\} .
%\end{align*}
%Thus by relabeling $\kappa_+ + (-1)^{p-1} \kappa_-$ and $v-R$ as $\kappa_+$ and $v$, we may assume that $Q = 0$, $\kappa_+ \neq 0$, and $\kappa_- = 0$.  By scaling, we may assume that $\kappa_+ = b = 1$.  Therefore, it suffices to show that there are no weak solutions $v \in C^{p-1,1}(\mathbb{R}^n_+)$ to
%\begin{gather}
%	\Delta v = 0 \text{ in } \mathbb{R}^n_+, \nonumber \\
%	\label{irregularity eqn4} D_n v = (x_1^+)^{p-1} \text{ on } \mathbb{R}^{n-1} \times \{0\} .
%\end{gather}

Consider the function $w : \overline{\mathbb{R}^n_+} \rightarrow \mathbb{R}$ defined by
\begin{equation} \label{irregularity eqn5}
	w(x) = \op{Re} \left( (-i)^p \kappa_+ (x_n+ix_1)^p \left( \frac{1}{\pi p} \log(x_n+ix_1) - \frac{1}{\pi p^2} + \frac{i}{2p} \right) \right) ,
\end{equation}
where $i = \sqrt{-1}$ and we use the branch of $\log(x_n+ix_1)$ on $\{ x_n+ix_1 : x_n \geq 0 \}$ for which $\log(x_n + i 0) = \log|x_n|$.  One readily verifies that $w \in C^{p-1}(\overline{\mathbb{R}^n_+})$ but $w \not\in C^{p-1,1}(\overline{\mathbb{R}^n_+})$.  We claim that \eqref{irregularity eqn4} holds true with $w$ in place of $v$.  Clearly $w \in W^{2,2}(\mathbb{R}^n_+)$ and $\Delta w = 0$ in $\mathbb{R}^n_+$.
%Clearly $\Delta w = 0$ in $\overline{\mathbb{R}^n_+} \setminus \{x_1 = x_n = 0\}$.  Hence by a standard cutoff function argument, $w$ is a weak solution to $\Delta w = 0$ in $\mathbb{R}^n_+$.
By differentiating \eqref{irregularity eqn5},
\begin{equation} \label{irregularity eqn6}
	D_n w(x',0) = \op{Re} \left( \kappa_+ \,x_1^{p-1} \left( \frac{1}{i \pi} \log(ix_1) + \frac{1}{2} \right) \right) = \kappa_+ (x_1^+)^{p-1} ,
\end{equation}
for all $(x',0) \in \mathbb{R}^{n-1} \times \{0\}$, where the last step follows from $\op{Im} \log(ix_1) = \pi/2$ if $x_1 > 0$ and $\op{Im} \log(ix_1) = -\pi/2$ if $x_1 < 0$.

Now let $v \in C^{p-1,1}(\overline{\mathbb{R}^n_+})$ solve \eqref{irregularity eqn4} and $w$ be as in \eqref{irregularity eqn5}.  $\Delta (v-w) = 0$ in $\mathbb{R}^n_+$ and $D_n (v-w) = 0$ on $\mathbb{R}^{n-1} \times \{0\}$.  Thus the extension of $v-w$ across $\mathbb{R}^{n-1} \times \{0\}$ via odd reflection is a harmonic function on $\mathbb{R}^n$.  Hence $v-w$ is smooth on $\overline{\mathbb{R}^n_+}$.  Since $w \not\in C^{p-1,1}(\overline{\mathbb{R}^n_+})$, we have $v \not\in C^{p-1,1}(\overline{\mathbb{R}^n_+})$.
\end{proof}

\section{Normalized solutions} \label{sec:almgren setup sec}

Let $u \in W^{1,2}(\Omega)$ be a solution to \eqref{main ptop}.  We define the \emph{contact set} $\Lambda(u)$ and the \emph{free boundary} $\Sigma(u)$ of $u$, respectively, by
\begin{equation*}
	\Lambda(u) = \{ x \in \Gamma : u(x) = h(x) \} , \quad\quad \Sigma(u) = \partial_{\Gamma} \Lambda(u) ,
\end{equation*}
where $\partial_{\Gamma}$ denotes the boundary (frontier) with respect to the topology of $\Gamma$.  The remainder of the paper will be focused on studying the regularity of the free boundary $\Sigma(u)$.  %In Sections~\ref{sec:almgren setup sec}, \ref{sec:almgren sec}, \ref{sec:growth sec}, and \ref{sec:tangent fn sec},
We will use Almgren's frequency function to determine the local asymptotic behavior of $u$ at a free boundary, using a combination of approaches of~\cite{GP09},~\cite{GSVG14}, and~\cite{DJ}.

We shall assume the following.  Let $2 \leq p < \infty$.  Let $\Omega = B^+_1(0)$ and $\Gamma = B'_1(0)$.  For a fixed positive integer  $\kappa$, let $a^{ij} \in C^{\kappa-1,1}(B^+_1(0))$ be such that $a^{ij} = a^{ji}$ on $B^+_1(0)$, the ellipticity condition \eqref{ellipticity hyp} holds  for some constants $0 < \lambda \leq \Lambda < \infty$, and the normalization assumptions \eqref{identity_origin} and \eqref{regularity assumption} are satisfied.  Let $k_+,k_- \in C^{0,1}(B'_1(0))$ with $k_+,k_- \geq 0$, and let $h \in C^{\kappa,1}(B'_1(0))$.  We explicitly observe  that the regularity assumptions on  $a^{ij}$ and $h$ will be used in Lemma~\ref{h taylor thm} below to construct the extension $\overline{h}$ of a Taylor polynomial of $h$.  In the special case that $h = 0$ on $B'_1(0)$, it suffices to assume that $a^{ij} \in C^{0,1}(B^+_1(0) \cup B'_1(0))$.  Finally, let $u \in W^{1,2}(B^+_1(0))$ be a weak solution to \eqref{main ptop} and $0 \in \Sigma(u)$.

Following the approach in~\cite{GP09}, we will extend the degree $\kappa$ Taylor polynomial of $h$ to a polynomial solution $\overline{h}$ to $|D_i (a^{ij} D_j \overline{h})| \leq C |x|^{\kappa-1}$ in $B^+_1(0)$ and $D_n \overline{h} = 0$ on $B'_1(0)$, which we will then subtract from $u$ (see \eqref{frequency v} below).

\begin{lemma} \label{h taylor thm}
Let $\kappa$ be a positive integer.  Let $a^{ij} \in C^{\kappa-1,1}(B^+_1(0))$ such that $a^{ij} = a^{ji}$ on $B^+_1(0)$ and \eqref{ellipticity hyp} holds true for some constants $0 < \lambda \leq \Lambda < \infty$.  Let $h \in C^{\kappa,1}(B'_1(0))$.  There exists a unique polynomial $\overline{h} : \mathbb{R}^n \rightarrow \mathbb{R}$ of degree at most $\kappa$ such that
\begin{align}\label{h taylor concl}
	\big| D_i (a^{ij} D_j \overline{h}) \big| &\leq C |x|^{\kappa-1} \text{ in } B^+_1(0), \quad D_n \overline{h} = 0 \text{ on } B'_1(0), \nonumber \\
	\big| \overline{h}(x',0) - h(x') \big| &+ \sum_{i=1}^{n-1} |x'| \big| D_i \overline{h}(x',0) - D_i h(x') \big| \nonumber \\
		&+ \sum_{i,j=1}^{n-1} |x'|^2 \big| D_{ij} \overline{h}(x',0) - D_{ij} h(x') \big| \leq C |x'|^{\kappa+1},
\end{align}
for some constant $C = C(n, \lambda, \|a^{ij}\|_{C^{\kappa-1,1}(B^+_1(0))}, \|h\|_{C^{\kappa,1}(B'_1(0))}) \in (0,\infty)$.
\end{lemma}

\begin{proof}
We will prove the result only in the case $\kappa\geq 2$, and leave the easy modifications for the case $\kappa =1$ to the reader. Express the  Taylor polynomial of $a^{ij}$ of degree $(\kappa-1)$ as
\begin{equation*}
	\sum_{k=0}^{\kappa-1} a^{ij}_k(x') \,x_n^k
\end{equation*}
for some polynomials $a^{ij}_k(x')$ on $\mathbb{R}^{n-1}$ of degree at most $\kappa-1 - k$.  %By \eqref{regularity assumption}, $a^{in}_0(x') = a^{ni}_0(x') = 0$ for $i = 1,2,\ldots,n-1$.
Also, let
\begin{equation*}
	\overline{h}(x',x_n) = \sum_{k=0}^{\kappa} h_k(x') \,x_n^k
\end{equation*}
for some polynomials $h_k(x')$ on $\mathbb{R}^{n-1}$ of degree at most $\kappa - k$.  We have that
%Setting the degree $(\kappa-2)$ Taylor polynomial of $D_i (a^{ij} D_j \overline{h})$ at the origin equal to zero, %and recalling \eqref{regularity assumption},
\begin{align*}
	D_i (a^{ij} D_j \overline{h}) =& \sum_{k=0}^{\kappa-2} \sum_{\ell=0}^k \sum_{i,j=1}^{n-1} D_i (a^{ij}_{k-\ell}(x') \,D_j h_{\ell}(x')) \,x_n^k
		+ \sum_{k=0}^{\kappa-2} \sum_{\ell=0}^k \sum_{i=1}^{n-1} (\ell+1) \,D_i (a^{in}_{k-\ell}(x') \,h_{\ell+1}(x')) \,x_n^k
		\\&+ \sum_{k=0}^{\kappa-2} \sum_{\ell=0}^{k+1} \sum_{j=1}^{n-1} (k+1) \,a^{nj}_{k-\ell+1}(x') \,D_j h_{\ell}(x') \,x_n^k
		+ \sum_{k=0}^{\kappa-2} \sum_{\ell=0}^k (\ell+1)(k+1) \,a^{nn}_{k-\ell+1}(x') \,h_{\ell+1}(x') \,x_n^k
		\\&+ \sum_{k=0}^{\kappa-2} (k+2)(k+1) \,a^{nn}_0(x') \,h_{k+2}(x') \,x_n^k + E(x) ,
\end{align*}
where $E(x)$ is an error term such that $|E(x)| \leq C |x|^{\kappa-1}$ for some constant $C = C(n, \|a^{ij}\|_{C^{\kappa-1,1}(B^+_1(0))},$ $\|h\|_{C^{\kappa,1}(B'_1(0))}) \in (0,\infty)$.  Therefore \eqref{h taylor concl} holds true if and only if $h_0(x')$ is the degree $\kappa$ Taylor polynomial of $h(x')$, $h_1(x') = 0$ on $\mathbb{R}^{n-1}$, and $h_{k+2}(x')$ satisfies the recurrence relation
\begin{align}\label{h taylor recurrence}
	h_{k+2}(x') =&\, \frac{-1}{(k+1)(k+2) \,a^{nn}_0(x')} \left( \sum_{\ell=0}^k \sum_{i,j=1}^{n-1} D_i (a^{ij}_{k-\ell}(x') \,D_j h_{\ell}(x')) \right. \nonumber
		\\&+ \sum_{\ell=0}^k \sum_{i=1}^{n-1} (\ell+1) \,D_i (a^{in}_{k-\ell}(x') \,h_{\ell+1}(x'))
		+ \sum_{\ell=0}^{k+1} \sum_{j=1}^{n-1} (k+1) \,a^{nj}_{k-\ell+1}(x') \,D_j h_{\ell}(x') \nonumber
		\\&\left. + \sum_{\ell=0}^k (\ell+1)(k+1) \,a^{nn}_{k-\ell+1}(x') \,h_{\ell+1}(x') \right) + E(x)
\end{align}
for $k = 0,1,2,\ldots,\kappa-2$, where $E(x)$ is a polynomial such that $|E(x)| \leq C |x|^{\kappa-k-1}$ for some constant $C \in (0,\infty)$ (independent of $|x|$).  In particular, observe that the recurrence relation \eqref{h taylor recurrence} uniquely defines $h_{k+2}(x')$.
\end{proof}

Now recall that $u \in W^{1,2}(B^+_1(0))$ is a solution to \eqref{main ptop} with $\Omega = B^+_1(0)$ and $\Gamma = B'_1(0)$, and that $0 \in \Sigma(u)$.  Following~\cite{GP09}, take an integer $\kappa \geq 2$ and let $\overline{h} = \overline{h}_{\kappa}$ be as in Lemma~\ref{h taylor thm}.  Set
\begin{gather}
	\widetilde{h}(x',x_n) = \widetilde{h}_{\kappa}(x',x_n) = \overline{h}(x',x_n) - \overline{h}(x',0) + h(x'), \nonumber \\
	v(x',x_n) = v_{\kappa}(x',x_n) = u(x',x_n) - \widetilde{h}(x',x_n) , \nonumber \\
	\label{frequency v} f = f_{\kappa} = -D_i (a^{ij} D_j \widetilde{h})
\end{gather}
for each $(x',x_n) \in B^+_1(0)$.  By \eqref{h taylor concl}, %and \eqref{regularity assumption},
\begin{equation} \label{f bound0}
	|f| \leq \left| \sum_{i,j=1}^n D_i (a^{ij} D_j \overline{h}(x',x_n)) \right| + \left| \sum_{i=1}^n \sum_{j=1}^{n-1} D_i (a^{ij} D_j (h(x') - \overline{h}(x',0))) \right|
		\leq C |x|^{\kappa-1}
\end{equation}
in $B^+_1(0)$ for some constant $C = C(n, \lambda,  \|a^{ij}\|_{C^{\kappa-1,1}(B^+_1(0))}, \|h\|_{C^{\kappa,1}(B'_1(0))}) \in (0,\infty)$.  Since $u$ is $J$-minimizing, $v$ is a minimizer for the functional in $w \in W^{1,2}(B^+_1(0))$
\begin{equation*}
	J(w + \widetilde{h}) = \frac{1}{2} \int_{B^+_1(0)} a^{ij} D_i (w + \widetilde{h}) \,D_j (w + \widetilde{h})
		+ \frac{1}{p} \int_{B'_1(0)} (k_+ (w^+)^p + k_- (w^-)^p) .
\end{equation*}
Using \eqref{regularity assumption}, and the facts that $D_i (a^{ij} D_j \widetilde{h}) = -f$ weakly in $B^+_1(0)$ and $D_n \widetilde{h} = 0$ on $B'_1(0)$ (since $h_1(x')=0$ on $\mathbb{R}^{n-1}$), we obtain
\begin{align*}
	&\frac{1}{2} \int_{B^+_1(0)} (a^{ij} D_i(w + \widetilde{h})\,D_j(w + \widetilde{h}) - a^{ij} D_i(v + \widetilde{h})\,D_j(v + \widetilde{h}))
	\\&= \frac{1}{2} \int_{B^+_1(0)} a^{ij} D_i(w - v) \,D_j(w + v + 2\widetilde{h})
	\\&= \frac{1}{2} \int_{B^+_1(0)} (a^{ij} D_i w \,D_j w - a^{ij} D_i v \,D_j v + 2 \,(w - v) \,f)
\end{align*}
for all $w \in W^{1,2}(B^+_1(0))$ with $w = v$ on $(\partial B_1(0))^+$.  Hence $v$ is a minimizer for the functional in $w \in W^{1,2}(B^+_1(0))$
\begin{equation} \label{J minus h}
	\widetilde{J}(w) = \frac{1}{2} \int_{B^+_1(0)} (a^{ij} D_i w \,D_j w + 2 \,w \,f) + \frac{1}{p} \int_{B'_1(0)} (k_+ (w^+)^p + k_- (w^-)^p) .
\end{equation}
By a standard variational argument (see for instance the proof of Lemma~4.1 of~\cite{ALP15}),
\begin{equation} \label{ptop freq integral eqn}
	\int_{B^+_1(0)} (a^{ij} D_j v \,D_i \zeta + f \,\zeta) + \int_{B'_1(0)} (k_+ (v^+)^{p-1} - k_- (v^-)^{p-1}) \,\zeta = 0
\end{equation}
for each $\zeta \in C^1_c(B^+_1(0) \cup B'_1(0))$.  That is, $v$ is a weak solution to
\begin{gather}
	D_i (a^{ij} D_j v) = f \text{ in } B^+_1(0), \nonumber \\
	a^{nn} D_n v = k_+ (v^+)^{p-1} - k_- (v^-)^{p-1} \text{ on } B'_1(0) . \label{ptop freq eqn}
\end{gather}

\section{Almgren's monotonicity formula} \label{sec:almgren sec}

Throughout this section, we shall assume the following:

\begin{hypothesis}\label{almgren hyp}
Let $2 \leq p < \infty$ and $\kappa$ be a positive integer.  Let $a^{ij} \in C^{0,1}(B^+_1(0) \cup B'_1(0))$ such that $a^{ij} = a^{ji}$ on $B^+_1(0)$, \eqref{ellipticity hyp} holds true for some constants $0 < \lambda \leq \Lambda < \infty$, and satisfying \eqref{identity_origin} and \eqref{regularity assumption}.  Let $f \in L^{\infty}(B^+_1(0))$ such that
\begin{equation} \label{f bound}
	\sup_{x \in B^+_1} |x|^{1-\kappa} |f(x)| < \infty .
\end{equation}
Let $k_+,k_- \in C^{0,1}(B'_1(0))$ such that $k_+, k_- \geq 0$ on $B'_1(0)$.
\end{hypothesis}

We will show that a frequency function of Almgren's type   for $v$ is monotone nondecreasing by suitably modifying an approach  from~\cite{GSVG14}.  It will be convenient to define $b^{ij}, \mu \in C^{0,1}(B^+_1(0) \cup B'_1(0))$ by
\begin{align}
	\label{b defn} a^{ij}(x) &= \delta_{ij} + b^{ij}(x) \\
	\label{mu defn} \mu(x) &= \sum_{i,j=1}^n a^{ij}(x) \,\frac{x_i x_j}{|x|^2} = 1 + \sum_{i,j=1}^n b^{ij}(x) \,\frac{x_i x_j}{|x|^2}
\end{align}
for all $x \in B^+_1(0) \cup B'_1(0)$.  By \eqref{ellipticity hyp},
\begin{equation}\label{mu positive}
	\lambda \leq \mu(x) \leq \Lambda
\end{equation}
on $B^+_1(0) \cup B'_1(0)$.  Recalling that $a^{ij}(0) = \delta_{ij}$,
\begin{equation}\label{b eqn}
	|b^{ij}(x)| \leq C |x|, \quad\quad |Db^{ij}(x)| \leq C, \quad\quad |\mu(x) - 1| \leq C |x|, \quad\quad |D\mu(x)| \leq C
\end{equation}
on $B^+_1(0) \cup B'_1(0)$, for some constant $C = C(n, \|Da^{ij}\|_{L^{\infty}(B^+_1)}) \in (0,\infty)$.  The relevant Almgren's frequency function is defined as follows.

\begin{definition} \label{frequency defn} {\rm
Assume Hypothesis~\ref{almgren hyp} and let $v \in C^1(B^+_1(0) \cup B'_1(0))$ be a solution to \eqref{ptop freq eqn}.  We let
\begin{align}
	H(\rho) = H_{v}(\rho) &= \rho^{1-n} \int_{(\partial B_{\rho}(0))^+}  \mu \,v^2 ,\label{eq.H} \\
	I(\rho) = I_{v}(\rho) &= \rho^{2-n} \int_{(\partial B_{\rho}(0))^+} a^{ij} \,v \,D_i v \,\frac{x_j}{\rho} ,\label{eq.I} \\
	D(\rho) = D_{v}(\rho) &= \rho^{2-n} \int_{B^+_{\rho}(0)} a^{ij} D_i v \,D_j v
		+ \frac{2}{p} \,\rho^{2-n} \int_{B'_{\rho}(0)} (k_+ (v^+)^p + k_- (v^-)^p) \label{eq.D}
\end{align}
for each $0 < \rho < 1$, where $\mu$ is as in \eqref{mu defn}.
} \end{definition}

\begin{lemma}
Assume Hypothesis~\ref{almgren hyp} and let $v \in C^1(B^+_1 \cup B'_1)$ be a solution to \eqref{ptop freq eqn}.  Then
\begin{equation} \label{variation concl1}
	\int_{B^+_1} (a^{ij} D_i v \,D_j v \,\zeta + a^{ij} v \,D_j v \,D_i \zeta + v \,f \,\zeta) + \int_{B'_1} (k_+ (v^+)^p + k_- (v^-)^p) \,\zeta = 0
\end{equation}
for each $\zeta \in C^1_c(B^+_1 \cup B'_1)$.  Moreover,
\begin{align} \label{variation concl2}
	&\int_{B^+_1} (a^{ij} D_i v \,D_j v \,\div \zeta - 2 \,a^{ij} D_j v \,D_k v \,D_i \zeta^k + D_k a^{ij} \,D_i v \,D_j v \,\zeta^k
		- 2 \,D_i v \,f \,\zeta^i) \nonumber \\&
		+ \frac{2}{p} \int_{B'_1} ( \div_{B'_1}(k_+ \zeta) \,(v^+)^p + \div_{B'_1}(k_- \zeta) \,(v^-)^p) = 0
\end{align}
for each $\zeta = (\zeta^1, \ldots, \zeta^n) \in C^1_c(B^+_1 \cup B'_1, \mathbb{R}^n)$ with $\zeta^n = 0$ on $B'_1$, where $\div$ and $\div_{B'_1}$ denote the divergence operators on $B^+_1$ and $B'_1$ respectively.
\end{lemma}

\begin{proof}
Replacing $\zeta$ with $v \,\zeta$ in \eqref{ptop freq integral eqn} and noting that $\pm (v^{\pm})^{p-1} v = (v^{\pm})^p$ on $B'_1$ gives us \eqref{variation concl1}.  To obtain \eqref{variation concl2}, notice that since $v$ is $\widetilde{J}$-minimizing (where $\widetilde{J}$ is as in \eqref{J minus h}),
\begin{align} \label{variation eqn1}
	&\left. \frac{d}{dt} \widetilde{J}(v(x + t \zeta(x))) \right|_{t=0} \nonumber
		\\&= \left. \frac{d}{dt} \left( \frac{1}{2} \int_{B^+_1} a^{ij}(x) \,D_k v(x + t \zeta(x)) \,(\delta_{ik} + t D_i \zeta^k(x)) \,D_l v(x + t \zeta(x))
			\,(\delta_{jl} + t D_j \zeta^l(x)) \,dx  \right. \right. \nonumber \\&\hspace{15mm}
			+ \int_{B^+_1} v(x + t \zeta(x)) \,f(x)
			+ \frac{1}{p} \int_{B'_1} k_+(x) \,(v(x + t \zeta(x))^+)^p \,dx \nonumber \\&\hspace{15mm} \left.
			+ \frac{1}{p} \int_{B'_1} k_-(x) \,(v(x + t \zeta(x))^-)^p \,dx \,\right|_{t=0}
\end{align}
for all $\zeta = (\zeta^1, \ldots, \zeta^n) \in C^1_c(B^+_1 \cup B'_1, \mathbb{R}^n)$.  We have that %For the term in \eqref{variation eqn1} with $f$ we have that
\begin{equation*}
	\left. \frac{d}{dt} \int_{B^+_1} v(x + t \zeta(x)) \,f(x) \,dx \,\right|_{t=0} = \int_{B^+_1} D_i v \,f \,\zeta^i .
\end{equation*}
For the remaining terms in \eqref{variation eqn1}, we use the change of variables $y = x + t \zeta(x)$ and note that $x = y - t \zeta(y) + O(t^2)$.  By the dominated convergence theorem,
\begin{align*}
	&\left. \frac{d}{dt} \int_{B^+_1} a^{ij}(x) \,D_k v(x + t \zeta(x)) \,(\delta_{ik} + t D_i \zeta^k(x)) \,D_l v(x + t \zeta(x))
		\,(\delta_{jl} + t D_j \zeta^l(x)) \,dx \,\right|_{t=0} \nonumber
	\\&= \left. \frac{d}{dt} \int_{B^+_1} a^{ij}(y - t \zeta(y)) \,D_k v(y) \,(\delta_{ik} + t D_i \zeta^k(y)) \,D_l v(y)
		\,(\delta_{jl} + t D_j \zeta^l(y)) \,(1 - t \div \zeta(y)) \,dy \,\right|_{t=0}
	\\&= -\int_{B^+_1} (a^{ij} D_i v \,D_j v \,\div \zeta - 2 \,a^{ij} D_j v \,D_k v \,D_i \zeta^k
		+ D_k a^{ij} D_i v \,D_j v \,\zeta^k)
\end{align*}
and
\begin{align*}
	&\left. \frac{d}{dt} \int_{B'_1} (k_+(x) \,(v(x + t \zeta(x))^+)^p + k_-(x) \,(v(x + t \zeta(x))^-)^p) \,dx \,\right|_{t=0} \nonumber
	\\&= \left. \frac{d}{dt} \int_{B'_1} (k_+(y - t \zeta(y)) \,(v(y)^+)^p + k_-(y - t \zeta(y)) \,(v(y)^-)^p) (1 - t \div_{B'_1} \zeta(y)) \,dy \,\right|_{t=0}
	\\&= -\int_{B'_1} ((k_+ \div_{B'_1} \zeta + \nabla^{B'_1} k_+ \cdot \zeta) \,(v^+)^p
		+ (k_- \div_{B'_1} \zeta + \nabla^{B'_1} k_- \cdot  \zeta) \,(v^-)^p)
	\\&= -\int_{B'_1} (\div_{B'_1} (k_+ \zeta) \,(v^+)^p + \div_{B'_1} (k_- \zeta) \,(v^-)^p) ,
\end{align*}
where $\nabla^{B'_1}$ denotes the gradient on $B'_1$, and thus \eqref{variation concl2} holds true.
\end{proof}

\begin{lemma}
Let $a^{ij} \in C^{0,1}(B^+_1(0) \cup B'_1(0))$ such that $a^{ij} = a^{ji}$ on $B^+_1(0)$, \eqref{ellipticity hyp} holds true for some constants $0 < \lambda \leq \Lambda < \infty$, and satisfying \eqref{identity_origin} and \eqref{regularity assumption}. Let $w \in W^{1,2}(B^+_1)$ and
\begin{equation}\label{eq.H_w}
	H_{w}(\rho) = \rho^{1-n} \int_{(\partial B_{\rho})^+} \mu \,w^2
\end{equation}
for each $0 < \rho < 1$, where $\mu$ is as in \eqref{mu defn}.  Then $H_{w}$ is absolutely continuous as a function of $\rho \in (0,1)$ and
\begin{equation} \label{variation eqn2}
	H'_{w}(\rho) =\, \frac{2}{\rho} I(\rho)
		+ \rho^{1-n} \int_{(\partial B_{\rho})^+} w^2 \left( (1-n)\,\frac{\mu-1}{\rho} + D_i\bigg( b^{ij} \frac{x_j}{|x|} \bigg) \right)
\end{equation}
for $\mathcal{L}^1$-a.e.~$0 < \rho < 1$, where $b^{ij}$ is as in \eqref{b defn}.
\end{lemma}

\begin{proof}
By the divergence theorem, \eqref{b defn}, \eqref{mu defn}, and \eqref{regularity assumption},
\begin{align*}
	\int_{(\partial B_{\rho})^+} \mu \,w^2
		&= \int_{(\partial B_{\rho})^+} a^{ij} \frac{x_i x_j}{|x|^2} \,w^2
		\\&= \int_{B'_{\rho}} a^{nj} \frac{x_j}{|x|} \,w^2 + \int_{B_{\rho}^+} \left( 2 a^{ij} \,w \,D_i w \,\frac{x_j}{|x|} + w^2 \,D_i\bigg( a^{ij} \frac{x_j}{|x|} \bigg)
		\right) \\&= \int_{B_{\rho}^+} \left( 2 a^{ij} \,w \,D_i w \,\frac{x_j}{|x|} + w^2 \,\frac{n-1}{|x|} + w^2 \,D_i\bigg( b^{ij} \frac{x_j}{|x|} \bigg) \right) .
\end{align*}
Hence by the coarea formula, $H_w(\rho)$ is absolutely continuous and %$\rho^{1-n} \int_{(\partial B_{\rho})^+} \mu \,w^2$ is absolutely continuous and
\begin{align*}
	\frac{\partial}{\partial \rho} \left( \rho^{1-n} \int_{(\partial B_{\rho})^+} \mu \,w^2 \right)
		=& (1-n) \,\rho^{-n} \int_{(\partial B_{\rho})^+} \mu \,w^2 + \rho^{1-n} \,\frac{d}{d\rho} \left( \int_{(\partial B_{\rho})^+} \mu \,w^2 \right)
		\\=&\, 2 \rho^{1-n} \int_{(\partial B_{\rho})^+} a^{ij} \,w \,D_i w \,\frac{x_j}{\rho}
			\\&+ \rho^{1-n} \int_{(\partial B_{\rho})^+} w^2 \left( (1-n)\,\frac{\mu-1}{\rho} + D_i\bigg( b^{ij} \frac{x_j}{|x|} \bigg) \right)
\end{align*}
for $\mathcal{L}^1$-a.e.~$0 < \rho < 1$.
\end{proof}

\begin{lemma}
Assume Hypothesis~\ref{almgren hyp} and let $v \in C^1(B^+_1 \cup B'_1)$ be a solution to \eqref{ptop freq eqn}.  Let $I(\rho)$ and $D(\rho)$ be as in Definition~\ref{frequency defn}.  Then $D(\rho)$ and $I(\rho)$ are locally absolutely continuous functions of $\rho \in (0,1)$, with
\begin{equation} \label{variation eqn4}
	D(\rho) = I(\rho) - \rho^{2-n} \int_{B^+_{\rho}} v \,f - \frac{p-2}{p} \,\rho^{2-n} \int_{B'_{\rho}} (k_+ \,(v^+)^p + k_- \,(v^-)^p)
\end{equation}
for all $0 < \rho \leq 1$ and
\begin{align} \label{variation eqn5}
	D'(\rho) \geq&\, 2 \rho^{2-n} \int_{(\partial B_{\rho})^+} \frac{1}{\mu} \left( \sum_{i,j=1}^n a^{ij} D_j v \,\frac{x_i}{\rho} \right)^2
		- 2 \rho^{1-n} \int_{B^+_{\rho}} \frac{1}{\mu} \,a^{ij} \,D_j v \,f \,x_i \nonumber \\&
		%+ \frac{2}{p} \,\rho^{1-n} \int_{B'_{\rho}} (k_+ \,(v^+)^p + k_- \,(v^-)^p) \nonumber \\&\hspace{30mm}
		+ \frac{2}{p} \,\rho^{1-n} \int_{B'_{\rho}} \frac{1}{\mu} \,a^{ij} x_i \,(D_j k_+ \,(v^+)^p + D_j k_- \,(v^-)^p) - C D(\rho)
\end{align}
for $\mathcal{L}^1$-a.e.~$0 < \rho < 1$, where $\mu$ is as in \eqref{mu defn} and $C = C(n, \lambda, \|Da^{ij}\|_{L^{\infty}(B^+_1)}) \in (0,\infty)$ is a constant.
\end{lemma}

\begin{proof}
By \eqref{variation concl1} with $\zeta$ approximating the characteristic function $\mathbf{1}_{B^+_{\rho}}$ on $B^+_{\rho}$,
\begin{equation*}
	\int_{B^+_{\rho}} (a^{ij} D_i v \,D_j v + v \,f) - \int_{(\partial B_{\rho})^+} a^{ij} \,v \,D_j v \,\frac{x_i}{|x|} + \int_{B'_{\rho}} (k_+ (v^+)^p + k_- (v^-)^p) = 0 ,
\end{equation*}
which after multiplying by $\rho^{2-n}$ and rearranging terms gives us \eqref{variation eqn4}.  Let $Z = (Z^1,Z^2,\ldots,Z^n) \in C^{0,1}(B^+_1 \cup B'_1, \mathbb{R}^n)$ be the vector field defined by
\begin{equation*}
	Z^i(x) = \frac{1}{\mu(x)} \sum_{j=1}^n a^{ij}(x) \,x_j .
\end{equation*}
By the definition of $\mu$ in \eqref{mu defn} and by \eqref{mu positive} and \eqref{b eqn},
\begin{gather}
	\label{Z eqn1} \sum_{i=1}^n Z^i(x) \,x_i = \frac{1}{\mu} \sum_{i,j=1}^n a^{ij}(x) \,x_i \,x_j = r^2, \\
	\label{Z eqn2} |D_i Z^k(x) - \delta_{ik}| \leq C |x| , \quad |\op{div} Z(x) - n| \leq C |x|
\end{gather}
on $B^+_1 \cup B'_1$, where $C = C(n, \lambda, \|Da^{ij}\|_{L^{\infty}(B^+_1)}) \in (0,\infty)$ is a constant.  By \eqref{variation concl2} with $\zeta(x)$ approximating $\mathbf{1}_{B^+_{\rho}}(x) \,Z(x)$ and using \eqref{Z eqn1}, see also \cite[LemmaA.9]{GSVG14},
\begin{align*}
	&\int_{B^+_{\rho}} (a^{ij} D_i v \,D_j v \,\div Z - 2 \,a^{ij} D_j v \,D_k v \,D_i Z^k + D_k a^{ij} \,D_i v \,D_j v \,Z^k - 2 \,D_i v \,f \,Z^i) \nonumber \\&
		+ \frac{2}{p} \int_{B'_{\rho}} (k_+ (v^+)^p + k_- (v^-)^p) \div_{B'_1} Z
		+ \frac{2}{p} \int_{B'_{\rho}} (Z \cdot Dk_+ \,(v^+)^p + Z \cdot Dk_+ \,(v^-)^p) \nonumber \\
	=&\, \rho \int_{(\partial B_{\rho})^+} a^{ij} D_i v \,D_j v
		- 2\rho \int_{(\partial B_{\rho})^+} \frac{1}{\mu} \left( \sum_{i,j=1}^n a^{ij} D_j v \,\frac{x_i}{\rho} \right)^2
		+ \frac{2}{p} \,\rho \int_{\partial B'_{\rho}} (k_+ (v^+)^p + k_- (v^-)^p) .
\end{align*}
Hence by \eqref{Z eqn2}, together with $\lambda \,|Dv|^2 \leq a^{ij} D_i v \,D_j v$ on $B^+_1$ by \eqref{ellipticity hyp}, we have that
\begin{align} \label{variation Dderiv1}
	&(n-2) \int_{B^+_{\rho}} a^{ij} D_i v \,D_j v - 2 \int_{B^+_{\rho}} \frac{1}{\mu} \,a^{ij} D_j v \,f \,x_i \nonumber \\&
		+ \frac{2(n-1)}{p} \int_{B'_{\rho}} (k_+ (v^+)^p + k_- (v^-)^p)
		+ \frac{2}{p} \int_{B'_{\rho}} (Z \cdot Dk_+ \,(v^+)^p + Z \cdot Dk_+ \,(v^-)^p) \nonumber \\
	\leq&\, \rho \int_{(\partial B_{\rho})^+} a^{ij} D_i v \,D_j v
		- 2\rho \int_{(\partial B_{\rho})^+} \frac{1}{\mu} \left( \sum_{i,j=1}^n a^{ij} D_j v \,\frac{x_i}{\rho} \right)^2 \nonumber
		\\&+ \frac{2}{p} \,\rho \int_{\partial B'_{\rho}} (k_+ (v^+)^p + k_- (v^-)^p) + C \rho \int_{B^+_{\rho}} a^{ij} D_i v \,D_j v
			+ C\rho \int_{B'_{\rho}} (k_+ (v^+)^p + k_- (v^-)^p)
\end{align}
for some constant $C \in (0,\infty)$ depending only on $n$, $\lambda$, and $\|Da^{ij}\|_{L^{\infty}(B^+_1)}$.  On the other hand, by the coarea formula $D(\rho)$ is locally absolutely continuous on $(0,1)$ and
\begin{align} \label{variation Dderiv2}
	D'(\rho) =&\, (2-n) \,\rho^{1-n} \int_{B^+_{\rho}} a^{ij} D_i v \,D_j v
		+ \rho^{2-n} \int_{\partial B^+_{\rho}} a^{ij} D_i v \,D_j v \nonumber \\&
		+ \frac{2(2-n)}{p} \,\rho^{1-n} \int_{B'_{\rho}} (k_+ \,(v^+)^p + k_- \,(v^-)^p)
		+ \frac{2}{p} \,\rho^{2-n} \int_{\partial B'_{\rho}} (k_+ \,(v^+)^p + k_- \,(v^-)^p)
\end{align}
for $\mathcal{L}^1$-a.e.~$0 < \rho < 1$.  Combining \eqref{variation Dderiv1} and \eqref{variation Dderiv2} gives us \eqref{variation eqn5}.  Note that by \eqref{variation eqn4}, since $D(\rho)$ is absolutely continuous on $(0,1)$, $I(\rho)$ must also be absolutely continuous on $(0,1)$.
\end{proof}

\begin{lemma} \label{v L2 lemma}
Assume Hypothesis~\ref{almgren hyp} and let $v \in C^1(B^+_1 \cup B'_1)$ be a solution to \eqref{ptop freq eqn}.  Then
\begin{equation} \label{v L2 concl}
	\rho^{-n} \int_{B_{\rho}} v^2 \leq C H(\rho) + C \rho^{2\kappa+2}
\end{equation}
for all $0 < \rho < 1$, where $H(\rho)$ is as in \eqref{eq.H} and $C \in (0,\infty)$ is a constant depending only on $n$, $\lambda$, $\|Da^{ij}\|_{L^{\infty}(B^+_1)}$, and $\sup_{B^+_1} |x|^{1-\kappa} |f|$.
\end{lemma}

\begin{proof}
We compute that
\begin{equation} \label{v L2 eqn1}
	\frac{d}{d\rho} \left( \int_{(\partial B_{\rho})^+} \mu \,v^2 \right)
		= \frac{d}{d\rho} \left( \rho^{n-1} H(\rho) \right) = (n-1) \,\rho^{n-2} H(\rho) + \rho^{n-1} H'(\rho)
		\geq \rho^{n-1} H'(\rho)
\end{equation}
for $\mathcal{L}^1$-a.e.~$0 < \rho < 1$.  By \eqref{variation eqn2} with $w = v$, \eqref{b eqn}, and \eqref{mu positive},
\begin{equation*}
	\frac{d}{d\rho} \left( \int_{(\partial B_{\rho})^+} \mu \,v^2 \right) \geq 2 \rho^{n-2} I(\rho) - C \int_{(\partial B_{\rho})^+} \mu \,v^2 .
\end{equation*}
for $\mathcal{L}^1$-a.e.~$0 < \rho < 1$, where $C \in (0,\infty)$ is a constant depending only on $n$, $\lambda$, and $\|Da^{ij}\|_{L^{\infty}(B^+_1)}$.  By \eqref{variation eqn4},
\begin{equation} \label{v L2 eqn2}
	\frac{d}{d\rho} \left( \int_{(\partial B_{\rho})^+} \mu \,v^2 \right) \geq 2 \int_{B^+_{\rho}} f \,v - C \int_{(\partial B_{\rho})^+} \mu \,v^2
\end{equation}
for $\mathcal{L}^1$-a.e.~$0 < \rho < 1$, where $C \in (0,\infty)$ is a constant depending only on $n$, $\lambda$, and $\|Da^{ij}\|_{L^{\infty}(B^+_1)}$.  Multiplying by sides of \eqref{v L2 eqn2} by $e^{C\rho}$ and rearranging terms,
\begin{equation*}
	\frac{d}{d\rho} \left( e^{C\rho} \int_{(\partial B_{\rho})^+} \mu \,v^2 \right) \geq 2 e^{C\rho} \int_{B^+_{\rho}} f \,v
\end{equation*}
for $\mathcal{L}^1$-a.e.~$0 < \rho < 1$.  Integrating over $[\sigma,\rho]$ gives us
\begin{equation*}
	e^{C\rho} \int_{(\partial B_{\rho})^+} \mu \,v^2 \geq \int_{(\partial B_{\sigma})^+} \mu \,v^2 - 2 e^{C\rho} \rho \int_{B^+_{\rho}} |f| \,|v|
\end{equation*}
for all $0 < \sigma \leq \rho \leq 1$.  Integrating over $\sigma \in (0,\rho]$ gives us
\begin{equation*}
	e^{C\rho} \rho \int_{(\partial B_{\rho})^+} \mu \,v^2 \geq \int_{B^+_{\rho}} \mu \,v^2 - 2 e^{C\rho} \rho^2 \int_{B^+_{\rho}} |f| \,|v|
\end{equation*}
for all $0 < \rho \leq 1$.  By applying the Cauchy-Schwarz inequality,
\begin{equation*}
	\int_{B^+_{\rho}} \mu \,v^2 \leq 2 e^{C \rho} \rho \int_{(\partial B_{\rho})^+} \mu \,v^2 + 4 e^{2 C \rho} \rho^4 \int_{B^+_{\rho}} \frac{f^2}{\mu} .
\end{equation*}
Multiplying by $\rho^{-n}$ and applying \eqref{f bound} and \eqref{mu positive} gives us \eqref{v L2 concl}.
\end{proof}

We now follow an approach introduced in~\cite{GSVG14} to define the relevant Almgren frequency function.  Recalling \eqref{variation eqn2} (with $w = v$), for each $0 < \rho < 1$ we define $G(\rho) = G_{v}(\rho)$ by
\begin{equation}\label{eq.G}
	G(\rho) = \frac{\rho^{1-n}}{H(\rho)} \int_{(\partial B_{\rho})^+} v^2 \left( (1-n)\,\frac{\mu-1}{\rho} + D_i\bigg( b^{ij} \frac{x_j}{|x|} \bigg) \right)
\end{equation}
if $H(\rho) > 0$ and $G(\rho) = 0$ if $H(\rho) = 0$.  By \eqref{mu positive} and \eqref{b eqn}, $G(\rho)$ is a bounded function on $(0,1)$ with
\begin{equation} \label{Alm mono G bounded}
	|G(\rho)| \leq C
\end{equation}
for some constant $C = C(n, \lambda, \|Da^{ij}\|_{L^{\infty}(B^+_1)}) \in (0,\infty)$.  Thus, we may define $K(\rho) = K_{v}(\rho)$ by
\begin{equation} \label{eq.K}
	K(\rho) = H(\rho) \exp\left( -\int_0^{\rho} G(\tau) \,d\tau \right)
\end{equation}
for all $0 < \rho \leq 1$.  By \eqref{Alm mono G bounded},
\begin{equation} \label{Alm mono eqn1}
	\frac{1}{C} \,H(\rho) \leq K(\rho) \leq C \,H(\rho)
\end{equation}
for all $0 < \rho < 1$ and some constant $C = C(n, \lambda, \|Da^{ij}\|_{L^{\infty}(B^+_1)}) \in [1,\infty)$.  By \eqref{variation eqn2} with $w = v$, we have that
\begin{equation} \label{variation eqn6}
	\frac{\rho \,K'(\rho)}{2 \,K(\rho)} = \frac{\rho \,H'(\rho)}{2 \,H(\rho)} - \frac{\rho}{2} \,G(\rho) = \frac{I(\rho)}{H(\rho)}
\end{equation}
for $\mathcal{L}^1$-a.e.~$0 < \rho < 1$ with $H(\rho) > 0$.  We introduce the $\mathcal{L}^1$-measurable function $\Phi(\rho) = \Phi_{v}(\rho)$, called the \emph{Almgren frequency function}, as
\begin{equation} \label{eq.Phi}
	\Phi(\rho) = \frac{\rho}{2} \,\frac{d}{d\rho} \log \max \{ K(\rho) , \rho^{2\kappa} \}
\end{equation}
for $\mathcal{L}^1$-a.e.~$0 < \rho < 1$.

We are now ready to prove the main result in this section, i.e. the almost-monotonicity of the Almgren frequency function.

\begin{theorem} \label{Almgren monotonicity thm}
Assume Hypothesis~\ref{almgren hyp} and let $v \in C^1(B^+_1 \cup B'_1)$ be a solution to \eqref{ptop freq eqn}.  There exists constants $\rho_0 \in (0,1/2)$ and $C \in (0,\infty)$ depending only on $n$, $p$, $\lambda$, $\|a^{ij}\|_{C^{0,1}(B^+_1)}$, $\sup_{B^+_1} |x|^{1-\kappa} |f|$, $\|k_+\|_{C^{0,1}(B'_1)}$, $\|k_+\|_{C^{0,1}(B'_1)}$, $\|v\|_{L^{\infty}(B_1)}$ such that the function
\begin{equation} \label{Alm mono concl}
 e^{C\sigma} \,\Phi(\sigma) + C e^C \sigma \leq e^{C\rho} \,\Phi(\rho) + C e^C \rho
\end{equation}
is nondecreasing for all $0 < \rho \leq \rho_0$.
\end{theorem}

\begin{remark}\label{Almgren monotonicity rmk}{\rm
Notice that $\Phi$ is a measurable function defined only up to a set of $\mathcal{L}^1$-measure zero.  We take Theorem~\ref{Almgren monotonicity thm} to assert that there is a nondecreasing function which equals $\Phi$ $\mathcal{L}^1$-a.e.~on $(0,1)$.  To be more precise, by \eqref{variation eqn6},
\begin{equation} \label{Alm mono eqn2}
	\Phi(\rho) = \frac{I(\rho)}{H(\rho)}
\end{equation}
on the open set $\{ \rho \in (0,1) : K(\rho) > \rho^{2\kappa} \}$, whereas  $\Phi(\rho) = \kappa$ for $\mathcal{L}^1$-a.e.~$\rho \in (0,1)$ with $K(\rho) \leq \rho^{2\kappa}$.  Thus $\Phi(\rho) = \overline{\Phi}(\rho)$ for $\mathcal{L}^1$-a.e.~$\rho \in (0,1)$ where $\overline{\Phi} : (0,1) \rightarrow [0,\infty)$ is the function defined by
\begin{equation*}
	\overline{\Phi}(\rho) = \begin{cases}
		I(\rho)/H(\rho) &\text{if } K(\rho) > \rho^{2\kappa} \\
		\kappa &\text{if } K(\rho) \leq \rho^{2\kappa}
	\end{cases}
\end{equation*}
for each $\rho \in (0,1)$.  We can take Theorem~\ref{Almgren monotonicity thm} to mean that $\rho \mapsto e^{C\rho} \overline{\Phi}(\rho) + C e^C \rho$ is monotone nondecreasing on $(0,1)$.

%Since $I(\rho)$ and $H(\rho)$ are locally absolutely continuous in $(0,1)$, $\Phi(\rho)$ is locally absolutely continuous in $(\alpha, \beta)$.  We will show that $e^{C\rho} \Phi(\rho) + C \rho$ has nonnegative derivative on $(\alpha, \beta)$.  Since $K(\rho) > \rho^{2\kappa}$ for all $\alpha < \rho < \beta$ and $K(\rho) = \rho^{2\kappa}$ for $\rho = \alpha,\beta$, $\Phi(\alpha+) \geq \kappa$ and $\Phi(\beta+) \leq \kappa$.  $\Phi(\rho) = \kappa$ for $\mathcal{L}^1$-a.e.~$0 < \rho < 1$ with $K(\rho) \leq \rho^{2\kappa}$.
}\end{remark}

\begin{proof}[Proof of Theorem~\ref{Almgren monotonicity thm}]
Let's take any open interval $(\alpha, \beta) \subset \{ \rho \in (0,\rho_0) : K(\rho) > \rho^{2\kappa} \}$ and show that $e^{C\rho} \Phi(\rho) + C e^C \rho$ is monotone nondecreasing on $(\alpha, \beta)$.  Unless otherwise stated, throughout the proof we will let $C$ denote positive constants depending only on $n$, $p$, $\lambda$, $\|a^{ij}\|_{C^{0,1}(B^+_1)}$, $\sup_{B^+_1} |x|^{-\kappa+1} |f|$, $\|k_+\|_{C^{0,1}(B'_1)}$, $\|k_+\|_{C^{0,1}(B'_1)}$, $\|v\|_{L^{\infty}(B_1)}$.  Note that by \eqref{f bound}, $\sup_{B^+_1} |x|^{-\kappa+1} |f| < \infty$.

By \eqref{Alm mono eqn1},
\begin{equation} \label{Alm mono eqn3}
	\rho^{2\kappa} < K(\rho) \leq C H(\rho)
\end{equation}
for each $\alpha < \rho < \beta$.  By \eqref{v L2 concl} and \eqref{Alm mono eqn3},
\begin{equation} \label{Alm mono eqn4}
	\rho^{-n} \int_{B^+_{\rho}} v^2 \leq C H(\rho) + C \rho^{2\kappa+2} \leq C H(\rho)
\end{equation}
for all $\alpha < \rho < \beta$.  Using \eqref{ellipticity hyp}, \eqref{mu positive}, \eqref{Alm mono eqn3}, \eqref{Alm mono eqn4}, and the Cauchy-Schwarz inequality, %\eqref{f bound},
\begin{align}
	\label{Alm mono eqn5} \left| \rho^{-n} \int_{B^+_{\rho}} f \,v \right|
		&\leq \sup_{B^+_{\rho}} |f| \left(\rho^{-n} \int_{B^+_{\rho}} v^2 \right)^{1/2}
		\leq C \rho^{\kappa-1} H(\rho)^{1/2} \leq \frac{C}{\rho} \,H(\rho) , \\
	\label{Alm mono eqn6} \left| \rho^{-n} \int_{B^+_{\rho}} \frac{1}{\mu} \,a^{ij} \,D_j v \,f \,x_i \right|
		&\leq C \sup_{B^+_{\rho}} |f| \left(\rho^{2-n} \int_{B^+_{\rho}} |Dv|^2 \right)^{1/2}
		\leq C \rho^{\kappa-1} D(\rho)^{1/2} \nonumber \\&\leq \frac{C}{\rho} \,(H(\rho) + D(\rho)) , \\
	\label{Alm mono eqn7} \left| \rho^{1-n} \int_{(\partial B_{\rho})^+} f \,v \right|
		&\leq C \sup_{B^+_{\rho}} |f| \left(\rho^{1-n} \int_{(\partial B_{\rho})^+} v^2 \right)^{1/2}
		\leq C \rho^{\kappa-1} H(\rho)^{1/2} \leq \frac{C}{\rho} \,H(\rho) ,
\end{align}
for all $\alpha < \rho < \beta$.  By~\cite[Lemma~4.2]{DJ},
\begin{equation*}
	\rho^{1-n} \int_{B'_{\rho}} v^2 \leq C H(\rho) + C D(\rho)
\end{equation*}
for all $\alpha < \rho < \beta$ and some constant $C = C(n) \in (0,\infty)$.  Moreover, since $p \geq 2$ we have
\begin{equation} \label{Alm mono eqn8}
	\rho^{1-n} \int_{B'_{\rho}} (k_+ (v^+)^p + k_- (v^-)^p) + \rho^{-n} \int_{B'_{\rho}} \frac{1}{\mu} a^{ij} x_i \,(D_j k_+ (v^+)^p + D_j k_- (v^-)^p)
		\leq C H(\rho) + C D(\rho)
\end{equation}
for all $\alpha < \rho < \beta$.  By \eqref{Alm mono G bounded} and \eqref{variation eqn6},
\begin{equation} \label{Alm mono eqn9}
	\left| H'(\rho) - \frac{2}{\rho} \,I(\rho) \,\right| \leq C H(\rho)
\end{equation}
for all $\alpha < \rho < \beta$ and some constant $C \in (0,\infty)$ depending only on $n$, $\lambda$, and $\|Da^{ij}\|_{L^{\infty}(B^+_1)}$.  By \eqref{variation eqn4} and \eqref{variation eqn5},
\begin{align*}
	I'(\rho) \geq&\,  2 \rho^{2-n} \int_{(\partial B_{\rho})^+} \frac{1}{\mu} \bigg( a^{ij} D_j v \,\frac{x_i}{\rho} \bigg)^2
		- 2 \rho^{1-n} \int_{B^+_{\rho}} \frac{1}{\mu} \,a^{ij} D_j v \,f \,x_i + (2-n) \rho^{1-n} \int_{B^+_{\rho}} v \,f \\&
		+ \rho^{2-n} \int_{(\partial B_{\rho})^+} v \,f + \frac{2}{p} \,\rho^{1-n} \int_{B'_{\rho}} \frac{1}{\mu} a^{ij} x_i \,(D_j k_+ (v^+)^p + D_j k_- (v^-)^p) \\&
		- \frac{(p-2)(n-2)}{p} \,\rho^{1-n} \int_{B'_{\rho}} (k_+ \,(v^+)^p + k_- \,(v^-)^p) \\&
		+ \frac{p-2}{p} \,\rho^{2-n} \int_{\partial B'_{\rho}} (k_+ \,(v^+)^p + k_- \,(v^-)^p) - C D(\rho)
\end{align*}
for all $0 < \rho < 1$.  Thus, by keeping in mind that $p \geq 2$ and applying \eqref{Alm mono eqn5}-\eqref{Alm mono eqn8}, we obtain
\begin{equation} \label{Alm mono eqn10}
	I'(\rho) \geq 2 \rho^{2-n} \int_{(\partial B_{\rho})^+} \frac{1}{\mu} \bigg( a^{ij} D_j v \,\frac{x_i}{\rho} \bigg)^2 - C H(\rho) - C D(\rho)
\end{equation}
for $\mathcal{L}^1$-a.e.~$\alpha < \rho < \beta$.  By \eqref{Alm mono eqn10} and \eqref{Alm mono eqn9},
\begin{align} \label{Alm mono eqn11}
	\Phi'(\rho) =&\, \frac{H(\rho) \,I'(\rho) - H'(\rho) \,I(\rho)}{H(\rho)^2} \nonumber \\
		\geq&\, \frac{2\rho^{3-2n}}{H(\rho)^2} \left( \left( \int_{(\partial B_{\rho})^+} \mu \,v^2 \right)
			\left( \int_{(\partial B_{\rho})^+} \frac{1}{\mu} \bigg( a^{ij} D_j v \,\frac{x_i}{\rho} \bigg)^2  \right)
			- \left( \int_{(\partial B_{\rho})^+} a^{ij} v \,D_j v \,\frac{x_i}{\rho} \right)^2 \right) \nonumber
		\\& - C \,\frac{H(\rho) + D(\rho) + |I(\rho)|}{H(\rho)} .
\end{align}
Using \eqref{Alm mono eqn5} and \eqref{Alm mono eqn8}, \eqref{variation eqn4} gives us
\begin{equation} \label{Alm mono eqn12}
	(1 - C \rho) \,D(\rho) - C \rho H(\rho) \leq I(\rho) \leq (1 + C \rho) \,D(\rho) + C \rho H(\rho)
\end{equation}
for all $\alpha < \rho < \beta$.  Thus provided $\rho \leq \rho_0$ for $\rho_0 < 1/(2C)$, $D(\rho) \leq 2I(\rho) + 2C \rho H(\rho)$.  Thus if $\alpha < \rho < \beta$ is such that $I(\rho) \geq 0$, then \eqref{Alm mono eqn12} gives us
\begin{equation} \label{Alm mono eqn13}
	\Phi'(\rho) \geq -C \,(\Phi(\rho) + 1) .
\end{equation}
If instead $I(\rho) < 0$, then $|I(\rho)| = -I(\rho) \leq C \rho H(\rho)$ and $D(\rho) \leq C \rho H(\rho)$ for all $\alpha < \rho < \beta$.  Thus \eqref{Alm mono eqn11} gives us $\Phi(\rho) \geq -C\rho$ and $\Phi'(\rho) \geq -C$.  In either case we conclude that \eqref{Alm mono eqn13} holds true for all $\alpha < \rho < \beta$.  Hence, letting $C$ be a constant such that \eqref{Alm mono eqn13} holds true and $\Phi(\rho) \geq -C\rho$ for all $\alpha < \rho < \beta$, we have that
\begin{equation} \label{Alm mono eqn14}
	\frac{d}{d\rho} \Big( e^{2C\rho} \Phi(\rho) + 2C e^{2C} \rho \Big) \geq C e^{2C \rho} (\Phi(\rho) - 1) + 2C e^{2C}
		\geq -C e^{2C \rho} (C\rho + 1) + 2C e^{2C} \geq 0
\end{equation}
for all $\alpha < \rho < \beta$ provided $\rho \leq \rho_0$ for $\rho_0 < 1/C$.  This gives us \eqref{Alm mono concl} with $2C$ in place of $C$.

Now let $0 < \sigma < \rho \leq \rho_0$.  We want to show that \eqref{Alm mono concl} holds true, where we use the convention that $\Phi(\tau) = \kappa$ if $K(\tau) \leq \tau^{2\kappa}$ for $0 < \tau \leq \rho_0$ (see Remark~\ref{Almgren monotonicity rmk}).  In the case that $\sigma$ and $\rho$ belong to the same connected component of the open set $\{ \tau \in (0,\rho_0) : K(\tau) > \tau^{2\kappa} \}$, we showed above that \eqref{Alm mono concl} holds true.  Suppose that $\sigma$ belongs to the connected component $(\alpha,\beta)$ of $\{ \tau \in (0,\rho_0) : K(\tau) > \tau^{2\kappa} \}$.  Then $K(\tau) > \tau^{2\kappa}$ for all $\alpha < \tau < \beta$ and $K(\tau) = \tau^{2\kappa}$ for $\tau = \alpha, \beta$.  Hence,
\begin{equation} \label{Alm mono eqn15}
	\Phi(\alpha^+) \geq \kappa, \quad\quad \Phi(\beta^-) \leq \kappa .
\end{equation}
If $\rho$ belongs to the connected component $(\gamma,\delta)$ of $\{ \tau \in (0,\rho_0) : K(\tau) > \tau^{2\kappa} \}$ we have $\alpha < \sigma < \beta \leq \gamma < \rho < \delta$ and in addition to \eqref{Alm mono eqn15} we have $\Phi(\gamma^+) \geq \kappa$.  Thus
\begin{equation*}
	e^{C\sigma} \,\Phi(\sigma) + C e^C \sigma \leq e^{C\beta} \,\Phi(\beta^-) + Ce^C \beta \leq e^{C\beta} \,\kappa + C e^C \beta
		\leq e^{C\gamma} \,\Phi(\gamma^+) + Ce^C \gamma \leq e^{C\rho} \,\Phi(\rho) + Ce^C \rho ,
\end{equation*}
proving \eqref{Alm mono concl}.  If instead $K(\rho) \leq \rho^{2\kappa}$, then we have $\alpha < \sigma < \beta \leq \rho$ and thus using \eqref{Alm mono eqn15}
\begin{equation*}
	e^{C\sigma} \,\Phi(\sigma) + C e^C \sigma \leq e^{C\beta} \,\Phi(\beta^-) + Ce^C \beta \leq e^{C\beta} \,\kappa + Ce^C \beta
		\leq e^{C\rho} \,\kappa + C e^C \rho = e^{C\rho} \,\Phi(\rho) + Ce^C \rho ,
\end{equation*}
again proving \eqref{Alm mono concl}.  Similarly, if $K(\sigma) \leq \sigma^{2\kappa}$ and $K(\rho) > \rho^{2\kappa}$, \eqref{Alm mono concl} holds true.  In the case that $K(\sigma) \leq \sigma^{2\kappa}$ and $K(\rho) \leq \rho^{2\kappa}$, then $\Phi(\sigma) = \Phi(\rho) = \kappa$ and thus \eqref{Alm mono concl} holds true.
\end{proof}

\section{Growth rates at free boundary points} \label{sec:growth sec}

Assume that $2 \leq p < \infty$.  Let $u \in W^{1,2}(B^+_1(0))$ be a weak solution to \eqref{main ptop} and $0 \in \Sigma(u)$.  For each positive integer $\kappa$, let $v_{\kappa}$ be as in \eqref{frequency v} and $\Phi_{v_{\kappa}}$ be as in \eqref{eq.Phi}.  By Theorem~\ref{Almgren monotonicity thm}, for each integer $\kappa$ the limit
\begin{equation*}
	\Phi_{v_{\kappa}}(0^+) = \lim_{\rho \downarrow 0} \Phi_{v_{\kappa}}(\rho)
\end{equation*}
exists.  However, $\Phi_{v_{\kappa}}(0^+)$ might depend on $\kappa$.  In Lemma~\ref{frequency consistency lemma} we show that $\Phi_{v_{\kappa}}(0^+) = \min\{\Phi_{v_{\lambda}}(0^+) , \kappa\}$ for integers $\kappa < \lambda$.  It follows in Definition~\ref{Almgren freq defn} that we can define the Almgren frequency $\mathcal{N}_u(0)$ of $u$ at the origin which is independent of $\kappa$, though possibility infinite.

\begin{lemma} \label{L2 growth lemma}
Assume Hypothesis~\ref{almgren hyp} and let $v \in C^1(B^+_1 \cup B'_1)$ be a solution to \eqref{ptop freq eqn}.  Let  $H_v$ be as in \eqref{eq.H} and $K_v$ be as in \eqref{eq.K}. There exists $\rho_0 > 0$ depending only on $n$, $p$, $\lambda$, $\|a^{ij}\|_{C^{0,1}(B^+_1)}$, $\sup_{B^+_1} |x|^{1-\kappa} |f|$, $\|k_+\|_{C^{0,1}(B'_1)}$, $\|k_-\|_{C^{0,1}(B'_1)}$, and $\|v\|_{L^{\infty}(B_1)}$ such that the following holds true.
%Let $\rho_0$ be as in Theorem~\ref{Almgren monotonicity thm}.
\begin{enumerate}
	\item[(a)]  If $0 < \alpha < \sigma < \rho \leq \rho_0$ such that $K_v(\tau) \geq \tau^{2\kappa}$ for all $\alpha \leq \tau \leq \rho$, then
	\begin{gather}
		\left(\frac{\sigma}{\rho}\right)^{2 e^{C\rho} \Phi_v(\rho) + 2 Ce^C\rho} K_v(\rho) \leq K_v(\sigma)
			\leq e^{2C \,(\Phi_v(\alpha) + e^C) \,(\rho-\sigma)} \left(\frac{\sigma}{\rho}\right)^{2 \Phi_v(\alpha)} K_v(\rho) , \label{L2 growth concl1} \\
		\frac{1}{C} \left(\frac{\sigma}{\rho}\right)^{2 e^{C\rho} \Phi_v(\rho) + 2 Ce^C\rho} H_v(\rho) \leq H_v(\sigma)
			\leq C e^{2C \,(\Phi_v(\alpha) + e^C) \,(\rho-\sigma)} \left(\frac{\sigma}{\rho}\right)^{2 \Phi_v(\alpha)} H_v(\rho) , \label{L2 growth concl2}
	\end{gather}
	where $C \in (0,\infty)$ is a constant depending only on $n$, $p$, $\lambda$, $\|a^{ij}\|_{C^{0,1}(B^+_1)}$, $\sup_{B^+_1} |x|^{1-\kappa} |f|$, $\|k_+\|_{C^{0,1}(B'_1)}$, $\|k_-\|_{C^{0,1}(B'_1)}$, and $\|v\|_{L^{\infty}(B_1)}$.

	\item[(b)]  If $0 < \sigma < \rho \leq \rho_0$ such that $K_v(\tau) \geq \tau^{2\kappa}$ for all $0 < \tau \leq \rho$, then
	\begin{align} \label{L2 growth concl3}
		\frac{1}{C} \left(\frac{\sigma}{\rho}\right)^{2 e^{C\rho} \Phi_v(\rho) + 2 Ce^C\rho} &\rho^{-n} \int_{B^+_{\rho}} v^2
			\leq \sigma^{-n} \int_{B^+_{\sigma}} v^2 \nonumber \\&\leq C e^{2C \,(\Phi_v(0^+) + e^C) \,(\rho-\sigma)}
			\left(\frac{\sigma}{\rho}\right)^{2 \Phi_v(0^+)} \rho^{-n} \int_{B^+_{\rho}} v^2 ,
	\end{align}
	where $C \in (0,\infty)$ is a constant depending only on $n$, $p$, $\lambda$, $\|a^{ij}\|_{C^{0,1}(B^+_1)}$, $\sup_{B^+_1} |x|^{1-\kappa} |f|$, $\|k_+\|_{C^{0,1}(B'_1)}$, $\|k_-\|_{C^{0,1}(B'_1)}$, and $\|v\|_{L^{\infty}(B_1)}$.
\end{enumerate}
\end{lemma}

\begin{proof}
Throughout this proof we let $C \in (0,\infty)$ denote constants depending only on $n$, $p$, $\lambda$, $\|a^{ij}\|_{C^{0,1}(B^+_1)}$, $\sup_{B^+_1} |x|^{1-\kappa} |f|$, $\|k_+\|_{C^{0,1}(B'_1)}$, $\|k_-\|_{C^{0,1}(B'_1)}$, and $\|v\|_{L^{\infty}(B_1)}$.  Let $\rho_0 \in (0,1/2)$ be a small constant to be later determined.  To see (a), suppose that $0 < \alpha \leq \sigma < \rho \leq \rho_0$ such that $K_v(\tau) > \tau^{2\kappa}$ for all $\alpha \leq \tau \leq \rho$.  Then $\Phi(\tau) = \tau \,K'_v(\tau) / (2 \,K_v(\tau))$ for all $\alpha \leq \tau \leq \rho$.  Hence by Theorem~\ref{Almgren monotonicity thm},
\begin{equation} \label{L2 growth eqn1}
	e^{C\alpha} \,\Phi_v(\alpha) + C e^C\alpha \leq e^{C\tau} \,\frac{\tau K'_v(\tau)}{2K_v(\tau)} + C e^C\tau \leq e^{C\rho} \,\Phi_v(\rho) + C e^C \rho
\end{equation}
for all $\alpha \leq \tau \leq \rho$.  Note that by \eqref{Alm mono eqn2} and \eqref{Alm mono eqn12} we have $\Phi_v(\alpha) \geq -C \alpha$ for some constant $C \in (0,\infty)$.  Thus, assuming $\rho_0$ is sufficiently small, $e^{C\alpha} \,\Phi_v(\alpha) + C e^C\alpha \geq 0$ and thus $e^{C\rho} \,\Phi_v(\rho) + C e^C\rho \geq 0$ (for $C$ as in \eqref{L2 growth eqn1}).
If $\Phi_v(\alpha) \geq 0$, then by rearranging terms in \eqref{L2 growth eqn1} and using $1 - C\tau \leq e^{-C\tau} \leq 1$ yield
\begin{equation} \label{L2 growth eqn2}
	(1 - C\tau) \,\Phi_v(\alpha) - C e^C\tau \leq \frac{\tau K'_v(\tau)}{2K_v(\tau)} \leq e^{C\rho} \,\Phi_v(\rho) + C e^C\rho
\end{equation}
for all $\alpha \leq \tau \leq \rho\leq \rho_0$, provided $\rho_0$ is small enough that $1 - C\rho_0 \geq 0$, where $C$ is as in \eqref{L2 growth eqn2}.
If instead $\Phi_v(\alpha) < 0$, then rearranging terms in \eqref{L2 growth eqn1} gives us $\tau \,K'_v(\tau) / (2 \,K_v(\tau)) \geq -C e^C \tau$ for some constant $C \in (0,\infty)$.  Hence \eqref{L2 growth eqn2} holds true for all $\alpha \leq \tau \leq \rho \leq \rho_0$ provided $\rho_0$ is small enough that $1 - C\rho_0 \geq 0$, where $C$ is as in \eqref{L2 growth eqn2}.
Multiplying \eqref{L2 growth eqn2} by $2/\tau$ and integrating over $\tau \in [\sigma, \rho]$ gives us \eqref{L2 growth concl1}.  By using \eqref{Alm mono eqn1} in \eqref{L2 growth concl1} gives us \eqref{L2 growth concl2}.  Integrating \eqref{L2 growth concl2} with $\alpha = 0^+$ and using \eqref{mu positive} gives us \eqref{L2 growth concl3}.
\end{proof}

\begin{lemma} \label{frequency consistency lemma}
Let $2 \leq p < \infty$, and let $\kappa$ and $\ell$ be positive integers with $\ell\geq \kappa$.  Let $a^{ij} \in C^{\ell -1,1}(B^+_1(0))$ be such that $a^{ij} = a^{ji}$ on $B^+_1(0)$, the ellipticity condition \eqref{ellipticity hyp} holds  for some constants $0 < \lambda \leq \Lambda < \infty$, and the normalization assumptions \eqref{identity_origin} and \eqref{regularity assumption} are satisfied.  Let $k_+,k_- \in C^{0,1}(B'_1(0))$ with $k_+,k_- \geq 0$, and let $h \in C^{\ell,1}(B'_1(0))$.   Let $u \in W^{1,2}(B^+_1(0))$ be a weak solution to \eqref{main ptop} and $0 \in \Sigma(u)$.  Then, if $v_{\kappa}$ is as in \eqref{frequency v},
\begin{equation} \label{freq consistency concl2}
	\Phi_{v_{\kappa}}(0^+) \leq \kappa .
\end{equation}
In addition
\begin{equation} \label{freq consistency concl1}
	\Phi_{v_{\kappa}}(0^+) = \min\{\Phi_{v_{\ell}}(0^+),\kappa\} ,
\end{equation}
where  $v_{\ell}$ is as in \eqref{frequency v} with $\ell$ in place of $\kappa$.

\end{lemma}

\begin{proof}
Let's first take a positive integer $\kappa$ and show \eqref{freq consistency concl2}.  Let $v = v_{\kappa}$ be as in \eqref{frequency v} and $K_{v_{\kappa}}$ be as in \eqref{eq.K}.  Notice that if there exists $\rho_j \rightarrow 0^+$ such that $K_{v_{\kappa}}(\rho_j) < \rho_j^{2\kappa}$, then $\Phi_{v_{\kappa}}(\rho_j) = \kappa$ and thus letting $j \rightarrow \infty$ gives us $\Phi_{v_{\kappa}}(0^+) = \kappa$.  Hence we may assume that there exists $\delta > 0$ such that
\begin{equation} \label{freq consistency eqn1}
	K_{v_{\kappa}}(\rho) \geq \rho^{2\kappa} \text{ for all } 0 < \rho < \delta .
\end{equation}
By \eqref{freq consistency eqn1} we can apply~\eqref{L2 growth concl1} to obtain
\begin{equation*}
	\sigma^{2\kappa} \leq K_{v_{\kappa}}(\sigma) \leq C \left(\frac{\sigma}{\rho}\right)^{2 \Phi_{v_{\kappa}}(0^+)} K_{v_{\kappa}}(\rho)
\end{equation*}
for all $0 < \sigma < \rho$ sufficiently small and some constant $C \in (0,\infty)$ independent of $\sigma$ and $\rho$.   Hence \eqref{freq consistency concl2} must hold true.

Now take positive integers $\kappa < \ell$.  Let $v_{\kappa}$ be as in \eqref{frequency v} and $v_{\ell}$ is as in \eqref{frequency v} with $\ell$ in place of $\kappa$.  We want to show \eqref{freq consistency concl1}.  We will separately consider the cases $\Phi_{v_{\ell}}(0^+) < \kappa$ and $\kappa \leq \Phi_{v_{\ell}}(0^+) \leq \ell$.

First suppose that $\Phi_{v_{\ell}}(0^+) < \kappa$.  Let $0 < \varepsilon < \kappa - \Phi_{v_{\ell}}(0^+)$.  As we discussed above, since $\Phi_{v_{\ell}}(0^+) < \ell$ there exists $\delta > 0$ such that
\begin{equation} \label{freq consistency eqn2}
	K_{v_{\ell}}(\rho) \geq \rho^{2\ell} \text{ for all } 0 < \rho < \delta.
\end{equation}
Thus we can apply \eqref{L2 growth concl2} to obtain
\begin{equation} \label{freq consistency eqn3}
	\frac{1}{C} \left(\frac{\sigma}{\rho}\right)^{2\Phi_{v_{\ell}}(0^+) + 2\varepsilon} H_{v_{\ell}}(\rho) \leq H_{v_{\ell}}(\sigma)
			\leq C \left(\frac{\sigma}{\rho}\right)^{2 \Phi_{v_{\ell}}(0^+)} H_{v_{\ell}}(\rho)
\end{equation}
for all sufficiently small $0 < \sigma < \rho$ and some constant $C \in (0,\infty)$ independent of $\sigma$ and $\rho$.  Let $\overline{h}_{\kappa}$ be as in Lemma~\ref{h taylor thm} and $\overline{h}_{\ell}$ be as in Lemma~\ref{h taylor thm} with $\ell$ in place of $\kappa$.  Since $\overline{h}_{\kappa}$ is the unique degree $\kappa$ polynomial such that \eqref{h taylor concl} holds true and \eqref{h taylor concl} holds true with the polynomial $\overline{h}_{\ell}$ in place of $\overline{h}_{\kappa}$, $\overline{h}_{\kappa}$ and $\overline{h}_{\ell}$ must be equal up to terms of degree $\leq \kappa$.  In particular,
\begin{equation*}
	\big| \overline{h}_{\kappa}(x) - \overline{h}_{\ell}(x) \big| \leq C |x|^{\kappa+1}
\end{equation*}
for each $x \in B^+_1$, where $C \in (0,\infty)$ is a constant depending only on $n$, $\ell$, $\|a^{ij}\|_{C^{\ell-1,1}(B^+_1)}$, and $\|h\|_{C^{\ell,1}(B'_1)}$.  It follows from the definition of $v_{\kappa}$ and $v_{\ell}$ in \eqref{frequency v} that
\begin{equation} \label{freq consistency eqn4}
	|v_{\kappa}(x) - v_{\ell}(x)| \leq C |x|^{\kappa+1}
\end{equation}
for all $x \in B^+_1$, where $C \in (0,\infty)$ is a constant depending only on $n$, $\ell$, $\|a^{ij}\|_{C^{\kappa-1,1}(B^+_1)}$, and $\|h\|_{C^{\kappa,1}(B'_1)}$.  Hence by \eqref{freq consistency eqn3},
\begin{equation} \label{freq consistency eqn6}
	\frac{1}{2C} \left(\frac{\sigma}{\rho}\right)^{2\Phi_{v_{\ell}}(0^+) + 2\varepsilon} H_{v_{\ell}}(\rho) - C \sigma^{2(\kappa+1)}
		\leq H_{v_{\kappa}}(\sigma) \leq 2C \left(\frac{\sigma}{\rho}\right)^{2 \Phi_{v_{\ell}}(0^+)} H_{v_{\ell}}(\rho) + 2C \sigma^{2(\kappa+1)}
\end{equation}
for all sufficiently small $0 < \sigma < \rho$ sufficiently small and some constant $C \in (1,\infty)$ independent of $\sigma$ and $\rho$.  Note that by \eqref{Alm mono eqn1} and \eqref{freq consistency eqn2}, $H_{v_{\ell}}(\rho) > 0$.  Thus by fixing $\rho$ and applying \eqref{Alm mono eqn1} and \eqref{freq consistency eqn6},
\begin{equation*}	
	K_{v_{\kappa}}(\sigma) \geq \frac{1}{C} \,H_{v_{\kappa}}(\sigma)
		> \frac{1}{2C^2} \left(\frac{\sigma}{\rho}\right)^{2\Phi_{v_{\ell}}(0^+) + 2\varepsilon} H_{v_{\ell}}(\rho) - C \sigma^{2(\kappa+1)}
		> \sigma^{2\kappa}
\end{equation*}
for all sufficiently small $\sigma > 0$, where $C \in (1,\infty)$ denotes constants independent of $\sigma$ and $\rho$.  Thus we can apply \eqref{L2 growth concl2} to obtain
\begin{equation} \label{freq consistency eqn7}
	\frac{1}{C} \left(\frac{\sigma}{\rho}\right)^{2\Phi_{v_{\kappa}}(0^+) + 2\varepsilon} H_{v_{\kappa}}(\rho) \leq H_{v_{\kappa}}(\sigma)
			\leq C \left(\frac{\sigma}{\rho}\right)^{2 \Phi_{v_{\kappa}}(0^+)} H_{v_{\kappa}}(\rho) .
\end{equation}
for all sufficiently small $0 < \sigma < \rho$ and some constant $C \in (1,\infty)$ independent of $\sigma$ and $\rho$.  By \eqref{freq consistency eqn6} and \eqref{freq consistency eqn7},
\begin{equation*}
	\frac{1}{2C} \left(\frac{\sigma}{\rho}\right)^{2\Phi_{v_{\ell}}(0^+) + 2\varepsilon} H_{v_{\ell}}(\rho) - C \sigma^{2(\kappa+1)}
		\leq H_{v_{\kappa}}(\sigma)
			\leq C \left(\frac{\sigma}{\rho}\right)^{2 \Phi_{v_{\kappa}}(0^+)} H_{v_{\kappa}}(\rho)
\end{equation*}
for all sufficiently small $0 < \sigma < \rho$ sufficiently small and thus $\Phi_{v_{\kappa}}(0^+) \leq \Phi_{v_{\ell}}(0^+) + \varepsilon$.  On the other hand, by \eqref{freq consistency eqn6} and \eqref{freq consistency eqn7},
\begin{equation*}
	\frac{1}{C} \left(\frac{\sigma}{\rho}\right)^{2\Phi_{v_{\kappa}}(0^+) + 2\varepsilon} H_{v_{\kappa}}(\rho)
		\leq H_{v_{\kappa}}(\sigma)
			\leq 2C \left(\frac{\sigma}{\rho}\right)^{2 \Phi_{v_{\ell}}(0^+)} H_{v_{\ell}}(\rho) + C \sigma^{2(\kappa+1)}
\end{equation*}
for all sufficiently small $0 < \sigma < \rho$ sufficiently small and thus $\Phi_{v_{\ell}}(0^+) \leq \Phi_{v_{\kappa}}(0^+) + \varepsilon$.  In other words,
\begin{equation} \label{freq consistency eqn8}
	\Phi_{v_{\ell}}(0^+) - \varepsilon \leq \Phi_{v_{\kappa}}(0^+) \leq \Phi_{v_{\ell}}(0^+) + \varepsilon.
\end{equation}
Letting $\varepsilon \rightarrow 0^+$ gives us $\Phi_{v_{\kappa}}(0^+) = \Phi_{v_{\ell}}(0^+)$.

Suppose instead that $\kappa \leq \Phi_{v_{\ell}}(0^+) \leq \ell$.  Note that if $\Phi_{v_{\kappa}}(0^+)=\kappa$, then there is nothing to show. Assume thus $\Phi_{v_{\kappa}}(0^+)<\kappa$, and let $0<\varepsilon<\kappa - \Phi_{v_{\kappa}}(0^+)$. We can then repeat the above arguments, with the roles of $\kappa$ and $\ell$ interchanged, to show $\Phi_{v_\kappa}(0^+)=\Phi_{v_\ell}(0^+)$. However, since $\Phi_{v_\kappa}(0^+)<\kappa\leq\Phi_{v_\ell}(0^+)$, we have a reached a contradiction and the proof is complete.
\end{proof}

In light of Lemma~\ref{frequency consistency lemma}, we can define the Almgren frequency $\mathcal{N}_u(0)$ of $u$ at the origin as follows.

\begin{definition} \label{Almgren freq defn} {\rm
Let $u \in W^{1,2}(B^+_1(0))$ be a solution to \eqref{main ptop} and $x_0 \in \Sigma(u)$ as above.  Translate $x_0$ to the origin and then apply a linear change of variables so that $a^{ij}(0) = \delta_{ij}$.  Then for each positive integer $\kappa$ define $v_{\kappa}$ as in \eqref{frequency v}.  If there exists a positive integer $\kappa$ such that $\Phi_{v_{\kappa}}(0^+) < \kappa$, where $v_{\kappa}$ is as in \eqref{frequency v}, then we define $\mathcal{N}_u(x_0) = \Phi_{v_{\kappa}}(0^+)$.  Otherwise, we define $\mathcal{N}_u(x_0) = +\infty$.
} \end{definition}

It follows from Lemma~\ref{frequency consistency lemma} that $\mathcal{N}_u(x_0)$ is well-defined and $\Phi_{v_{\kappa}}(0^+) = \min\{\mathcal{N}_u(x_0),\kappa\}$ for each positive integer $\kappa$.

\section{Blow-ups} \label{sec:tangent fn sec}

Assume Hypothesis~\ref{almgren hyp} and let $v \in C^1(B^+_1(0) \cup B'_1(0))$ be a solution to \eqref{ptop freq eqn}.  Suppose $\Phi_v(0^+) < \infty$.  We want to determine the asymptotic behavior of $v$ at the origin.  Let
\begin{equation} \label{v rescaled1}
	v_{\rho}(x) = \frac{v(\rho x)}{H_v(\rho)^{1/2}}
\end{equation}
for each $0 < \rho \leq 1$ and $x \in B^+_{1/(2\rho)} \cup B'_{1/(2\rho)}$, where $H_v$ is as in \eqref{eq.H}.  A \emph{blow-up} of $v$ at the origin is a function $v^* \in C^1(\overline{\mathbb{R}^n_+})$ such that for some sequence $\rho_{\ell} \rightarrow 0^+$ we have $v_{\rho_{\ell}} \rightarrow v^*$ in the $C^1$-topology on compact subsets of $\overline{\mathbb{R}^n_+}$.

\begin{theorem}\label{tangent thm}
Assume Hypothesis~\ref{almgren hyp} and let $v \in C^1(B^+_1(0) \cup B'_1(0))$ be a solution to \eqref{ptop freq eqn} with $\Phi_v(0^+) < \infty$.  For each sequence $\rho_{\ell} \rightarrow 0^+$, there exists a subsequence $\{\rho_{\ell'}\} \subset \{\rho_{\ell}\}$ and a non-zero function $v^* \in C^1(\overline{\mathbb{R}^n_+})$ such that
\begin{equation} \label{tangent concl1}
	v_{\rho_{\ell'}} \rightarrow v^*
\end{equation}
in $C^1(B^+_{\sigma} \cup B'_{\sigma})$ for all $\sigma \in (0,\infty)$, where $v_{\rho_{\ell'}}$ is as in \eqref{v rescaled1} with $\rho = \rho_{\ell'}$.  Furthermore, the even reflection of $v^*$ across $\{x_n = 0\}$ (also denoted by $v^*$) is a homogeneous harmonic polynomial of degree $\Phi_v(0^+)$.
\end{theorem}

Before proving Theorem~\ref{tangent thm}, let's first prove the following estimate:

\begin{lemma}\label{elliptic est lemma}
Assume Hypothesis~\ref{almgren hyp} and let $v \in C^1(B^+_1(0) \cup B'_1(0))$ be a solution to \eqref{ptop freq eqn}.  Then
\begin{align}\label{elliptic est concl}
	&\sup_{B^+_{\sigma/2}(0)} |v| + \sigma \sup_{B^+_{\sigma/2}(0)} |Dv| + \sigma^{3/2} [Dv]_{1/2,B^+_{\sigma/2}(0)}
	\leq C \left( \sigma^{-n/2} \|v\|_{L^2(B^+_{\sigma}(0))} + \sigma^2 \sup_{B^+_{\sigma}(0)} |f| \right)
\end{align}
for each $\sigma \in (0,1)$ and some constant $C \in (0,\infty)$ depending only on $n$, $p$, $\lambda$, $\|a^{ij}\|_{C^{0,1}(B^+_1)}$, $\sup_{B^+_1} |x|^{1-\kappa} |f|$, $\|k_+\|_{C^{0,1}(B'_1)}$, $\|k_-\|_{C^{0,1}(B'_1)}$, and $\|v\|_{L^{\infty}(B_1)}$.
\end{lemma}

\begin{proof}
Throughout the proof we will denote by $C \in (0,\infty)$ constants depending only on $n$, $p$, $\lambda$, $\|a^{ij}\|_{C^{0,1}(B^+_1)}$, $\sup_{B^+_1} |x|^{1-\kappa} |f|$, $\|k_+\|_{C^{0,1}(B'_1)}$, $\|k_-\|_{C^{0,1}(B'_1)}$, and $\|v\|_{L^{\infty}(B_1)}$. Without loss of generality assume $\sigma = 1$.  First we claim that for each $\delta > 0$ there exists $\tau_0 = \tau_0(p, \|v\|_{C^0(B_1)}) \in (0,1/2]$ such that if $B_{\tau}(z) \subset B_1(0)$ with $0 < \tau \leq \tau_0$ then
\begin{align}\label{elliptic est eqn1}
	&\,\sup_{B_{\tau/16}(z) \cap B^+_1(0)} |v| + \tau \sup_{B_{\tau/16}(z) \cap B^+_1(0)} |Dv| + \tau^{3/2} [Dv]_{1/2,B_{\tau/16}(z) \cap B^+_1(0)} \nonumber
	\\ \leq&\, C \delta \left( \sup_{B_{\tau}(z) \cap B^+_1(0)} |v| + \tau^{1/2} [v]_{1/2,B_{\tau}(z) \cap B^+_1(0)} \right) \nonumber
		\\& + C \left( \tau^{-n/2} \|v\|_{L^2(B_{\tau}(z) \cap B^+_1(0))} + \tau^2 \sup_{B_{\tau}(z) \cap B^+_1(0)} |f| \right) .
\end{align}
  To see this, fix $B_{\tau}(z) \subset B_1(0)$.  Notice that if $B_{\tau/8}(z) \subset B^+_1(0)$, then \eqref{elliptic est eqn1} holds true by standard elliptic estimates~\cite[Theorem~8.32]{GT}.  Otherwise, $(z',0) \in B_{\tau/8}(z) \cap B'_1(0)$, where $z = (z',z_n)$.  By $C^{1,\alpha}$-estimates for weak solutions to the oblique derivative problems~\cite[Proposition 5.53]{L13} applied to the solution $v$ to \eqref{ptop freq eqn},
\begin{align}\label{elliptic est eqn2}
	&\sup_{B^+_{\tau/4}(z',0)\cap B^+_1(0)} |v| + \tau \sup_{B^+_{\tau/4}(z',0)\cap B^+_1(0)} |Dv| + \tau^{3/2} [Dv]_{1/2,B^+_{\tau/4}(z',0)\cap B^+_1(0)} \nonumber\\
	& \leq C \left( \tau^{-n/2} \|v\|_{L^2(B^+_{\tau/2}(z',0)\cap B^+_1(0))}   + \tau \sup_{B'_{\tau/2}(z',0)\cap B^+_1(0)} |v|^{p-1}\right. \nonumber \\	& \left.	+ \tau^{3/2} \sup_{B'_{\tau/2}(z',0)\cap B^+_1(0)} |v|^{p-2} \,[v]_{1/2,B'_{\tau/2}(z',0)} + \tau^2 \sup_{B^+_{\tau/2}(z',0)\cap B^+_1(0)} |f| \right).
\end{align}
 Notice that $B_{\tau/16}(z) \subset B_{\tau/4}(z',0)$ and $B_{\tau/2}(z',0) \subset B_{\tau}(z)$.  Hence if we require that $0 < \tau \leq \tau_0$ where
\begin{equation*}
	\tau_0 \sup_{B_1} |v|^{p-2} \leq \delta ,
\end{equation*}
then \eqref{elliptic est eqn2} gives us \eqref{elliptic est eqn1}.  Now having shown \eqref{elliptic est eqn1}, it follows by interpolation that
\begin{align*}
	&\,\sup_{B_{\tau/16}(z) \cap B^+_1(0)} |v| + \tau \sup_{B_{\tau/16}(z) \cap B^+_1(0)} |Dv| + \tau^{3/2} [Dv]_{1/2,B_{\tau/16}(z) \cap B^+_1(0)} \nonumber
	\\&\leq C \delta \tau^{3/2} [Dv]_{1/2,B_{\tau}(z) \cap B^+_1(0)} + C \left( \tau^{-n/2} \|v\|_{L^2(B_{\tau}(z) \cap B^+_1(0))} + \tau^2 \sup_{B_{\tau}(z) \cap B^+_1(0)} |f| \right) .
\end{align*}
   Then using Lemma~2 of Section~2.8 of~\cite{SimonHarmonicMaps}, it follows that \eqref{elliptic est concl} holds true.
\end{proof}

%This is the version of the Lemma that we can use
%\begin{lemma} If
%\begin{equation}\label{simon1}
%\tau^k\varphi(B_{\tau/8}(z))\leq \tau^k\varepsilon\varphi(B_{\tau}(z))+\gamma
%\end{equation}
%whenever $B_{2\tau}(z)\subset B_\rho(y)$, then
%$$
%\rho^k\varphi(B_{\rho/2}(y))\leq C\gamma.
%$$
%\end{lemma}
%\begin{proof}
%Let $Q=\sup{\tau^k\varphi(B_{\tau}(z))}$, where the supremum is taken over all balls $B_{\tau}(z)$ such that $B_{2\tau}(z)\subset B_\rho(y)$. Take an arbitrary ball $B_{\tau}(z)$ such that $B_{2\tau}(z)\subset B_\rho(y)$, and cover it with balls $B_{\tau/16}(z_i)$, with $z_i\in B_{\tau}(z)$ and $B_{\tau}(z_i)\subset B_\rho(y)$. Let $S$ be the number of such balls. Apply \eqref{simon1} with $z_i$ instead of $z$, and $\tau/16$ instead of $\tau/8$. We obtain
%$$
%\left(\frac{\tau}{2}\right)^k\varphi(B_{\tau/16}(z_i))\leq \left(\frac{\tau}{2}\right)^k\varepsilon\varphi(B_{\tau/2}(z_i))+\gamma\leq Q\varepsilon+\gamma.
%$$
%Using the sub-additivity of $\varphi$ we infer
%$$
%\left(\frac{\tau}{2}\right)^k\varphi(B_{\tau}(z))\leq Q\varepsilon S+\gamma S.
%$$
%Taking the sup on the LHS yields
%$$
%Q\leq 2^kQ\varepsilon S+2^k\gamma S.
%$$
%Choosing $\varepsilon$ so that $2^k\varepsilon S\leq \frac{1}{2}Q,$ we obtain
%$$
%Q\leq 2^{k+1}\gamma S,
%$$
%which in turn implies
%$$
%\rho^k\varphi(B_{\rho/2}(y))\leq C\gamma.
%$$
%\end{proof}

\begin{proof}[Proof of Theorem~\ref{tangent thm}]
Fix $1 < \sigma < \infty$.  First we will show that $v_{\rho}$ is uniformly bounded in $C^{1,1/2}(B^+_{\sigma})$.  First let $\kappa$ be an integer such that $\Phi_v(0^+) < \kappa$, and notice that we thus have $K_v(\tau) > \tau^{2\kappa}$ for all sufficiently small $\tau > 0$.  Thus we may apply \eqref{mu positive} and \eqref{L2 growth concl2} to obtain
\begin{align*}
	\lambda \int_{B^+_{2\rho\sigma}} v^2
	&\leq \int_{B^+_{2\rho\sigma}} v^2 \mu
	= \int_0^{2\rho\sigma} \int_{(\partial B_{\tau})^+} v^2 \mu
	\\&\leq C \int_0^{2\rho\sigma} \left(\frac{\tau}{\rho}\right)^{n-1+2\Phi_v(0^+)} \int_{(\partial B_{\rho})^+} v^2 \mu \,d\tau
	\leq C \rho \int_{(\partial B_{\rho})^+} v^2 \mu
\end{align*}
for all sufficiently small $\rho > 0$ and some constant $C \in (0,\infty)$ independent of $\rho$.  By multiplying both sides by $(\rho\sigma)^{-n}$,
\begin{equation} \label{tangent eqn1}
	\sigma^{-n} \int_{B^+_{2\sigma}} v_{\rho}^2 = \frac{1}{(\rho\sigma)^n H_v(\rho)} \int_{B^+_{2\rho\sigma}} v^2 \leq C
\end{equation}
for all sufficiently small $\rho > 0$ and some constant $C \in (0,\infty)$ independent of $\rho$.  Also by \eqref{L2 growth concl2},
\begin{equation} \label{tangent eqn2}
	H_v(\rho) \geq c \rho^{2\kappa}
\end{equation}
for all sufficiently small $\rho > 0$ and some constant $c \in (0,\infty)$ independent of $\rho$.

By Lemma~\ref{elliptic est lemma} with $\rho\sigma$ in place of $\sigma$,
\begin{equation}\label{tangent eqn31}
	\sup_{B^+_{\rho\sigma}} |v| + \rho\sigma \sup_{B^+_{\rho\sigma}} |Dv| + (\rho\sigma)^{3/2} [Dv]_{1/2,B^+_{\rho\sigma}}
	\leq C \left( (\rho\sigma)^{-n/2} \|v\|_{L^2(B^+_{2\rho\sigma})} + (\rho\sigma)^2 \sup_{B^+_{\rho\sigma}} |f| \right) ,
\end{equation}
where $f$ is as in \eqref{frequency v} and $C \in (0,\infty)$ is a constant.  By dividing \eqref{tangent eqn31} by $H_v(\rho)^{1/2}$ and rescaling,
\begin{equation} \label{tangent eqn3}
	\sup_{B^+_{\sigma}} |v_{\rho}| + \sigma \sup_{B^+_{\sigma}} |Dv_{\rho}| + \sigma^{3/2} [Dv_{\rho}]_{1/2,B^+_{\sigma}}
	\leq C \left( \sigma^{-n/2} \|v_{\rho}\|_{L^2(B^+_{2\sigma})} + \sigma^2 \sup_{B^+_{\sigma}} |f_{\rho}| \right) ,
\end{equation}
where $f_{\rho} \in L^{\infty}(B^+_{\sigma})$ is defined by
\begin{equation} \label{tangent eqn4}
	f_{\rho}(x) = \frac{\rho^2 f(\rho x)}{H_v(\rho)^{1/2}}
\end{equation}
for each $x \in B^+_{\sigma}$.  By \eqref{f bound} and \eqref{tangent eqn2},
\begin{equation} \label{tangent eqn5}
	\sigma^2 \sup_{B^+_{\sigma}} |f_{\rho}|
	= \frac{ (\rho\sigma)^2 \sup_{B^+_{\rho\sigma}} |f| }{ H_v(\rho)^{1/2} }
		\leq \frac{ C \,(\rho\sigma)^{\kappa+1} }{ H_v(\rho)^{1/2} }
	 \leq C \rho
\end{equation}
for all sufficiently small $\rho$, where $C \in (0,\infty)$ denotes constants independent of $\rho$.  By bounding the right-hand side of \eqref{tangent eqn3} using \eqref{tangent eqn1} and \eqref{tangent eqn5},
\begin{equation} \label{tangent eqn6}
	\sup_{B^+_{\sigma}} |v_{\rho}| + \sigma \sup_{B^+_{\sigma}} |Dv_{\rho}| + \sigma^{3/2} [Dv_{\rho}]_{1/2,B^+_{\sigma}} \leq C
\end{equation}
for all sufficiently small $\rho$ and some constant $C \in (0,\infty)$ independent of $\rho$.  It now follows from \eqref{tangent eqn6} and the Arzela-Ascoli theorem that for each sequence $\rho_{\ell} \rightarrow 0^+$ there exists a subsequence $\{\rho_{\ell'}\} \subset \{\rho_{\ell}\}$ and a function $v^* \in C^1(\overline{\mathbb{R}^n_+})$ such that \eqref{tangent concl1} holds true with convergence in the $C^1(B^+_{\sigma} \cup B'_{\sigma})$ topology for all $\sigma \in (0,\infty)$.  Note that since
\begin{align*}
	\int_{(\partial B_1)^+} (v^*)^2
	&= \lim_{\ell \rightarrow \infty} \int_{(\partial B_1)^+} (v_{\rho_{\ell}}(x))^2 \,\mu(\rho_{\ell} \,x) \,dx
	\\&= \lim_{\ell \rightarrow \infty} \frac{1}{\rho_{\ell}^{n-1} H_v(\rho_{\ell})} \int_{(\partial B_{\rho_{\ell}})^+} (v(\rho_{\ell} x))^2 \,\mu(\rho_{\ell} \,x) \,dx = 1,
\end{align*}
$v^*$ is a non-zero function.

Again fix $1 < \sigma < \infty$.  Notice that by \eqref{ptop freq eqn}, for each sufficiently small $\rho > 0$, $v_{\rho}$ satisfies
\begin{gather}
	D_i (a^{ij}(\rho x) \,D_j v_{\rho}) = f_{\rho} \text{ in } B^+_{\sigma}, \nonumber \\
	D_n v_{\rho} = \psi_{\rho} \text{ on } B'_{\sigma} \label{tangent eqn7}
\end{gather}
where $f_{\rho}$ is as in \eqref{tangent eqn4} and $\psi_{\rho} \in C^{0,1}(B'_{\sigma})$ is defined by
\begin{equation*}
	\psi_{\rho}(x) = \frac{\rho}{H_v(\rho)^{1/2}} \left( k_+(\rho x) \,(v^+(\rho x))^{p-1} + k_-(\rho x) \,(v^-(\rho x))^{p-1} \right)
\end{equation*}
for each $x \in B'_{\sigma}$.  Since $a^{ij}$ is Lipschitz continuous, $a^{ij}(\rho x) \rightarrow a^{ij}(0) = \delta_{ij}$ uniformly on $B_{\sigma}$ as $\rho \rightarrow 0^+$.  By \eqref{tangent eqn5}, $f_{\rho} \rightarrow 0$ uniformly on $B_{\sigma}$ as $\rho \rightarrow 0^+$.  By \eqref{tangent eqn6},
\begin{equation*}
	\sup_{B'_{\sigma}} |\psi_{\rho}| \leq C\rho \sup_{B^+_1} (k_+ + k_-) \,|v|^{p-2} .
\end{equation*}
Therefore, $\psi_{\rho} \rightarrow 0$ uniformly on $B_{\sigma}$ as $\rho \rightarrow 0^+$.  Therefore, by letting $\rho = \rho_{\ell'} \rightarrow 0^+$ in \eqref{tangent eqn7} and noting that $\sigma$ is arbitrary,
\begin{gather*}
	\Delta v^* = 0 \text{ in } \mathbb{R}^n_+, \nonumber \\
	D_n v^* = 0 \text{ on } \mathbb{R}^{n-1} \times \{0\} .
\end{gather*}
Hence the even extension of $v^*$ across $\{x_n = 0\}$ is a harmonic function in $\mathbb{R}^n$.

Let $0 < \sigma < \tau < \infty$.  By rescaling \eqref{L2 growth concl3}, for each sufficiently small $\rho > 0$
\begin{align*}
	\frac{1}{C} \left(\frac{\sigma}{\tau}\right)^{2 e^{C\rho\tau} \Phi_v(\rho\tau) + 2Ce^C\rho\tau} \tau^{-n} \int_{B^+_{\tau}} |v_{\rho}|^2
		\leq \sigma^{-n} \int_{B^+_{\sigma}} |v_{\rho}|^2 \leq C \left(\frac{\sigma}{\tau}\right)^{2 \Phi_v(0^+)} \tau^{-n} \int_{B^+_{\tau}} |v_{\rho}|^2 .
\end{align*}
Letting $\rho = \rho_{\ell'} \rightarrow 0^+$,
\begin{align} \label{tangent eqn8}
	\frac{1}{C} \left(\frac{\sigma}{\tau}\right)^{2 \Phi_v(0^+)} \tau^{-n} \int_{B^+_{\tau}} |v^*|^2
		\leq \sigma^{-n} \int_{B^+_{\sigma}} |v^*|^2 \leq C \left(\frac{\sigma}{\tau}\right)^{2 \Phi_v(0^+)} \tau^{-n} \int_{B^+_{\tau}} |v^*|^2
\end{align}
for all $0 < \sigma < \tau < \infty$.  Recalling that the even extension of $v^*$ across $\{x_n = 0\}$ is a harmonic function in $\mathbb{R}^n$, it follows from \eqref{tangent eqn8} using the Liouville Theorem that the even extension of $v^*$ across $\{x_n = 0\}$ is also a homogeneous polynomial of degree $\Phi_v(0^+)$.  %In particular, $\Phi_v(0^+)$ must be a positive integer.
\end{proof}

\section{Weiss' monotonicity formula} \label{sec:Weiss sec}

Assume Hypothesis~\ref{almgren hyp} and let $v \in C^1(B^+_1(0) \cup B'_1(0))$ be a solution to \eqref{ptop freq eqn}.  For $\nu\in\mathbb{R}$, we introduce the \emph{Weiss function} $\mathcal{W}(\rho) = \mathcal{W}_{v,\nu}(\rho)$ as
\begin{equation} \label{Weiss fun}
	\mathcal{W}(\rho) = \frac{1}{\rho^{2\nu}} \,(I_v(\rho) - \nu H_v(\rho))
		= \rho^{1-n-2\nu} \int_{(\partial B_{\rho})^+} \left( a^{ij} v \,D_j v \,{x_i} - \mu\nu \,v^2 \right)
\end{equation}
for each $0 < \rho < 1$, where $I_v(\rho)$ and $H_v(\rho)$ are as in \eqref{eq.I} and \eqref{eq.H} respectively.

\begin{theorem}\label{Weiss mono}
Let $\nu \in \mathbb{R}$ and $\kappa\in\mathbb{N}$, with $\nu\leq \kappa$.  Assume Hypothesis~\ref{almgren hyp} and let $v \in C^1(B^+_1(0) \cup B'_1(0))$ be a solution to \eqref{ptop freq eqn} with $\Phi_v(0^+) \geq \nu$.  Let $\rho_0$ be as in Lemma~\ref{L2 growth lemma}.  For each $0 < \sigma < \rho \leq \rho_0/2$ such that $K_{v}(\tau) > \tau^{2\kappa}$ for all $0 < \tau \leq 2\rho$,
\begin{equation}\label{Weiss concl}
	\mathcal{W}(\sigma) + C \sigma \leq \mathcal{W}(\rho) + C \rho
\end{equation}
for some constant $C \in (0,\infty)$ depending only on $n$, $p$, $\lambda$, $\|a^{ij}\|_{C^{0,1}(B^+_1)}$, $\sup_{B^+_1} |x|^{1-\kappa} |f|$, $\|k_+\|_{C^{0,1}(B'_1)}$, $\|k_-\|_{C^{0,1}(B'_1)}$, and $\|v\|_{L^{\infty}(B_1)}$.
\end{theorem}

\begin{proof}
Throughout the proof we will denote by $C \in (0,\infty)$  constants depending only on $n$, $p$, $\lambda$, $\|a^{ij}\|_{C^{0,1}(B^+_1)}$, $\sup_{B^+_1} |x|^{1-\kappa} |f|$, $\|k_+\|_{C^{0,1}(B'_1)}$, $\|k_-\|_{C^{0,1}(B'_1)}$, and $\|v\|_{L^{\infty}(B_1)}$. We compute
\begin{equation*}
	\frac{d}{d\rho} \mathcal{W}(\rho) = \frac{1}{\rho^{2\nu}} \,\left( I'(\rho) - \nu H'(\rho) \right)
		- \frac{2\nu}{\rho^{2\nu+1}} \,\left({I(\rho)} - \nu H(\rho) \right)
\end{equation*}
for $\mathcal{L}^1$-a.e.~$0 < \rho < \rho_0/2$.  Taking into account \eqref{Alm mono eqn9} and \eqref{Alm mono eqn10},
\begin{align}\label{eq.W1}
	\frac{d}{d\rho} \mathcal{W}(\rho)
	\geq &\, 2\rho^{2-n-2\nu} \int_{(\partial B_{\rho})^+} \frac{1}{\mu} \left( \sum_{i,j=1}^n a^{ij} D_j v \,\frac{x_i}{\rho} \right)^2
		-\frac{4\nu}{\rho^{2\nu+1}} I(\rho) + \frac{2\nu^2}{\rho^{2\nu+1}} H(\rho) \\
		&- \frac{C}{\rho^{2\nu}} (D(\rho)+H(\rho)), \notag \\
	= &\, 2\rho^{-n-2\nu}\left\{ \rho^2\int_{(\partial B_{\rho})^+} \frac{1}{\mu} \left( \sum_{i,j=1}^n a^{ij} D_j v \,\frac{x_i}{\rho} \right)^2
		- 2\nu\rho \int_{(\partial B_{\rho})^+} a^{ij} v \,D_j v \,\frac{x_i}{\rho} +\nu^2 \int_{(\partial B_{\rho})^+} \mu \,v^2\right\} \notag \\
		&- \frac{C}{\rho^{2\nu}} (D(\rho)+H(\rho)) \notag
\end{align}
for $\mathcal{L}^1$-a.e.~$0 < \rho < \rho_0/2$.  We now observe
\begin{align}\label{eq.W2}
	\rho^2 \int_{(\partial B_{\rho})^+} \frac{1}{\mu} \left( \sum_{i,j=1}^n a^{ij} D_j v \,\frac{x_i}{\rho} \right)^2
		&- 2\rho\nu \int_{(\partial B_{\rho})^+} a^{ij} v \,D_j v \,\frac{x_i}{\rho} +\nu^2 \int_{(\partial B_{\rho})^+} \mu \,v^2 \\
	&= \int_{(\partial B_{\rho})^+} \frac{1}{\mu} \left( \sum_{i,j=1}^n a^{ij} D_j v \,x_i - \nu \,\mu \,v \right)^2 \notag
\end{align}
for each $0 < \rho < \rho_0/2$.  Combining \eqref{eq.W1} and \eqref{eq.W2}, we infer
\begin{equation}\label{eq weiss prelim}
	\frac{d}{d\rho} \mathcal{W}(\rho)
		\geq 2 \rho^{-n-2\nu} \int_{(\partial B_{\rho})^+} \frac{1}{\mu} \left( \sum_{i,j=1}^n a^{ij} D_j v \,x_i - \nu \,\mu \,v \right)^2
		- \frac{C}{\rho^{2\nu}} \left( D(\rho) + H(\rho) \right).
\end{equation}
for $\mathcal{L}^1$-a.e.~$0 < \rho < \rho_0/2$.  Since $p \geq 2$, an application of  Lemma~\ref{elliptic est lemma},
\begin{align*}
	H(\rho) + D(\rho) &\leq \sup_{(\partial B_{\rho})^+} v^2 + \rho^2 \sup_{B^+_{\rho}} |Dv|^2 + \sup_{B'_{\rho}} (k_+ + k_-) \,|v|^p
		\\&\leq C \left( \rho^{-n} \|v\|_{L^2(B^+_{2\rho})}^2 + \rho^4 \sup_{B^+_{2\rho}} |f|^2 \right) .
\end{align*}
 Since by assumption $K(\tau) > \tau^{2\kappa}$ for all $0 < \tau \leq 2\rho$, we can use \eqref{L2 growth concl3}, \eqref{f bound}, and the fact that $\nu\leq \kappa+1$ to obtain
\begin{equation*}
	H(\rho) + D(\rho) \leq C \left( \rho^{2\nu} + \rho^{2\kappa+2} \right) \leq C \rho^{2\nu}.
\end{equation*}
 Hence \eqref{eq weiss prelim} gives us
\begin{equation}\label{eq.W3}
	\frac{d}{d\rho} \mathcal{W}(\rho)
		\geq 2 \rho^{-n-2\nu} \int_{(\partial B_{\rho})^+} \frac{1}{\mu} \left( \sum_{i,j=1}^n a^{ij} D_j v \,x_i - \nu \,\mu \,v \right)^2 - C .
\end{equation}
for each $0 < \rho < 1/2$.   Integrating \eqref{eq.W3} gives us \eqref{Weiss concl}.
\end{proof}

\section{Monneau's monotonicity formula} \label{sec:Monneau sec}

Assume Hypothesis~\ref{almgren hyp} and let $v \in C^1(B^+_1(0) \cup B'_1(0))$ be a solution to \eqref{ptop freq eqn}.  Let $\nu$ be any positive integer and $p_{\nu}$ be any homogeneous degree $\nu$ harmonic polynomial on $\mathbb{R}^n$ such that $p_{\nu}(x',x_n) = p_{\nu}(x',-x_n)$ for each $x' \in \mathbb{R}^{n-1}$ and $x_{n+1} \in \mathbb{R}$.  We introduce the \emph{Monneau function} $\mathcal{M}(\rho) = \mathcal{M}_{v,p_{\nu}}(\rho)$ as
\begin{equation} \label{Monneau fun}
	\mathcal{M}(\rho) = \frac{1}{\rho^{n-1+2\nu}} \int_{(\partial B_{\rho})^+} (v - p_{\nu})^2 \mu .
\end{equation}
for each $0 < \rho < 1$.

\begin{theorem}\label{Monneau mono}
Let $\nu \in \mathbb{R}$ and $\kappa\in\mathbb{N}$, with $\nu\leq \kappa$. Assume Hypothesis~\ref{almgren hyp} and let $v \in C^1(B^+_1(0) \cup B'_1(0))$ be a solution to \eqref{ptop freq eqn} with $\Phi_v(0^+) \geq \nu$.  Let $p_{\nu}$ be a  harmonic polynomial on $\mathbb{R}^n$, homogeneous of degree $\nu$, such that $p_{\nu}(x',x_n) = p_{\nu}(x',-x_n)$ for each $x' \in \mathbb{R}^{n-1}$ and $x_{n+1} \in \mathbb{R}$.  Let $\rho_0$ be as in Lemma~\ref{L2 growth lemma}.  For each $0 < \sigma < \rho \leq \rho_0/2$ such that $K_{v}(\tau) > \tau^{2\kappa}$ for all $0 < \tau \leq 2\rho$,
\begin{equation}\label{Monneau concl}
	\mathcal{M}(\sigma) + C \sigma \leq \mathcal{M}(\rho) + C \rho
\end{equation}
for some constant $C \in (0,\infty)$ depending only on $n$, $p$, $\lambda$, $\|a^{ij}\|_{C^{0,1}(B^+_1)}$, $\sup_{B^+_1} |x|^{1-\kappa} |f|$, $\|k_+\|_{C^{0,1}(B'_1)}$, $\|k_-\|_{C^{0,1}(B'_1)}$, $\|v\|_{L^{\infty}(B_1)}$, and $\|p_{\nu}\|_{L^{\infty}(B_1)}$.
\end{theorem}
	
\begin{proof}
Throughout the proof we will denote by $C \in (0,\infty)$ constants depending only on $n$, $p$, $\lambda$, $\|a^{ij}\|_{C^{0,1}(B^+_1)}$, $\sup_{B^+_1} |x|^{1-\kappa} |f|$, $\|k_+\|_{C^{0,1}(B'_1)}$, $\|k_-\|_{C^{0,1}(B'_1)}$, $\|v\|_{L^{\infty}(B_1)}$, and $\|p_{\nu}\|_{L^{\infty}(B_1)}$. Set $w = v - p_{\nu}$ so that
\begin{equation}\label{Monneau eqn2}
	\mathcal{M}(\rho) = \frac{1}{\rho^{n-1+2\nu}} \int_{(\partial B_{\rho})^+} \mu \,w^2 = \frac{1}{\rho^{2\nu}} \,H_{w}(\rho) ,
\end{equation}
where $H_{w}(\rho)$ is given by \eqref{eq.H_w}.  Let $\mathcal{W}_{v}(\rho)$ and $\mathcal{W}_{p_{\nu}}(\rho)$ be given by \eqref{Weiss fun}.  Using $v = p_{\nu} + w$ and recalling the definitions of the Weiss function in \eqref{Weiss fun} and of $\mu$ in \eqref{mu defn}, we have
\begin{equation*}
	\mathcal{W}_{v}(\rho) = \mathcal{W}_{p_{\nu}}(\rho) + \mathcal{W}_{w}(\rho)
		+ \rho^{1-n-2\nu} \int_{(\partial B_{\rho})^+} \left( a^{ij} p_{\nu} \,D_j w \,x_i + a^{ij} w \,D_j p_{\nu} \,x_i - 2 \nu \,\mu \,p_{\nu} \,w \right) . \notag
\end{equation*}
In the last term, using $w = v - p_{\nu}$ we obtain
\begin{equation*}
	\mathcal{W}_{v}(\rho) =\mathcal{W}_{w}(\rho) - \mathcal{W}_{p_{\nu}}(\rho)
		+ \rho^{1-n-2\nu} \int_{(\partial B_{\rho})^+} \left( a^{ij} p_{\nu} \,D_j v \,x_i + a^{ij} v \,D_j p_{\nu} \,x_i - 2 \nu \,\mu \,p_{\nu} \,v \right) .
\end{equation*}
By the homogeneity of $p_{\nu}$, $x \cdot \nabla p_{\nu} = \nu \,p_{\nu}$, which together with \eqref{b defn} gives us
\begin{align}\label{Monneau eqn3}
	\mathcal{W}_{v}(\rho) =& \mathcal{W}_{w}(\rho) - \mathcal{W}_{p_{\nu}}(\rho)
		\\&+ \rho^{1-n-2\nu} \int_{(\partial B_{\rho})^+} \left( a^{ij} p_{\nu} \,D_j v \,x_i + v \,x \cdot \nabla p_{\nu} \,
			+ b^{ij} v \,D_j p_{\nu} \,x_i - 2 \nu \,\mu \,p_{\nu} \,v \right) \notag \\
	=&\, \mathcal{W}_{w}(\rho) - \mathcal{W}_{p_{\nu}}(\rho)
		+ \underbrace{ \rho^{1-n-2\nu} \int_{(\partial B_{\rho})^+} \left( a^{ij} p_{\nu} \,D_j v \,x_i - v \,x \cdot \nabla p_{\nu} \, \right) }_{I_1} \notag
		\\&+ \underbrace{ \rho^{1-n-2\nu} \int_{(\partial B_{\rho})^+} (b^{ij} v \,D_j p_{\nu} \,x_i - 2 \nu \,(\mu-1) \,p_{\nu} \,v) }_{I_2}. \notag
\end{align}
Now we want to bound the terms $\mathcal{W}_{p_{\nu}}(\rho)$, $I_1$, and $I_2$ on the right-hand side of \eqref{Monneau eqn3}.  By the homogeneity of $p_{\nu}$ and \eqref{b defn},
\begin{align*}
	\mathcal{W}_{p_{\nu}}(\rho)
	=&\, \rho^{1-n-2\nu} \int_{(\partial B_{\rho})^+} \left( a^{ij} p_{\nu} \,D_j p_{\nu} \,x_i - \nu \,\mu\, p_{\nu}^2 \right)
	\\=&\, \rho^{1-n-2\nu} \int_{(\partial B_{\rho})^+} p_{\nu} \left( x \cdot \nabla p_{\nu} - \nu\, p_{\nu} \right)
		\\&\,+ \rho^{1-n-2\nu} \int_{(\partial B_{\rho})^+} \left( b^{ij} p_{\nu} \,D_j p_{\nu} \,x_i - \nu \,(\mu-1)\, p_{\nu}^2 \right)
	\\=&\, \rho^{1-n-2\nu} \int_{(\partial B_{\rho})^+} \left( b^{ij} p_{\nu} \,D_j p_{\nu} \,x_i - \nu \,(\mu-1)\, p_{\nu}^2 \right) .
\end{align*}
Hence by \eqref{b eqn} and the homogeneity of $p_{\nu}$,
\begin{equation}\label{Monneau eqn4}
	|\mathcal{W}_{p_{\nu}}(\rho)| \leq C \rho
\end{equation}
for some constant $C \in (0,\infty)$ depending on $n$ and $\|p_{\nu}\|_{L^{\infty}(B^+_1)}$.   To bound $I_1$, notice that by the divergence theorem and using \eqref{ptop freq integral eqn} with $\zeta = p_{\nu} \mathbf{1}_{B^+_{\rho}}$, where $\mathbf{1}_{B^+_{\rho}}$ is the characteristic function of $B^+_{\rho}$,
\begin{equation}\label{Monneau eqn5}
	\int_{(\partial B_{\rho})^+} a^{ij} p_{\nu} \,D_j v \,\frac{x_i}{\rho}
	= \int_{B^+_{\rho}} (a^{ij} D_j v \,D_i p_{\nu} + f \,p_{\nu}) + \int_{B'_{\rho}} p_{\nu} \,(k_+ (v^+)^{p-1} - k_- (v^-)^{p-1}) .
\end{equation}
By again applying the divergence theorem, using the homogeneity of $p_{\nu}$, $\Delta p_{\nu} = 0$ in $B^+_{\rho}$, and $D_n p_{\nu} = 0$ on $B'_{\rho}$,
\begin{equation}\label{Monneau eqn6}
	\int_{(\partial B_{\rho})^+} v \,\frac{x}{\rho} \cdot \nabla p_{\nu}
	= \int_{B'_{\rho}} v \,D_n p_{\nu} + \int_{B^+_{\rho}} (v \,\Delta p_{\nu} + \nabla v \cdot \nabla p_{\nu})
	= \int_{B^+_{\rho}} \nabla v \cdot \nabla p_{\nu} .
\end{equation}
Therefore, by subtracting \eqref{Monneau eqn6} from \eqref{Monneau eqn5} and  using \eqref{b defn},
\begin{align}\label{Monneau eqn7}
	I_1 =&\, \rho^{2-n-2\nu} \int_{B^+_{\rho}} b^{ij} \,D_j v \,D_i p_{\nu} + \rho^{2-n-2\nu} \int_{B^+_{\rho}} f \, p_{\nu}
		\\&+ \rho^{2-n-2\nu} \int_{B'_{\rho}} p_{\nu} \,(k_+ (v^+)^{p-1} - k_- (v^-)^{p-1}) . \notag
\end{align}
We can bound the right-hand side of \eqref{Monneau eqn7} using elliptic estimates.  By Lemma~\ref{elliptic est lemma},
\begin{align*}
	\sup_{B^+_{\rho}} |v| + \rho \sup_{B^+_{\rho}} |Dv| \leq C \,( \rho^{-n/2} \|v\|_{L^2(B^+_{2\rho})} + \rho^2 \sup_{B^+_{2\rho}} |f| ).
\end{align*}
  Since by assumption $K_{v}(\tau) > \tau^{2\kappa}$ for all $0 < \tau \leq 2\rho$, we can use \eqref{L2 growth concl3} and \eqref{f bound} to obtain
\begin{equation}\label{Monneau eqn8}
	\sup_{B^+_{\rho}} |v| + \rho \sup_{B^+_{\rho}} |Dv| \leq C \rho^{\nu}.
\end{equation}
  By the homogeneity of $p_{\nu}$, \eqref{f bound}, \eqref{b eqn}, \eqref{Monneau eqn8}, and the facts $\kappa\geq \nu$ and $p\geq 2$, we can bound the right-hand side of \eqref{Monneau eqn7} by
\begin{equation}\label{Monneau eqn9}
	|I_1| \leq \rho^{2-\nu} \sup_{B^+_{\rho}} |Dv| + \rho^{\kappa-\nu+1} + \rho^{1-\nu} \sup_{B'_{\rho}} |v|^{p-1} \leq C \rho.
\end{equation}
  To bound $I_2$, by the homogeneity of $p_{\nu}$, \eqref{Monneau eqn8}, and \eqref{b eqn} we have that
\begin{equation}\label{Monneau eqn10}
	|I_2| = \left| \rho^{1-n-2\nu} \int_{(\partial B_{\rho})^+} (b^{ij} v \,D_j p_{\nu} \,x_i - 2 \nu \,(\mu-1) \,p_{\nu} \,v) \right|
	\leq \rho^{1-\nu} \sup_{(\partial B_{\rho})^+} |v|
	\leq C \rho.
\end{equation}
 Combining  \eqref{Monneau eqn3} with \eqref{Monneau eqn4}, \eqref{Monneau eqn9}, and \eqref{Monneau eqn10} we infer
\begin{equation}\label{Monneau eqn11}
	\big| \mathcal{W}_{v}(\rho) - \mathcal{W}_{w}(\rho) \big| \leq C \rho.
\end{equation}
On the other hand, by differentiating \eqref{Monneau eqn2},
\begin{equation*}
	\mathcal{M}'(\rho) = \frac{1}{\rho^{2\nu}} \,H'_{w}(\rho) - \frac{2\nu}{\rho^{2\nu+1}} \,H_{w}(\rho) .
\end{equation*}
Applying \eqref{variation eqn2} yields
\begin{align*}
	\mathcal{M}'(\rho) =& \frac{2}{\rho^{2\nu+1}} \,I_{w}(\rho)
		+ \rho^{1-n-2\nu} \int_{(\partial B_{\rho})^+} w^2 \left( (1-n)\,\frac{\mu-1}{\rho} + D_i\bigg( b^{ij} \frac{x_j}{|x|} \bigg) \right)
		- \frac{2\nu}{\rho^{2\nu+1}} \,H_{w}(\rho) \\
	=& \frac{2}{\rho} \,\mathcal{W}_{w}(\rho)
		+ \rho^{1-n-2\nu} \int_{(\partial B_{\rho})^+} w^2 \left( (1-n)\,\frac{\mu-1}{\rho} + D_i\bigg( b^{ij} \frac{x_j}{|x|} \bigg) \right) \notag
\end{align*}
By \eqref{Monneau eqn11},
\begin{equation}\label{Monneau eqn12}
	\mathcal{M}'(\rho) \geq \frac{2}{\rho} \,\mathcal{W}_{v}(\rho)
		+ \rho^{1-n-2\nu} \int_{(\partial B_{\rho})^+} w^2 \left( (1-n)\,\frac{\mu-1}{\rho} + D_i\bigg( b^{ij} \frac{x_j}{|x|} \bigg) \right) - C .
\end{equation}
%\begin{align}\label{Monneau eqn12}
%	\mathcal{W}_{v}(\rho) =&\, \frac{\rho}{2} \,\mathcal{M}'(\rho)
%		- \frac{1}{2} \,\rho^{2-n-2\nu} \int_{(\partial B_{\rho})^+} w^2 \left( (1-n)\,\frac{\mu-1}{\rho} + D_i\bigg( b^{ij} \frac{x_j}{|x|} \bigg) \right) + O(\rho) ,
%\end{align}
%where $O(\rho)$ denotes terms bounded by $C\rho$ for some appropriate constant $C \in (0,\infty)$.
Recalling that $w = v - p_{\nu}$, and applying \eqref{b eqn} and \eqref{Monneau eqn8}, we can bound the integral on the right-hand side of \eqref{Monneau eqn12} by
\begin{align*}
	&\left| \rho^{1-n-2\nu} \int_{(\partial B_{\rho})^+} w^2 \left( (1-n)\,\frac{\mu-1}{\rho} + D_i\bigg( b^{ij} \frac{x_j}{|x|} \bigg) \right) \right| \\
	&\leq C \rho^{-2\nu} \sup_{(\partial B_{\rho})^+} w^2
	\leq C \rho^{-2\nu} \left( \sup_{(\partial B_{\rho})^+} v^2 + \rho^{2\nu} \right)
	\leq C .
\end{align*}
%\begin{align*}
%	&\left| \rho^{2-n-2\nu} \int_{(\partial B_{\rho})^+} w^2 \left( (1-n)\,\frac{\mu-1}{\rho} + D_i\bigg( b^{ij} \frac{x_j}{|x|} \bigg) \right) \right| \\
%	&\leq C \rho^{1-2\nu} \sup_{(\partial B_{\rho})^+} w^2
%	\leq C \rho^{1-2\nu} \left( \sup_{(\partial B_{\rho})^+} v^2 + \rho^{2\nu} \right)
%	\leq C \rho .
%\end{align*}
Therefore, \eqref{Monneau eqn12} gives us
\begin{equation}\label{Monneau eqn13}
	\mathcal{M}'(\rho) \geq \frac{2}{\rho} \,\mathcal{W}_{v}(\rho) - C .
\end{equation}
%\begin{equation}\label{Monneau eqn13}
%	\big| \mathcal{W}_{v}(\rho) - \frac{\rho}{2} \,\mathcal{M}'(\rho) \big| \leq C \rho.
%\end{equation}
Notice that by Theorem~\ref{Weiss mono}, $\mathcal{W}_{v}(0^+) = \lim_{\rho \rightarrow 0^+} \mathcal{W}_{v}(\rho)$ exists.  Since $\Phi_v(0^+) \geq \nu$, $\mathcal{W}_{v}(0^+) \geq 0$.  Thus by again using Theorem~\ref{Weiss mono} we obtain $\mathcal{W}_{v}(\rho) \geq -C\rho$, which by \eqref{Monneau eqn13} gives us $\mathcal{M}'(\rho) \geq -C$.  The desired conclusion follows by integration.
\end{proof}

\section{Rectifiability of the free boundary}\label{rect}

Let $2 \leq p < \infty$, $u \in W^{1,2}(B^+_1(0))$ be a solution to \eqref{main ptop}, and $0 \in \Sigma(u)$ with $\Phi_u(0) < \infty$.  We want to show that the free boundary $\Sigma(u)$ is countably $(n-2)$-rectifiable in an open neighborhood of the origin.  We will follow the approach introduced in \cite{GP09}, based on the Weiss' and Monneau's monotonicity formulas proved in Sections \ref{sec:Weiss sec} and \ref{sec:Monneau sec}.

Let $\Omega = B^+_1(0)$ and $\Gamma = B'_1(0)$.  Fix a positive integer $\kappa$, and let $a^{ij} \in C^{\kappa -1,1}(B^+_1(0) \cup B'_1(0))$ such that $a^{ij} = a^{ji}$ on $B^+_1(0)$, satisfying \eqref{ellipticity hyp}  for some constants $0 < \lambda \leq \Lambda < \infty$, and such that \eqref{regularity assumption} holds true on $B'_1(0)$.  Let $k_+,k_- \in C^{0,1}(B'_1(0))$ with $k_+,k_- \geq 0$ and $h \in C^{\kappa,1}(B'_1(0))$.  Let $u \in W^{1,2}(B^+_1(0))$ be a weak solution to \eqref{main ptop}, and fix $x_0 \in \Sigma(u)$.  We begin by observing that, as a consequence of Theorem \ref{FBreg}, we may assume $\nabla_{x'} u(x_0) = \nabla_{x'} h(x_0)$. By Lemma~\ref{h taylor thm},   there exists a polynomial $\overline{h}_{x_0,\kappa}$ of degree at most $\kappa$ such that $\overline{h}_{x_0,\kappa}(x',0)$ is the Taylor polynomial of degree $\kappa$, of $h(x')$ at $x_0$ and
\begin{equation*}
	\big| D_i (a^{ij} D_j \overline{h}_{x_0,\kappa}) \big| \leq C |x-x_0|^{\kappa-1} \text{ in } B^+_1(0), \quad
	D_n \overline{h}_{x_0,\kappa} = 0 \text{ on } B'_1(0),
\end{equation*}
for some constant $C = C(n, \lambda, \|a^{ij}\|_{C^{\kappa-1,1}(B^+_1(0))}, \|h\|_{C^{\kappa,1}(B'_1(0))}) \in (0,\infty)$.  Using the notations introduced in \eqref{frequency v}, for fixed $x_0 \in \Sigma(u)$ let
\begin{equation}\label{rectifiability v}
	v_{x_0,\kappa}(x) = u(x+x_0) - \overline{h}_{x_0,\kappa}(x+x_0) + \overline{h}_{x_0,\kappa}((x+x_0)',0) - h((x+x_0)',0)
\end{equation}
for each $x = (x',x_n) \in B^+_1(0) \cup B'_1(0)$.  For each positive integer $\nu$ define
\begin{equation*}
	\Sigma_{\nu}(u) = \{ x \in \Sigma(u) : \mathcal{N}_u(x_0) = \nu \} ,
\end{equation*}
where $\mathcal{N}_u(x_0)$ is as in Definition~\ref{Almgren freq defn}. Since we are assuming $\nabla_{x'} u(x_0) = \nabla_{x'} h(x_0)$,   by virtue of Theorem \ref{tangent thm} in the sequel we can restrict our attention to the case $\nu\geq 2.$

We begin by proving the non-degeneracy property of the solution.

\begin{lemma}\label{nondeg lemma}
Let $2 \leq p < \infty$, and let $\kappa$, $\nu$ be positive integers with $\nu < \kappa$.  Let $a^{ij} \in C^{\kappa -1,1}(B^+_1(0) \cup B'_1(0))$ such that $a^{ij} = a^{ji}$ on $B^+_1(0)$, satisfying \eqref{ellipticity hyp}  for some constants $0 < \lambda \leq \Lambda < \infty$, and such that \eqref{regularity assumption} holds true on $B'_1(0)$.  Let $k_+,k_- \in C^{0,1}(B'_1(0))$ with $k_+,k_- \geq 0$ and $h \in C^{\kappa ,1}(B'_1(0))$.  Let $u \in W^{1,2}(B^+_1(0))$ be a solution to \eqref{main ptop} and $x_0 \in \Sigma_{\nu}(u)$.  Let $v = v_{x_0,\kappa}$ be as in \eqref{rectifiability v}.  Then there exists constants $\delta \in (0,1]$ and $c > 0$ such that for each $\rho \in (0,\delta]$
\begin{equation*}
	\sup_{(\partial B_{\rho}(x_0))^+} |v| \geq c \rho^{\nu} .
\end{equation*}
\end{lemma}

\begin{proof}
After translating $x_0$ to the origin and applying a linear change of variables, assume that $x_0 = 0$ and $a^{ij}(0) = \delta_{ij}$.  Suppose there exists a sequence $\rho_{\ell} \rightarrow 0^+$ such that
\begin{equation}\label{nondeg eqn1}
	\lim_{\ell \rightarrow \infty} \sup_{(\partial B_{\rho_{\ell}})^+} \frac{|v|}{\rho_{\ell}^{\nu}} = 0 .
\end{equation}
Set $h_{\ell} = H_v(\rho_{\ell})^{1/2}$, where $H_v$ is as in \eqref{eq.H}.  By Theorem~\ref{tangent thm}, there exists an harmonic  polynomial $v^*$, homogeneous of degree $\nu$ and even in $x_n$, such that after passing to a subsequence
\begin{equation}\label{nondeg eqn2}
	v_{\ell}(x) = \frac{v(\rho_{\ell} x)}{h_{\ell}} \rightarrow v^*(x)
\end{equation}
in the $C^1$-topology on compact subsets of $\mathbb{R}^n \cap \{x_n \geq 0\}$ as $\ell \rightarrow \infty$.  The monotonicity of $\mathcal{M}_{v,v^*}$, established in Theorem \ref{Monneau mono}, implies that $\mathcal{M}_{v,v^*}(0^+) = \lim_{\rho \rightarrow 0^+} \mathcal{M}_{v,v^*}(\rho)$ exists.  Moreover, by \eqref{nondeg eqn1} and the homogeneity of $v^*$,
\begin{align*}
	\mathcal{M}_{v,v^*}(0^+)
	&= \lim_{\ell \rightarrow \infty} \mathcal{M}_{v,v^*}(\rho_{\ell})
	= \lim_{\ell \rightarrow \infty} \rho_{\ell}^{1-n-2\nu} \int_{(\partial B_{\rho_{\ell}})^+} (v - v^*)^2 \mu
	\\&= \lim_{\ell \rightarrow \infty} \int_{(\partial B_1)^+} \left( \frac{v(\rho_{\ell} x)}{\rho_{\ell}^{\nu}} - v^*(x) \right)^2 \mu_{\rho_{\ell}}(x) \,dx
	= \int_{(\partial B_1)^+} (v^*)^2 ,
\end{align*}	
where we let $\mu_{\rho_{\ell}}(x) = \mu(\rho_{\ell} x)$ for each $x \in (\partial B_1)^+$ and we note that $\mu_{\rho_{\ell}} \rightarrow 1$ uniformly on $(\partial B_1)^+$ by \eqref{b eqn}.  Using again the monotonicity of $\mathcal{M}_{v,v^*}$ and the homogeneity of $v^*$, we infer
\begin{align*}
	\rho_{\ell}^{1-n-2\nu} \int_{(\partial B_{\rho_{\ell}})^+} (v^*)^2
	= \int_{(\partial B_1)^+} (v^*)^2
	&= \mathcal{M}_{v,v^*}(0^+) \leq \mathcal{M}_{v,v^*}(\rho_{\ell}) + C \rho_{\ell}
	\\&= \rho_{\ell}^{1-n-2\nu} \int_{(\partial B_{\rho_{\ell}})^+} (v - v^*)^2 \mu + C \rho_{\ell}
\end{align*}	
%or equivalently
and thus by \eqref{b eqn}
\begin{equation*}
	0 \leq \rho_{\ell}^{1-n-2\nu} \int_{(\partial B_{\rho_{\ell}})^+} (v^2 - 2 v v^*) \,\mu  + C \rho_{\ell}
\end{equation*}	
where $C \in (0,\infty)$ is a constant depending only on $n$, $p$, $\lambda$, $\|a^{ij}\|_{C^{\kappa-1,1}(B^+_1)}$, $\|h\|_{C^{\kappa,1}(B'_1)}$, $\|k_+\|_{C^{0,1}(B'_1)}$, $\|k_-\|_{C^{0,1}(B'_1)}$, $\|v\|_{L^{\infty}(B^+_1)}$, and $\|v^*\|_{L^{\infty}(B^+_1)}$.  By rescaling and recalling that $v_{\ell}(x) = v(\rho_{\ell} \,x)/h_{\ell}$ and $\mu_{\rho_{\ell}}(x) = \mu(\rho_{\ell} x)$,
\begin{equation*}
	0 \leq \int_{(\partial B_1)^+} \left( \frac{h_{\ell}^2}{\rho_{\ell}^{2\nu}} \,v_{\ell}^2 - 2 \frac{h_{\ell}}{\rho_{\ell}^{\nu}} \,v_{\ell} \,v^* \right) \mu_{\rho_{\ell}}
		+ C \rho_{\ell}
\end{equation*}	
Dividing by $\rho_{\ell}^{-\nu} h_{\ell}$,
\begin{equation}\label{nondeg eqn3}
	0 \leq \int_{(\partial B_1)^+} \left( \frac{h_{\ell}}{\rho_{\ell}^{\nu}} \,v_{\ell}^2 - 2 v_{\ell} \,v^* \right)\mu_{\rho_{\ell}} \,dx
		+ \frac{C \rho_{\ell}^{\nu+1}}{h_{\ell}} .
\end{equation}
Note that, since $\nu <\kappa$, we have $K(\tau)>\tau^{2\kappa}$ for all $\tau>0$ sufficiently small. We can then apply \eqref{L2 growth concl2} to get $h_{\ell} = H_v(\rho_{\ell})^{1/2} \geq c \rho_{\ell}^{\nu+1/2}$ for some constant $c \in (0,\infty)$ independent of $\ell$.  Hence letting $\ell \rightarrow \infty$ in \eqref{nondeg eqn3} and using $\mu_{\rho_{\ell}} \rightarrow 1$ and $\rho_{\ell}^{-\nu} h_{\ell} \rightarrow 0$ (by \eqref{nondeg eqn1} and \eqref{nondeg eqn2}),
\begin{equation*}
	0 \leq -2 \int_{(\partial B_1)^+} (v^*)^2 .
\end{equation*}
This fact, together with the homogeneity of $v^*$, implies that $v^*$ is identically zero on $\overline{\mathbb{R}^n_+}$, contradicting $v^*$ being a non-zero polynomial.
\end{proof}

\begin{theorem}\label{unique tangent thm}
Let $2 \leq p < \infty$, and let $\kappa$, $\nu$ be positive integers with $2\leq\nu < \kappa$.  Let $a^{ij} \in C^{\kappa -1,1}(B^+_1(0) \cup B'_1(0))$ such that $a^{ij} = a^{ji}$ on $B^+_1(0)$, satisfying \eqref{ellipticity hyp}  for some constants $0 < \lambda \leq \Lambda < \infty$, and such that \eqref{regularity assumption} holds true on $B'_1(0)$.  Let $k_+,k_- \in C^{0,1}(B'_1(0))$ with $k_+,k_- \geq 0$ and $h \in C^{\kappa ,1}(B'_1(0))$.  Let $u \in W^{1,2}(B^+_1(0))$ be a solution to \eqref{main ptop} and $x_0 \in \Sigma_{\nu}(u)$.  Let $v = v_{x_0,\kappa}$ be as in \eqref{rectifiability v}.  Then there exists a unique function $v^*=v^*_{x_0} \in C^1(\overline{\mathbb{R}^n_+})$ such that
\begin{equation} \label{unique tangent concl}
	\frac{v(\rho x)}{\rho^{\nu}} \rightarrow v^*_{x_0}(x) \text{ as }\rho  \rightarrow 0^+
\end{equation}
in $C^1(B^+_{\sigma} \cup B'_{\sigma})$ for every $\sigma \in (0,\infty)$.  Furthermore, the even reflection of $v^*_{x_0}$ across $\{x_n = 0\}$ (also denoted by $v^*=v^*_{x_0}$) is  homogeneous degree $\nu$ and is a smooth solution to
\begin{equation}\label{hom_eqn}
	a^{ij}(x_0) \,D_{ij} v^*_{x_0} = 0 \text{ in } \mathbb{R}^n .
\end{equation}
\end{theorem}

\begin{remark}{\rm
By \eqref{freq consistency eqn4}, the limit $v^*$ in Theorem~\ref{unique tangent thm} is independent of the choice of integer $\kappa$.
}\end{remark}

\begin{proof}[Proof of Theorem~\ref{unique tangent thm}]
We begin by noticing that $H_v(\rho)^{1/2}\approx \rho^\nu$ by Lemmas \ref {L2 growth lemma} and \ref{nondeg lemma}. We can thus follow the lines of the proof of Theorem~\ref{tangent thm} to show the existence of a sequence $\rho_{\ell} \rightarrow 0^+$ and of a  harmonic polynomial $v^* \in C^1(\overline{\mathbb{R}^n})$, homogeneous of degree $\nu$ and even in $x_n$, such that for each $\sigma \in (0,\infty)$
\begin{equation} \label{unique tangent eqn2}
	v_{\rho_{\ell}}(x) = \frac{v(\rho_{\ell} x)}{\rho_{\ell}^{\nu}} \rightarrow v^*(x)
\end{equation}
in $C^1(B^+_{\sigma} \cup B'_{\sigma})$ as $\ell \rightarrow \infty$.  Moreover, an application of Theorem~\ref{Monneau mono} yields
\begin{equation*}
	\mathcal{M}_{v,v^*}(0^+)
	= \lim_{\ell \rightarrow \infty} \mathcal{M}_{v,v^*}(\rho_{\ell})
	= \lim_{\ell \rightarrow \infty} \rho_{\ell}^{1-n-2\nu} \int_{(\partial B_{\rho_{\ell}})^+} (v - v^*)^2\mu
	= \lim_{\ell \rightarrow \infty} \int_{(\partial B_1)^+} (v_{\rho_{\ell}} - v^*)^2\mu_{\rho_{\ell}}
	= 0 ,
\end{equation*}
where $\mu_{\rho}(x) = \mu(\rho x)$ for each $x \in (\partial B_1)^+$ and $\rho > 0$ and in the last equality we have used \ref{unique tangent eqn2}. Therefore, by Theorem~\ref{Monneau mono} again,
\begin{equation*}
	\lim_{\rho \rightarrow 0^+} \int_{(\partial B_1)^+} (v_{\rho} - v^*)^2\mu_{\rho}
	= \lim_{\rho \rightarrow 0^+} \rho^{1-n-2\nu} \int_{(\partial B_{\rho})^+} (v - v^*)^2\mu
	= \lim_{\rho \rightarrow 0^+} \mathcal{M}_{v,v^*}(\rho)
	= \mathcal{M}_{v,v^*}(0^+) = 0 .
\end{equation*}
If follows that if for a different sequence of radii $(\rho'_{\ell})$ with $\rho'_{\ell} \rightarrow 0^+$ and some function $v' \in C^1(\overline{\mathbb{R}^n_+})$ we had
\begin{equation*}
	v_{\rho'_{\ell}}(x) = \frac{v(\rho'_{\ell} x)}{(\rho'_{\ell})^{\nu}} \rightarrow v'(x)
\end{equation*}
in $C^1(B^+_{\sigma} \cup B'_{\sigma})$ as $\ell \rightarrow \infty$, then
\begin{equation*}
	\int_{(\partial B_1)^+} (v' - v^*)^2
	= \lim_{\ell \rightarrow \infty} \int_{(\partial B_1)^+} (v_{\rho'_{\ell}} - v^*)^2\mu_{\rho'_{\ell}}
	= 0,
\end{equation*}
and thus $v' = v^*$. We leave the verification of \eqref{hom_eqn} to the reader.
\end{proof}

\begin{theorem}\label{conts tangent thm}
Let $2 \leq p < \infty$, and let $\kappa$, $\nu$ be positive integers with $2\leq \nu < \kappa$.  Let $a^{ij} \in C^{\kappa -1,1}(B^+_1(0) \cup B'_1(0))$ such that $a^{ij} = a^{ji}$ on $B^+_1(0)$, satisfying \eqref{ellipticity hyp}  for some constants $0 < \lambda \leq \Lambda < \infty$, and such that \eqref{regularity assumption} holds true on $B'_1(0)$.  Let $k_+,k_- \in C^{0,1}(B'_1(0))$ with $k_+,k_- \geq 0$ and $h \in C^{\kappa ,1}(B'_1(0))$.  Let $u \in W^{1,2}(B^+_1(0))$ be a solution to \eqref{main ptop}.  For each $x_0 \in \Sigma_{\nu}(u)$, let $v^*_{x_0}$ be as in Theorem~\ref{unique tangent thm}.  Then the mapping $x_0 \mapsto v^*_{x_0}$ from $\Sigma_{\nu}(u)$ to $C^{\nu}(\overline{B_1(0)})$ is continuous.  Moreover, for each compact set $K \subseteq \Sigma_{\nu}(u)$ there exists a modulus of continuity $\omega = \omega^K$ such that $\omega(0^+) = 0$ and
\begin{equation}\label{conts tangent concl}
	|v_{x_0,\kappa}(x) - v^*_{x_0}(x)| \leq \omega(|x - x_0|) \,|x - x_0|^{\nu}
\end{equation}
for all $x_0 \in K$, where  $v_{x_0,\kappa}$ is as in \eqref{rectifiability v}.
\end{theorem}

\begin{proof}
Assume first that $0 \in \Sigma_{\nu}(u)$. We will show that the map $x_0 \mapsto v^*_{x_0}$ is continuous at the origin using  Monneau's monotonicity formula, Theorem~\ref{Monneau mono}.  Recall that Theorem~\ref{Monneau mono} applies for $x_0 \in \Sigma_{\nu}(u)$ after an affine change of variables $x \mapsto L(x_0)^{-1} (x - x_0)$ translates $x_0$ to the origin and transforms $a^{ij}(x_0)$ to $\delta_{ij}$.  In particular, by the continuity of $a^{ij}$, for each $x_0 \in \Sigma_{\nu}(u)$ there exists an invertible $n \times n$ matrix $L(x_0) = (L_{ij}(x_0))$, with  inverse $L(x_0)^{-1} = (L^{ij}(x_0))$ such that $x_0 \,\mapsto\, L(x_0)$ is a continuous mapping on $\Sigma_{\nu}(u)$, $L_{ij}(0) = \delta_{ij}$, and
\begin{equation*}%\label{conts tangent eqn1}
	\sum_{k,l=1}^n \sqrt{\det L(x_0)} \,a^{kl}(x_0+L(x_0) \,x) \,L^{ik}(x_0) \,L^{jl}(x_0) = \delta_{ij}
\end{equation*}
for all $x_0 \in \Sigma_{\nu}(u)$ and $i,j = 1,2,\ldots,n$.  For each $x_0 \in \Sigma_{\nu}(u) \cap B'_{1/2}(0)$ let
\begin{equation*}
	\widetilde{v}_{x_0,\kappa}(x) = v_{x_0,\kappa}(L(x_0) \,x)
\end{equation*}
for all $x \in B^+_{1/4}(0) \cup B'_{1/4}(0)$.  Note that since $L_{ij}(0) = \delta_{ij}$, $\widetilde{v}_{0,\kappa} = v_{0,\kappa}$.  Fix $\varepsilon \in (0,1)$.  Select $\rho_{\varepsilon} > 0$ such that
\begin{equation*}
	\mathcal{M}_{\widetilde{v}_{0,\kappa},v^*_0}(\rho_{\varepsilon})
		= \rho_{\varepsilon}^{1-n-2\nu} \int_{(\partial B_{\rho_{\varepsilon}})^+} (\widetilde{v}_{0,\kappa} - v^*_0)^2\mu < \varepsilon .
\end{equation*}
 By the continuous dependence of $L(x_0)$ and $\widetilde{v}_{x_0,\kappa}$ on $x_0 \in \Sigma_{\nu}(u)$, there exists $\delta_{\varepsilon} > 0$ such that for each $x_0 \in \Sigma_{\nu}(u) \cap B'_{\delta_{\varepsilon}}(0)$ we have
\begin{gather}
	\label{conts tangent eqn2} |L_{ij}(x_0) - \delta_{ij}| < \varepsilon , \\
	\label{conts tangent eqn3} \mathcal{M}_{\widetilde{v}_{x_0,\kappa},v^*_0}(\rho_{\varepsilon})
		= \rho_{\varepsilon}^{1-n-2\nu} \int_{(\partial B_{\rho_{\varepsilon}})^+} (\widetilde{v}_{x_0,\kappa} - v^*_0)^2\mu < 2\varepsilon .
\end{gather}
By Theorem~\ref{Monneau mono} and \eqref{conts tangent eqn3},
\begin{equation}\label{conts tangent eqn4}
	\mathcal{M}_{\widetilde{v}_{x_0,\kappa},v^*_0}(\rho)
		= \rho^{1-n-2\nu} \int_{(\partial B_{\rho})^+} (\widetilde{v}_{x_0,\kappa} - v^*_0)^2\mu < 2\varepsilon + C \rho_{\varepsilon}
\end{equation}
for all $x_0 \in \Sigma_{\nu}(u) \cap B'_{\delta_{\varepsilon}}(0)$ and $0 < \rho \leq \rho_{\varepsilon}$, where $C \in (0,\infty)$ is a constant.  Integrating \eqref{conts tangent eqn4},
\begin{equation}\label{conts tangent eqn5}
	\rho^{-n-2\nu} \int_{B^+_{\rho}} (\widetilde{v}_{x_0,\kappa} - v^*_0)^2 \mu < C (\varepsilon + \rho_{\varepsilon})
\end{equation}
for all $x_0 \in \Sigma_{\nu}(u) \cap B'_{\delta_{\varepsilon}}(0)$ and $0 < \rho \leq \rho_{\varepsilon}$, where $C \in (0,\infty)$ is a constant.  Taking  \eqref{conts tangent eqn2} and \eqref{conts tangent eqn5} into account, for sufficiently small $\delta_\varepsilon>0$ we obtain
\begin{align}\label{conts tangent eqn6}
	\rho^{-n-2\nu} \int_{B^+_{\rho}} (v_{x_0,\kappa} - v^*_0)^2 \mu\nonumber
	\leq&\, C \rho^{-n-2\nu} \int_{B^+_{2\rho}} (v_{x_0,\kappa}(L(x_0)\,x) - v^*_0(L(x_0)\,x))^2\mu \,dx \nonumber
	\\=&\, C \rho^{-n-2\nu} \int_{B^+_{2\rho}} (\widetilde{v}_{x_0,\kappa}(x) - v^*_0(L(x_0)\,x))^2\mu \,dx \nonumber
	\\ \leq&\, 2C \rho^{-n-2\nu} \int_{B^+_{2\rho}} (\widetilde{v}_{x_0,\kappa}(x) - v^*_0(x))^2\mu \,dx \nonumber
		\\& + 2C \rho^{-n-2\nu} \int_{B^+_{2\rho}} (v^*_0(L(x_0)\,x) - v^*_0(x))^2 \mu\,dx \nonumber
	\\ <&\, C (\varepsilon + \rho_{\varepsilon}) + C \max_{1 \leq i,j \leq n} \|L_{ij}(x_0) - \delta_{ij}\| \nonumber
	\\ <&\, C (\varepsilon + \rho_{\varepsilon})
\end{align}
for all $x_0 \in \Sigma_{\nu}(u) \cap B'_{\delta_{\varepsilon}}(0)$ and $0 < \rho \leq \rho_{\varepsilon}/2$, where $C \in (0,\infty)$ denotes constants.  After rescaling,
\begin{equation*}
	\int_{B^+_1} \left(\frac{v_{x_0,\kappa}(\rho x)}{\rho^{\nu}} - v^*_0(x)\right)^2 \mu(\rho x)\,dx < C (\varepsilon + \rho_{\varepsilon})
\end{equation*}
for all $x'_0 \in \Sigma_{\nu}(u) \cap B'_{\delta_{\varepsilon}}(0)$ and $0 < \rho \leq \rho_{\varepsilon}/2$.  Thus letting $\rho \rightarrow 0^+$ using Theorem~\ref{unique tangent thm} and \eqref{b eqn},
\begin{equation}\label{conts tangent eqn7}
	\int_{B^+_1} (v^*_{x_0} - v^*_0)^2 < C (\varepsilon + \rho_{\varepsilon})
\end{equation}
for all $x_0 \in \Sigma_{\nu}(u) \cap B'_{\delta_{\varepsilon}}(0)$.  Therefore, $x_0 \mapsto v^*_{x_0}$ is continuous at the origin as a mapping from $\Sigma_{\nu}(u)$ to $L^2(B^+_1)$.  Since the space of all homogeneous degree $\nu$ polynomials on $\mathbb{R}^n$ is a finite dimensional vector space, $x_0 \mapsto v^*_{x_0}$ is continuous at the origin as a mapping from $\Sigma_{\nu}(u)$ to $C^{\kappa}(\overline{B_1})$.

Now let $K \subset \Sigma_{\nu}(u)$ be a compact set.  By \eqref{conts tangent eqn6} and \eqref{conts tangent eqn7}, for each $\varepsilon > 0$ and $x_0 \in \Sigma_{\nu}(u)$ there exists $\rho_{\varepsilon}(x_0) > 0$ and $\delta_{\varepsilon}(x_0) > 0$ such that
\begin{align*}
	\rho^{-n-2\nu} \int_{B^+_{\rho}} (v_{x'_0,\kappa} - v^*_{x'_0})^2\mu_{x_0}
		&\leq 2 \rho^{-n-2\nu} \int_{B^+_{\rho}} (v_{x'_0,\kappa} - v^*_{x_0})^2\mu_{x_0}
			+ 2 \rho^{-n-2\nu} \int_{B^+_{\rho}} (v^*_{x'_0} - v^*_{x_0})^2\mu_{x_0}
		\\&= 2 \rho^{-n-2\nu} \int_{B^+_{\rho}} (v_{x'_0,\kappa} - v^*_{x_0})^2\mu_{x_0}
			+ 2 \int_{B^+_1} (v^*_{x'_0} - v^*_{x_0})^2\mu_{x_0}
		\\&\leq C (\varepsilon + \rho_{\varepsilon})
\end{align*}
for all $x'_0 \in \Sigma_{\nu}(u) \cap B'_{\delta_{\varepsilon}(x_0)}(x_0)$ and $0 < \rho \leq \rho_{\varepsilon}(x_0)/2$, where $v_{x'_0,\kappa}$ is as in \eqref{rectifiability v} (with $x'_0$ replacing $x_0$), $\mu_{x_0}=\sum_{i,j=1}^n a^{ij}(x+x_0)(x+x_0)_i(x+x_0)_j/|x+x_0|^2$, and $C = C(x_0) \in (0,\infty)$ is a constant.  Since $K$ is compact, we can cover $K$ by open balls $B'_{\delta_{\varepsilon}(x^i_0)}(x^i_0)$, where $x^i_0 \in K$ for $i = 1,2,\ldots,N$.  Then setting
\begin{equation*}
	\rho^K_{\varepsilon} = \min \{ \rho_{\varepsilon}(x^i_0) : i = 1,2,\ldots,N \}, \quad
	C^K = \max \{ C(x^i_0) : i = 1,2,\ldots,N \} ,
\end{equation*}
we have that
\begin{equation}\label{conts tangent eqn8}
	\rho^{-n-2\nu} \int_{B^+_{\rho}} (v_{x'_0,\kappa} - v^*_{x'_0})^2\mu_{x'_0} \leq C^K (\varepsilon + \rho^K_{\varepsilon})
\end{equation}
for all $x'_0 \in K$ and $0 < \rho \leq \rho^K_{\varepsilon}/2$.  At this point we introduce the rescalings
$$
w^{(\rho)}(x)=w^{(\rho)}_{x'_0,\kappa}(x)=\frac{v_{x'_0,\kappa}(\rho x)}{\rho^\nu}\ \mbox{ and }\ f^{(\rho)}(x)=f^{(\rho)}_{x'_0,\kappa}(x)=\frac{f_{x'_0,\kappa}(\rho x)}{\rho^{\nu-2}},
$$
where $f_{x'_0,\kappa}$ is as in \eqref{frequency v}, but corresponding to the free boundary point $x'_0 \in \Sigma_{\nu}(u)$. Since $v_{x'_0,\kappa}$ as in \eqref{rectifiability v} satisfies \eqref{ptop freq eqn} (with $x'_0$ replacing the origin), and $v^*_{x'_0}$ satisfies
$$
a^{ij}(x'_0) \,D_{ij} v^*_{x'_0} = 0 \mbox{ in }\mathbb{R}^n_+\ \mbox{ and }\ D_n v^*_{x'_0} = 0\mbox{ on }\mathbb{R}^{n-1} \times \{0\},
$$
 we have
\begin{gather*}
	D_i (\tilde{a}^{ij} D_j (w^{(\rho)} - v^*_{x'_0})) =  f^{(\rho)} - D_i ((\tilde{a}^{ij}- \tilde{a}^{ij}(0)) \,D_j v^*_{x'_0})\ \text{ in }\ B^+_1, \nonumber \\
	\tilde a^{nn} D_n (w^{(\rho)} - v^*_{x'_0}) = \rho^{1+(p-2)\nu}\left(\tilde k_+ \,((w^{(\rho)})^+)^{p-1} - \tilde k_- \,((w^{(\rho)})^-)^{p-1}\right)\ \text{ on }\ B'_1 ,
\end{gather*}
where $\tilde{a}^{ij}(x)=a^{ij}(x'_0+\rho x)$ and $\tilde k_\pm (x)=  k_\pm (x'_0+\rho x)$. Using \eqref{f bound}, $a^{ij} \in C^{0,1}(B^+_1 \cup B'_1)$, and %the homogeneity of $v^*_{x'_0}$, and the fact
$2\leq\nu <\kappa$, we infer
\begin{equation}\label{RHSest}
	|f^{(\rho)}(x) - D_i ((\tilde{a}^{ij}(x)- \tilde{a}^{ij}(0)) \,D_j v^*_{x'_0}(x))|\leq C (\rho^{\kappa -\nu +1}+\rho)\leq C\rho.
\end{equation}
on $B^+_1$ for some constant $C \in (0,\infty)$ independent of $\rho$.  Observing that $(w^{(\rho)})^{p-1}$ is bounded, taking  \eqref{conts tangent eqn8} and \eqref{RHSest} into account, and applying  $L^2-L^\infty$ estimates for weak solutions to the oblique derivative problems (see for instance~\cite[Theorem 5.36]{L13}) we obtain
\begin{equation*}
	|w^{(\rho)}(x) - v^*_{x'_0}(x)| \leq C \left(\,(\varepsilon + \rho^K_{\varepsilon})^{1/2} +\rho^K_{\varepsilon}\right)\leq C_\varepsilon
\end{equation*}
for all $x'_0 \in K$ and $x\in B_{1/2}(x'_0)$, where $C \in (0,\infty)$ is a constant. It is easy to see that this implies the second part of the theorem, and the proof is thus complete.
\end{proof}

For each $x_0 \in \Sigma_{\nu}(u)$, since $v^*_{x_0}$ is a homogeneous degree $\nu$ polynomial, we can express $v^*_{x_0}$ as
\begin{equation}\label{whitney prereq taylor}
	v^*_{x_0}(x) = \sum_{|\alpha| = \nu} \frac{F_{\alpha}(x_0)}{\alpha!} \,x^{\alpha} ,
\end{equation}
where $F_{\alpha}(x_0) \in \mathbb{R}$ for all $|\alpha| = \nu$ and $x_0 \in \Sigma_{\nu}(u)$.  Define $F_{\alpha}(x_0) = 0$ for all $|\alpha| < \nu$ and $x_0 \in \Sigma_{\nu}(u)$.

\begin{lemma}\label{whitney prereq lemma}
For each compact set $K \subset \Sigma_{\nu}(u)$,
\begin{equation*}
	F_{\alpha}(x) = \sum_{|\beta| \leq \nu-|\alpha|} \frac{F_{\alpha+\beta}(x_0)}{\beta!} \,(x-x_0)^{\alpha+\beta} + R_{\alpha}(x,x_0)
\end{equation*}
for all $|\alpha| \leq \nu$ and $x,x_0 \in K$, where
\begin{equation}\label{whitney prereq concl}
	|R_{\alpha}(x,x_0)| \leq \omega_{\alpha}(|x-x_0|) \,|x-x_0|^{\nu-|\alpha|}
\end{equation}
for all $|\alpha| \leq \nu$ and $x,x_0 \in K$ and some modulus of continuity $\omega_{\alpha}$ with $\omega_{\alpha}(0^+) = 0$.
\end{lemma}

\begin{proof}
In the case that $|\alpha| = \nu$, then
\begin{equation*}
	R_{\alpha}(x,x_0) = F_{\alpha}(x) - F_{\alpha}(x_0) = D^{\alpha} v^*_x(x) - D^{\alpha} v^*_{x_0}(x_0)
\end{equation*}
for each $x,x_0 \in K$.  Thus \eqref{whitney prereq concl} holds true by the continuity of $x_0 \mapsto v^*_{x_0}$.  In the case that $|\alpha| < \nu$, then
\begin{equation*}
	R_{\alpha}(x,x_0) = -\sum_{|\beta| = \nu-|\alpha|} \frac{F_{\alpha+\beta}(x_0)}{\beta!} \,(x-x_0)^{\alpha+\beta} = -D^{\alpha} v^*_{x_0}(x-x_0)
\end{equation*}
for each $x,x_0 \in K$.  To show \eqref{whitney prereq concl}, suppose to the contrary that there exists $\varepsilon > 0$ and $x^i,x^i_0 \in K$ such that $\rho_i = |x^i - x^i_0| \rightarrow 0$ and
\begin{equation}\label{whitney prereq eqn1}
	|R_{\alpha}(x^i,x^i_0)| = |D^{\alpha} v^*_{x_0}(x^i - x^i_0)| \geq \varepsilon \,|x^i - x^i_0|^{\nu-|\alpha|} .
\end{equation}
Set
\begin{equation}\label{whitney prereq eqn2}
	w_i(x) = \frac{v_{x^i_0,\kappa}(\rho_i x)}{\rho_i^{\nu}} , \quad\quad
	\widehat{w}_i(x) = \frac{v_{x^i,\kappa}(\rho_i x)}{\rho_i^{\nu}} , \quad\quad
	\xi_i = \frac{x^i - x^i_0}{\rho_i} .
\end{equation}
After passing to a subsequence there exists $x_0 \in K$ such that $x^i \rightarrow x_0$ and $x^i_0 \rightarrow x_0$ as $i \rightarrow \infty$, and there exists $\xi$ such that $|\xi| = 1$ and $\xi_i \rightarrow \xi$ as $i \rightarrow \infty$.  By Theorem~\ref{conts tangent thm},
\begin{equation*}
	|w_i(x) - v^*_{x_0}(x)| \leq |w_i(x) - v^*_{x^i_0}(x)| + |v^*_{x^i_0}(x) - v^*_{x_0}(x)|
		\leq \omega(\rho_i |x|) \,|x|^{\nu} + \|v^*_{x^i_0} - v^*_{x_0}\|_{L^2(B^+_1)} \,|x|^{\nu}
\end{equation*}
for each $x \in B^+_{1/(2\rho_i)}(0) \cup B'_{1/(2\rho_i)}(0)$, where $\omega$ is a modulus of continuity with $\omega(0^+) = 0$ and $\|v^*_{x^i_0} - v^*_{x_0}\|_{L^2(B^+_1)} \rightarrow 0$ as $i \rightarrow \infty$.  Hence
\begin{equation}\label{whitney prereq eqn3}
	w^i(x) \rightarrow v^*_{x_0}(x)
\end{equation}
uniformly on each compact subset of $\overline{\mathbb{R}^n_+}$ as $i \rightarrow \infty$.  Similarly,
\begin{equation}\label{whitney prereq eqn4}
	\widehat{w}^i(x) \rightarrow v^*_{x_0}(x)
\end{equation}
uniformly on each compact subset of $\overline{\mathbb{R}^n_+}$ as $i \rightarrow \infty$.  By \eqref{whitney prereq eqn2} and \eqref{rectifiability v},
\begin{equation}\label{whitney prereq eqn5}
	w^i(\xi_i + x) = \frac{u(x^i + \rho x) - h((x^i + \rho x)')}{\rho_i^{\nu}} = \widehat{w}^i(x)
\end{equation}
for each $x \in B'_{1/(4\rho_i)}(0)$.  Letting $i \rightarrow \infty$ in \eqref{whitney prereq eqn5} using $\xi_i \rightarrow \xi$, \eqref{whitney prereq eqn3}, and \eqref{whitney prereq eqn4}, we obtain
\begin{equation*}
	v^*_{x_0}(\xi + x) = v^*_{x_0}(x)
\end{equation*}
for each $x \in \mathbb{R}^{n-1} \times \{0\}$.  Hence, recalling that $v^*_{x_0}$ satisfies \eqref{hom_eqn},
\begin{gather*}
	a^{ij}(x_0) \,D_{ij} v^*_{x_0}(\xi + x) = 0 = a^{ij}(x_0) \,D_{ij} v^*_{x_0}(x) \ \text {in }\ \mathbb{R}^n_+, \\
	v^*_{x_0}(\xi + x) = v^*_{x_0}(x) \ \text{on }\ \mathbb{R}^{n-1} \times \{0\}, \\
	D_n v^*_{x_0}(\xi + x) = 0 = D_n v^*_{x_0}(x) \ \text{on }\ \mathbb{R}^{n-1} \times \{0\} ,
\end{gather*}
and thus we must have that
\begin{equation*}
	v^*_{x_0}(\xi + x) = v^*_{x_0}(x)
\end{equation*}
for each $x \in \overline{\mathbb{R}^n_+}$.
Recalling that $|\alpha| < \nu$, it follows that $D^{\alpha} v^*_{x_0}(\xi) = D^{\alpha} v^*_{x_0}(0) = 0$.  However, by dividing both sides of \eqref{whitney prereq eqn1} by $\rho_i^{\nu-|\alpha|}$ and letting $i \rightarrow \infty$, we must have that $|D^{\alpha} v^*_{x_0}(\xi)| \geq \varepsilon$, yielding a contradiction.
\end{proof}

At this point, we are ready to give the proof of Theorem \ref{rectifiability thm}.  For each $x_0 \in \Sigma_{\nu}(u)$, there exists a largest linear subspace $S(v^*_{x_0}) \subseteq \mathbb{R}^{n-1} \times \{0\}$ (with respect to set inclusion) such that $v^*_{x_0}(z+x) = v^*_{x_0}(x)$ for all $x \in \mathbb{R}^n$ and $z \in S(v^*_{x_0})$.  Note that if $\dim S(v^*_{x_0}) = n-1$, then $v^*_{x_0}$ must be a non-zero linear function of $x_n$, contradicting $D_n v^*_{x_0} = 0$ on $\mathbb{R}^{n-1} \times \{0\}$.  Therefore, $\dim S(v^*_{x_0}) \leq n-2$ for each $x_0 \in \Sigma_{\nu}(u)$.  For each positive integer $d$ we define
\begin{equation*}
	\Sigma^d_{\nu}(u) = \{ x \in \Sigma_{\nu}(u) : \dim S(v^*_{x_0}) = d \} .
\end{equation*}

\begin{proof}[Proof of Theorem \ref{rectifiability thm}]
Let $K \subseteq \Sigma_{\nu}$ be a compact set.  By Lemma~\ref{whitney prereq lemma} and the Whitney extension theorem, there exists a function $F \in C^{\kappa}(\mathbb{R}^{n-1} \times \{0\})$ such that
\begin{equation}\label{rectifiability eqn1}
	D^{\alpha} F(x_0) = F_{\alpha}(x_0)
\end{equation}
for all $x_0 \in K$ and $|\alpha| \leq \nu$, where $F_{\alpha}(x_0)$ is as in \eqref{whitney prereq taylor}.

Let $x_0 \in \Sigma^d_{\nu}(u) \cap K$.  Since $\dim S(v^*_{x_0}) = d$, there exists a set $\{\xi_1,\xi_2,\ldots,\xi_{n-1-d}\} \subset \mathbb{R}^{n-1} \times \{0\}$ of $(n-1-d)$ linearly independent vectors which are orthogonal to $S(v^*_{x_0})$.  Hence $\xi_i \cdot \nabla v^*_{x_0}$ is not identically zero on $\mathbb{R}^n$ for $i = 1,2,\ldots,n-1-d$.  Thus for each $i = 1,2,\ldots,n-1-d$ there exists a multi-index $\beta_i$ such that $|\beta_i| = \nu-1$ and
\begin{equation*}
	\xi_i \cdot \nabla D^{\beta_i} v^*_{x_0}(0) \neq 0 .
\end{equation*}
By \eqref{whitney prereq taylor} and \eqref{rectifiability eqn1}, $\xi_i \cdot \nabla D^{\beta_i} v^*_{x_0}(0) = \xi_i \cdot \nabla D^{\beta_i} F(x_0)$ and thus
\begin{equation*}
	\xi_i \cdot \nabla D^{\beta_i} F(x_0) \neq 0 .
\end{equation*}
However,
\begin{equation*}
	\Sigma^d_{\nu}(u) \cap K \subseteq \bigcap_{i=1}^{n-1-d} \{ D^{\beta_i} F = 0 \} .
\end{equation*}
Using the implicit function theorem it follows that $\Sigma^d_{\nu}(u) \cap K$ is contained in a $d$-dimensional submanifold in some open neighborhood of $x_0$.  Hence by a standard covering argument, $\Sigma_{\nu}^d$ is contained in a countable union of $d$-dimensional $C^1$-submanifolds of $B'_1(0)$.
\end{proof}

We conclude  by noting that  $\mathcal{N}_u(x_0) < \infty$ for all $x_0 \in \Sigma(u)$ in some special cases.

\begin{corollary}
Let $2 \leq p < \infty$ and $\nu$ be a positive integer.  Let $a^{ij} \in C^{0,1}(B^+_1(0) \cup B'_1(0))$ such that $a^{ij} = a^{ji}$ on $B^+_1(0)$, satisfying \eqref{ellipticity hyp} holds true for some constants $0 < \lambda \leq \Lambda < \infty$, and such that \eqref{regularity assumption} holds true on $B'_1(0)$.  Let $k_+,k_- \in C^{0,1}(B'_1(0))$ with $k_+,k_- \geq 0$ and $h : B'_1(0) \rightarrow \mathbb{R}$ be a function.  Let $u \in W^{1,2}(B^+_1(0))$ be a solution to \eqref{main ptop} with $h = 0$ on $B'_1(0)$.  If either
\begin{enumerate}
	\item[(i)]  $h = 0$ on $B'_1(0)$, or
	\item[(ii)]  $a^{ij}$ and $h$ are both locally real-analytic at each point of $B'_1(0)$,
\end{enumerate}
then $\mathcal{N}_u(x_0) < \infty$ for all $x_0 \in \Sigma(u)$ and thus $\Sigma(u)$ is contained in a countable union of $(n-2)$-dimensional $C^1$-submanifolds of $B'_1(0)$.
\end{corollary}

\begin{proof}
First we show that case (ii) reduces to case (i).  Suppose that $a^{ij}$ and $h$ are both locally real-analytic at each point of $B'_1(0)$.  After rescaling, assume that the radii of convergence of $a^{ij}$ and $h$ at the origin are greater than one.  Then by the Cauchy-Kowalevski theorem there exists a locally real-analytic function $\overline{h} : B^+_1(0) \cup B'_1(0) \rightarrow \mathbb{R}$ such that
\begin{gather}
	D_i (a^{ij} D_j \overline{h}) = 0 \text{ on } B^+_1(0), \nonumber \\
	\label{h taylor analytic} \overline{h}(x',0) = h(x'), \quad D_n \overline{h}(x',0) = 0 \text{ on } B'_1(0) .
\end{gather}
Thus we can set $v = u - \overline{h}$ so that
\begin{gather*}
	D_i (a^{ij} D_j v) = 0 \text{ on } B^+_1(0), \nonumber \\
	a^{nn} D_n v = k_+ (v^+)^{p-1} - k_- (v^-)^{p-1} \text{ on } B'_1(0) .
\end{gather*}
In other words, $v$ solves the boundary value obstacle problem with obstacle zero.

Now let's consider the case (i) where $h = 0$ on $B'_1(0)$.  Let $x_0 \in \Sigma(u)$.  After translating $x_0$ to the origin and applying a linear change of variables, assume that $x_0 = 0$ and $a^{ij}(x_0) = \delta_{ij}$.  Since $h = 0$ on $B'_1(0)$, we can omit the truncation step in Section~\ref{sec:almgren setup sec} and simply set $v = u$.  Then we simply define Almgren frequency function $\Phi(\rho)$ by
\begin{equation*}
	\Phi(\rho) = \frac{I(\rho)}{H(\rho)}
\end{equation*}
for all $0 < \rho < 1$ with $H(\rho) > 0$.  Theorem~\ref{Almgren monotonicity thm}, $e^{C\rho} \Phi(\rho) + Ce^C\rho$ is monotone nondecreasing for all $0 < \rho \leq \rho_0$ with $H(\rho) > 0$, where $\rho_0 \in (0,1/2)$ and $C \in (0,\infty)$ are constants as in Theorem~\ref{Almgren monotonicity thm}.  By Lemma~\ref{L2 growth lemma}, either $u = 0$ in $B^+_1(0) \cup B'_1(0)$ or $H(\rho) > 0$ for all $0 < \rho \leq \rho_0$.  Hence if $u$ is not identically zero on $B^+_1(0) \cup B'_1(0)$, $e^{C\rho} \Phi(\rho) + Ce^C\rho$ is monotone nondecreasing for all $0 < \rho \leq \rho_0$.  In particular,
\begin{equation*}
	\Phi(\rho) \leq e^{C/2} \Phi(1/2) + Ce^C/2 < \infty
\end{equation*}
for all $0 < \rho \leq \rho_0$ and we can define the Almgren's frequency of $u$ at the origin to be equal to $\Phi(0^+) < \infty$.  The rest of the argument proceeds exactly like before with obvious modifications.
\end{proof}

\end{document}